%% file: magn_gauss_phys_a_sublinear.tex
\documentclass[reqno]{amsart}

\input{def_phys_a_sublinear.tex}

\newcommand{\alt}[1]{}

\begin{document}

\title[Variational Gaussians for the magnetic Schr{\"o}dinger equation]
	{Variational Gaussian approximation for the magnetic Schr{\"o}dinger equation}
% \title[Gaussian wave packets for the magnetic Schr{\"o}dinger equation]
% 	{Gaussian wave packet approximation for the magnetic Schr{\"o}dinger equation}

\author[S. Burkhard]{Selina Burkhard}
\address{Institute for Applied and Numerical Mathematics, 
	Karlsruhe Institute of Technology, 76149 Karlsruhe, Germany}
\curraddr{}

\email{\{selina.burkhard, benjamin.doerich,
marlis.hochbruck\}@kit.edu	
%	selina.burkhard@kit.edu, benjamin.doerich@kit.edu,
%	 marlis.hochbruck@kit.edu
 }
\thanks{Funded by the Deutsche Forschungsgemeinschaft (DFG, German Research Foundation) – Project-ID 258734477 – SFB 1173.}

\author[B. D{\"o}rich]{Benjamin D{\"o}rich}
\author[M. Hochbruck]{Marlis Hochbruck}
\author[C. Lasser]{Caroline Lasser}

%\address{
%	Fakult\"at f\"ur Mathematik,
%	Technische Universit\"at M\"unchen,
%	85748 Garching bei M\"unchen, Germany}
\address{
	Department of Mathematics,
	Technical University of Munich,
	85748 Garching bei M\"unchen, Germany}
\curraddr{}

\email{classer@ma.tum.de}

%    \subjclass is required.
%\subjclass[2020]{Primary:
%	65M12,   % Stability and convergence of numerical methods
%	65M15. %  Partial differential equations, initial value and time-dependent initial-boundary value problems, Error bounds 
%	%
% 	65J15,   % Numerical solutions to equations with nonlinear operators 
%	Secondary:
%	65M60, %Finite   element,   Rayleigh-Ritz   and   Galerkin methods for initial value and initial-boundary value problems involving PDEs
%	%
%	35L05}%,  % wave equations
%%}


\date{} 

\dedicatory{}

%    Abstract is required.
\begin{abstract}
In the present paper we consider the semiclassical magnetic 
%\Schroed
Schr{\"o}\-din\-ger
equation, which describes the dynamics of particles under the influence of a magnetic field. The solution of the \timedependent Schr{\"o}dinger equation is approximated by a 
single Gaussian wave packet via the time-dependent Dirac--Frenkel variational principle.
For the approximation we derive ordinary differential equations of motion for the parameters of the variational solution. Moreover, we prove $L^2$-error bounds and observable error bounds for the approximating Gaussian wave packet.
\end{abstract}

\maketitle

%\todoin{Notation
%\begin{itemize}
%\item phase space center $\psc = (\pscone, \psctwo)$
%\item position and momentum parameters of Gaussian wave packet $\pos, \mom$
%\item classical position and momentum variables $\pstup = (\varone, \vartwo)$
%\end{itemize}
%i.e., 
%\begin{equation}
%\op(\varone-\pscone) = x-\pos, \quad \op(\vartwo-\psctwo) = -\ii\scp \nabla -\mom, \quad \op(\pstup-\psc)\vsol = (x-\pos)_\wm
%\end{equation}
%}

%\todoin{J Phys a Seitenlimit 18 Seiten - aufteilen auf zwei Seiten: gut für Folgearbeiten, Maliks Masterarbeit
%\begin{itemize}
%\item Bewegungsgleichungen (nur magnetisch, damit ohne Pseudodiff), $L^2$-Fehler (magnetisch)
%\item Bewegungsgleichungen (allgemein - braucht Pseudo), $L^2$-Fehler, Observablenfehler
%\end{itemize}
%oder langes Paper mit Schnittlinie (erster Teil ohne Pseudo, zweiter Teil mit) anbieten;
%oder cmp Journal}

\input{intro_sublinear.tex}
\input{setting_sublinear.tex}
\input{main_sublinear}

\input{equations_motion_sublinear}

\input{L2_eb_sublinear}
\input{expectation_values_sublinear}
\input{obs_errorbound4_sublinear}

%\input{obs_errorbound_old}

%\input{appendix_calculations}
\input{appendix_order4_sublinear}

%\newpage
\section*{Acknowledgments}
We thank Clotilde Fermanian-Kammerer and Didier~Robert for helpful discussions on the time-dependent
Egorov theorem.

\bibliographystyle{plainnat}
\bibliography{../refs}

\end{document}

%% file: def_phys_a_sublinear.tex
\usepackage{amssymb}
\usepackage{mathtools}  % for e.g. \coloneqq command
\usepackage{todonotes}
\usepackage{ifthen}
\usepackage{bm}
\usepackage{xspace}

\usepackage[msc-links]{amsrefs}
\usepackage{hyperref}
\hypersetup{
	colorlinks,
	linkcolor={red!50!black},
	citecolor={blue!50!black},
	urlcolor={blue!80!black}
}

\usepackage[nameinlink]{cleveref}
\usepackage{autonum}
\usepackage{enumitem}
\setenumerate{label=\upshape(\alph*),itemsep=3pt,topsep=3pt}

\usepackage{todonotes}

\newenvironment{enumerateletters}{\begin{enumerate}[label=(\alph*),leftmargin=0.65cm]}{\end{enumerate}}

\newcommand{\tr}{\mathrm{tr}}
\newcommand{\e}{\mathrm{e}}
\newcommand{\pt}{\partial_t}
\newcommand{\pos}{q}
\newcommand{\mom}{p}
\newcommand{\varone}{{\widetilde{q}}}
\newcommand{\vartwo}{{\widetilde{p}}}
\newcommand{\psc}{z}
\newcommand{\pscone}{q}
\newcommand{\psctwo}{p}
\newcommand{\pstup}{{\widetilde{z}}}

\renewcommand{\Pr}{P_\vsol}
\newcommand{\Prt}{P_{\vsol(t)}}
\newcommand{\Mf}{\mathcal{M}}
\newcommand{\Ts}{\mathcal{T}_\vsol \Mf}
\newcommand{\Tst}{\mathcal{T}_{\vsol(t)} \Mf}
\newcommand{\obs}{\mathbf{A}}
\newcommand{\hobs}{\mathbf{\widetilde{A}}}
\newcommand{\op}{\mathrm{op}_\mathrm{Weyl}}
\newcommand{\re}{\mathrm{Re}\,}
\newcommand{\im}{\mathrm{Im}\,}
\newcommand{\eps}{\varepsilon}
\renewcommand{\l}{\left}
\renewcommand{\r}{\right}
\newcommand{\mPot}{\widetilde{V}}
\newcommand{\Ham}{H}
\newcommand{\refPot}{\widehat{U}}
\newcommand{\mgPot}{A}
\newcommand{\mltPot}{V}
\newcommand{\scp}{\varepsilon}
\newcommand{\cham}{h}
\newcommand{\flow}{\Phi}

\newcommand{\etime}{T}

\newcommand{\ii}{\mathrm{i}}
\renewcommand{\O}{\mathcal{O}}

\newcommand{\flowts}[2]{\flow^{#1,#2}}

\newcommand{\egor}[2]{\clobs \circ \flow^{#1,#2}}
\newcommand{\auxegor}[2]{\auxobs (#1,#2)}

\newcommand{\clobs}{\bm{{a}}}

\newcommand{\auxobs}{\widetilde{\clobs}}
\newcommand{\evof}{U}
\newcommand{\SchF}{\mathcal{S}}
\newcommand{\linMom}{\ensuremath{\mathcal{P}}}
\newcommand{\angMom}{\ensuremath{\mathcal{L}}}

\newcommand{\remEgorov}{\rho}

\newcommand{\wm}{\mathcal{C}}
\newcommand{\covm}{G}
\newcommand{\pha}{\zeta}
\newcommand{\RC}{{\wm_{\textup{R}}}}
\newcommand{\IC}{{\wm_\textup{I}}}
\newcommand{\ICinv}{{\wm_\textup{I}^{-1}}}
\newcommand{\param}{\nu}

\newcommand{\RCkl}[1]{{\wm_{\textup{R},#1 }}}

\newcommand{\ICinvkl}[1]{{(\wm^{-1}_{\textup{I}})_{#1} }}

\renewcommand{\matrix}{M}
\newcommand{\fcP}{P}
\newcommand{\fcQ}{Q}
\newcommand{\E}{\mathbb{E}}
\newcommand{\R}{\mathbb{R}}
\newcommand{\N}{\mathbb{N}}

\renewcommand{\dim}{d}
\newcommand{\C}{\mathbb{C}}
\newcommand{\advPot}{Y_\vsol}
\newcommand{\PotTaylor}{\widetilde{W}_\pos}
\newcommand{\advPotTaylor}{W_\pos}
\newcommand{\remPot}{W}
\newcommand{\remOp}{W}

\newcommand{\auxPot}{X_\vsol}

\newcommand{\divergence}{\operatorname{div}}

\newcommand{\dd}{\mathrm{d}}

\newcommand{\diff}{\mathop{}\!\mathrm{d}}

\newcommand{\dx}{\diff x}

\newcommand{\abs}[1]{\left\lvert #1 \right\rvert}

\newcommand{\norm}[2]{\left\lVert #2 \right\rVert_{#1}} 
\newcommand{\normbig}[2]{\bigl\Vert #2 \bigr\Vert_{#1}} 
\newcommand{\normBig}[2]{\Bigl\Vert #2 \Bigr\Vert_{#1}}

\newcommand{\normLp}[2]{\norm{L^{#1}}{#2}}
\newcommand{\normLpbig}[2]{\normbig{L^{#1}}{#2}}
\newcommand{\normLpBig}[2]{\normBig{L^{#1}}{#2}}
 
\newcommand{\normLtwo}[1]{\normLp{2}{#1}} 
\newcommand{\normLtwobig}[1]{\normLpbig{2}{#1}} 
\newcommand{\normLtwoBig}[1]{\normLpBig{2}{#1}}

\newcommand{\Lin}[2][]{\mathcal{L}(\ifthenelse{\equal{#1}{}}{#2}{#1,#2})}%{{#2} \leftarrow \ifthenelse{\equal{#1}{}}{#2}{#1}}% usage: \Lin[X]{Y} for A : X \to Y

\newcommand{\Id}{\mathrm{Id}}
\newcommand{\sol}{\psi}
\newcommand{\vsol}{u}

\usepackage{bm}
\usepackage{xcolor}
\definecolor{turkis}{HTML}{08756c}
\definecolor{bordeaux}{HTML}{550000}
\definecolor{rosa}{HTML}{e37db5}%{ab60a6}
\makeatletter
\AtBeginDocument{%
	\def\@linkcolor{turkis}
	\def\@citecolor{bordeaux}
	\def\@urlcolor{rosa}
}
\makeatother

\newtheorem{theorem}{Theorem}[section]
\newtheorem{lemma}[theorem]{Lemma}
\newtheorem{corollary}[theorem]{Corollary}
\newtheorem{assumption}[theorem]{Assumption}
\newtheorem{prop}[theorem]{Proposition}

\theoremstyle{definition}
\newtheorem{definition}[theorem]{Definition}

\theoremstyle{remark}
\newtheorem{remark}[theorem]{Remark}

\numberwithin{equation}{section}

\crefname{equation}{}{}
\crefname{enumi}{}{}
\crefdefaultlabelformat{#2\textup{#1}#3}
\crefname{assumption}{assumption}{assumptions}
\crefname{lemma}{lemma}{Lemmas}
\numberwithin{equation}{section}

\newcommand{\wellposedness}{wellposedness\xspace}

\newcommand{\wellpoAdj}{well posed\xspace}

\newcommand{\timedependent}{time-dependent\xspace}
\newcommand{\timedependence}{time-dependence\xspace}
\newcommand{\timeindependent}{time-independent\xspace}
\newcommand{\Weylquantization}{Weyl quantization\xspace}
\newcommand{\Weylquantized}{Weyl-quantized\xspace}
\newcommand{\ODE}{ODE\xspace}

\newcommand{\semiclassical}{semiclassical\xspace}
\newcommand{\wrt}{with respect to\xspace}
\newcommand{\GWP}{Gaussian wave packet\xspace}
\newcommand{\GWPs}{Gaussian wave packets\xspace}
\newcommand{\Schroed}{Schr{\"o}\-din\-ger\xspace}
\newcommand{\blowup}{blow-up\xspace}

%% file: intro_sublinear.tex
\section{Introduction} %\label{intro}
%\todoin{zitieren \cite{BoiV21}}

In the present paper we study 
the \semiclassical magnetic Schrödinger equation 
\begin{subequations}
\begin{equation} \label{eq:sproblem}
\ii\scp\pt \sol(t) =\Ham(t)\sol(t),\quad \sol(0) = \sol_0, \quad t\in \R,
\end{equation}
on $\R^\dim$ with magnetic Hamiltonian
\begin{equation}\label{eq:ham}
\Ham(t)=  \frac12 \bigl(\ii\scp \nabla_x + \mgPot(t,x)\bigr)^2 + \mltPot(t,x),
\end{equation}
 \end{subequations}
%% aus short paper
and initial value $\sol_0 \in L^2(\R^\dim)$ with \semiclassical parameter $0< \scp \ll 1$. Here, $\mgPot$ is a magnetic vector potential, and $\mltPot$ is the electric potential. This equation arises in the modeling of the quantum dynamics of nuclei in a molecule subject to external magnetic fields.
From a numerical point of view, solving this time-dependent partial differential equation raises three major problems. First, it is a high-dimensional problem, since the space dimension is typically given by $\dim = 3N$, where 
$N$ is the number of nuclear particles in the system. Further, the computational domain~$\R^\dim$ is naturally  unbounded, and thus most numerical methods require truncation before discretization. For the method of lines (first discretize space, then time), high dimension combined with an unbounded domain leads to 
inadequately if not unattractably large systems that have to be integrated in time. 
Another challenge is given by the high oscillations induced by the small \semiclassical parameter~$\scp$. For standard time integration schemes severe stepsize restrictions have to be imposed and leave these methods impracticable. 

We consider the case that the initial value $\sol_0$ is strongly localized and given by a Gaussian wave packet, 
\[
\psi_0(x) = \exp{\Bigl(\frac{\ii}{\scp}\bigl(\frac12(x-\pos)^T \wm (x-\pos)+(x-\pos)^T \mom +\pha\bigr)\Bigr)},
\]
where $\pos,\mom\in\R^d$ are the packet's position and momentum center, $\wm\in\C^{d\times d}$ is the width matrix of the envelope, and $\pha\in\C$ a phase and weight parameter.
For $\mgPot = 0$ it is well established that it is possible to reasonably approximate the solution by a Gaussian wave packet with parameters that are evolved according to ordinary differential equations. 
First studies in this direction are due to K.~Hepp \cite{Hep74} and G.~Hagedorn \cite{Hag80} from the perspective of mathematical physics, and E.~Heller \cite{Hel76,Hel81} as well as R.~Coalson, M.~Karplus \cite{CoaK90} with already an eye on numerical computation. The evolution equations for the parameters of all Gaussian wave packet approximations can be classified in two categories:

\smallskip
\textit{Variational:} The variational approach
%invokes
%\textcolor{red}{
relies on
%}
the time-dependent Dirac--Frenkel principle 
for deriving the parameter equations of motion. By the variational construction, the Gaussian wave packet automatically inherits several conservation properties of the exact solution.

\smallskip
\textit{Semiclassical:} The \semiclassical approach expands the wave packet ansatz with respect to the 
semiclassical parameter $\scp$ and derives $\eps$ independent parameter equations by matching terms 
with the same order.  

\smallskip
Both types of ordinary differential equations have the advantageous property, that their solutions are non-oscillatory. 
Both approximations have the same convergence order with respect to the \semiclassical parameter $\scp$ in $L^2$-norm, and both reproduce the  exact solution for the special case of Schr\"odinger operators with linear magnetic potential $\mgPot$ and quadratic electric potential $\mltPot$. 
For a further discussion, we refer to \cite[Chapter~10.2]{Car21} for a monograph that covers the \semiclassical construction, to \cite[Chapter II.4]{Lub08_book} or \cite[Chapter 3]{LasL20} for a short book and a review presenting the variational case, and to \cite{Van23} for a general presentation of Gaussian wave packet dynamics.

\subsection*{Contributions of the paper}
Our main contribution in this paper is to first show that for the magnetic Schr\"odinger equation the variational approximation is still given by a system of ordinary differential equations for the parameters defining the Gaussian wave packet.
% \in \Mf$.
Second, we prove rigorous error bounds for this approximation on finite time intervals $[0,T]$ in terms of the \semiclassical parameter $\scp$.
The presented results generalize the bounds established in \cite{Lub08_book,LasL20} to non-vanishing magnetic potentials~$\mgPot$
and further allow for time-dependencies in both the electric and the magnetic potential. 
We also treat the more general case where the dynamics are generated by the Weyl quantization 
of a smooth and subquadratic Hamiltonian function. 
This includes convergence in the $L^2$-norm with order $\mathcal{O}(\sqrt{\scp})$ as well as for expectation values of observables,
 which resemble certain measurable physical quantities of the wave function,
with order $\mathcal{O}(\scp^2)$. These estimates extend and improve the observable bound of \cite[Theorem~3.5]{LasL20} and the result of \cite{Ohs21} from the case of vanishing magnetic potential.  
%\textcolor{blue}{
%In order to show these bounds for the magnetic, time-dependent case several results from the semiclassical %analysis have to be generalized. 
%
%Most importantly,
%we need for our analysis a time-dependent version of Egorov's Theorem  
%which relates quantum observables to classical ones. 
%
%Even though, there are several theorems of this kind available in the literature, see \cite{LasL20} (more!),
%
%we are not aware of a version that covers our case here. 
%
%In particular, the fact that the Hamiltonain $\Ham(t)$ does not commute with the evolution family introduces severe difficulties.
%
%\todoin{Mehr Schwierigkeiten}
%\todoin{Hervorheben, dass man gegenueber Hagedorn symplektisch bleibt?}
%}
%\begin{itemize}
%	\item Difficulties: \timedependence
%	\begin{itemize}
%		\item $\Ham$ and evolution family $\evof$ don't commute; modify proof of \cite{LasL20}
%		\item proof of Egorov--theorem in \cite{LasL20} can't be used, because classical hamiltonian function is not constant (w.r.t.~time) along the classical flow
%		%\item For polynomially bounded symbols: (needed for Egorov; also in time-independent case; only for non-quadratic Hamiltonian or polynomially bounded observables; don't enter leading order of $\scp$, only with $\scp^2$, see proof)
%	\end{itemize}
%\end{itemize}
%\slin{The magnetic case, where $\mgPot$ is sublinear but not linear is not a special case of the subquadratic case. %However, both cases can be investigated similarly with some additional considerations in the magnetic case, which %is why we treat both cases here.}
Let us point out that the design and the analysis of time integrators for the magnetic variational equations of motion are currently under investigation.

\subsection*{Further wave packet results for $\mgPot=0$}
Hagedorn wave packets \cite{Hag80,Hag81,Hag98} are a multivariate anisotropic generalization of the Hermite functions. They are Gaussian wave packets with a polynomial prefactor, such that a family of them 
constitutes an orthonormal basis of $L^2(\R^d)$. In \cite{FaoGL09,GraH14,BlaG20}, time splitting integrators 
for Hagedorn wave packet approximations are proposed, that combine parameter propagation by ordinary differential equations with a Galerkin step. A spawning method for several families of Hagedorn wave packets is introduced in \cite{RieG22*}.  
For variational Gaussian wave packets, a time splitting integrator, which is robust in the semiclassical parameter $\scp$, is proposed in \cite{FaoL06}.
Recently in \cite{Ohs21}, T. Oshawa has analysed the expectation values of position and momentum for a variational Gaussian wave packet and proved $\mathcal{O}(\scp^{3/2})$ accuracy. 
Our results here generalize and improve this error bound in two ways: First, we allow for general sublinear observables. Second, our method of proof shows $\mathcal O(\scp^2)$ observable accuracy also for the case $\mgPot\neq0$. It is worthwhile emphasizing, that from the perspective of the observable error variational Gaussians are more accurate than their semiclassical counterparts.

\begin{table}
\begin{tabular}{l|l|l}\hline
Gaussian & $L^2$-norm & observables\\*[0.5ex]\hline 
\semiclassical & $\mathcal O(\sqrt\scp)$ & $\mathcal O(\scp)$\\*[0.5ex]
variational & $\mathcal O(\sqrt\scp)$ & $\mathcal O(\scp^2)$\\\hline
\end{tabular}\\*[1ex]
\caption{Error bounds for the \semiclassical and the variational approximation of magnetic Schr\"odinger dynamics according to \Cref{thm:L2} and \Cref{thm:obs}. The variational observable error estimate extends and improves previously known results.}
\end{table}

%Multiple Gaussians/Hagodorn wave packets \cite{RieG22*,BerL22}
%\textcolor{gray}{
%Another error bound of order $\scp^{3/2}$ was given in  for the \timeindependent, non-magnetic case  for a special case of observables. Moreover, they consider Hagedorn wave packets for the approximation of the exact solution. This paper constitutes an interesting part of future research to see if it is possible generalize their error bound to the \timedependent magnetic case for general observables.
%}

\subsection*{Related wave packet results for $\mgPot\neq 0$}
The most general result for the \semiclassical wave packet approach is given in 
\cite[Theorem~21]{RobC21_book} of the monograph by D.~Robert and  M.~Combescure. 
There, the propagation of Gaussian and Hagedorn wave packets is covered for a general class of 
time-dependent Hamiltonian operators~$H(t)$, that includes the magnetic Schr\"odinger operator. 
The error analysis is with respect to the $L^2$-norm, but not for observables. The \semiclassical 
construction there also receives corrections, such that it can be accurate to order $\mathcal O(\scp^{k/2})$ 
for any $k\ge1$. 
In \cite{BoiV21}, magnetic Schr\"odinger operators with polynomially bounded, time-independent magnetic 
fields and zero potential are considered. The initial coherent state has zero initial energy and its propagation 
is analysed for the long-time horizon $[0,T/\scp]$.   
In \cite{KinO20}, N.~King and  T.~Ohsawa derive the equations of motion for variational Gaussians in the presence of a magnetic field. They conduct numerical experiments for the expectation value of the position and the momentum operator suggesting that the variational Gaussians are more accurate than the \semiclassical ones. 
An extension of the Hagedorn Galerkin method \cite{FaoGL09} to the case of magnetic Schr\"odinger equations is studied in \cite{Zho14}, including an error analysis with respect to the $L^2$-norm. However, no error bounds for the observables are investigated there.
For linear magnetic potentials of a particular structure, in \cite{GraR21} a problem adapted splitting method for Hagedorn wave packets is derived but without error analysis. 
A slightly different approach, called the Gaussian wave packet transform, is proposed for the magnetic Schr\"odinger equation in \cite{ZhoR19}. There, the ordinary differential equations for the Gaussian parameters are the \semiclassical ones except for an additional term for the scalar parameter $\zeta$.

\subsection*{Outline of the paper}
The rest of the paper is structured as follows. 
For our error analysis we introduce the analytical framework
and the variational Gaussian wave packet ansatz in \Cref{sec:setmain}.
We present our main results for the magnetic Schr\"odinger equation in \Cref{sec:mainres},  including the equations for the parameters, the conservation of different quantities, the convergence in the $L^2$-norm and the convergence of the observables.
The proofs of the corresponding results are given in 
\Cref{sec:eqmo,sec:L2,sec:expv,sec:obs}.
%
%\Cref{sec:eqmo}
%\Cref{sec:L2}
%\Cref{sec:expv}
%\Cref{sec:obs}
%A technical computation regarding the symbols for the observables is postponed to  
%\Cref{appendix}.

\subsection*{Notation}

Throughout the paper, we denote by $L^p(\R^\dim)$ the classical Lebesgue spaces, and by
$\SchF(\R^\dim)$ the Schwartz space of rapidly decreasing functions.
Further, we make use of the multiindex notation and let for $\alpha =(\alpha_1,\ldots,\alpha_\dim) \in \N_0^\dim$, $x\in\R^\dim$, $f \in \SchF(\R^\dim)$
\begin{equation}
%|\alpha| \coloneqq \sum\limits_{j=1}^\dim \alpha_j,
|\alpha| \coloneqq \alpha_1 + \ldots + \alpha_\dim,
\qquad
x^\alpha \coloneqq x_1^{\alpha_1} \ldots x_\dim^{\alpha_\dim},
\qquad
\partial^\alpha f \coloneqq \partial_1^{\alpha_1} \ldots \partial_\dim^{\alpha_\dim} f.
\end{equation}
For a function $W \colon \R^\dim \to \R^L$, $L\geq 1$, we define the average
\begin{equation}
\langle W\rangle_\vsol \coloneqq \langle \vsol |W\vsol\rangle = \int_{\R^d}  W(x) |\vsol(x)|^2 \,\dx,
\end{equation}
if the integral exists. For a linear operator $\obs$ acting on $L^2(\R^d)$, we denote
\[
\langle \obs\rangle_\vsol \coloneqq \bigl\langle \vsol| \obs \vsol \bigr\rangle = 
\int_{\R^d}  \overline{\vsol(x)}(\obs\vsol)(x) \,\dx,
\] 
whenever the integral is well-defined. We also use the dot product of 
$v,w\in\C^L$ as $v\cdot w \coloneqq v^Tw = v_1w_1 + \cdots + v_L w_L$.

%% file: setting_sublinear.tex
\section{General setting}
\label{sec:setmain}
%\begin{itemize}
%	\item variational formulation using Projection onto tangent space
%	\item Tangent space here = quadratic potentials multiplied by $\vsol$
%	\item Exactness result for quadratic $\mltPot$, linear $\mgPot$ in space
%	\begin{itemize}
%		\item $\Delta \vsol$, $\i\scp \mgPot \cdot \nabla \vsol$
%	\end{itemize}
%\end{itemizez
%
%

We first discuss the analytic framework for our analysis and introduce the Gaussian wave packets. We further call some results on the \wellposedness from the literature.
For the vector potential we choose the  
Coulomb gauge, i.e. $\divergence \mgPot = 0$.
In order to shorten notation, we rewrite the Ha\-mil\-to\-ni\-an in \eqref{eq:ham} as
\begin{equation}\label{eq:hamop}
  \Ham(t)=  -\frac{\scp^2}{2}\Delta +\ii\scp\mgPot(t)\cdot \nabla + \mPot(t),
  \qquad
  {\mPot \coloneqq \frac{1}{2}|\mgPot|^2 + \mltPot} .
\end{equation}
%
%Throughout this paper we assume that
%\begin{itemize}
%	\item $\mltPot$ subquadratisch in $x$, $\mgPot$ sublinear in $x$, gleichmäßig auf  $[0,\etime]$, diff bzgl $t$
%\end{itemize}
%

Throughout this paper we make the following smoothness and growth assumption on the potentials.
%
%\begin{assumption} %\label[assumption]{ass:onpotentials}
%	 The scalar, subquadratic potential $\mltPot$ and the vector valued, sublinear, magnetic potential $\mgPot = (\mgPot_j)_{j=1, \ldots, \dim}$ are assumed to be smooth and might be time-dependent.
%\end{assumption}
%

%\unklar{polynomial ?}
\begin{assumption} \label{ass:onpotentials}
	The scalar potential $\mPot \colon \R \times \R^\dim \to \R$
	and the vector valued potential $\mgPot = (\mgPot_j)_{j=1, \ldots, \dim} \colon \R \times \R^\dim \to \R^\dim$ are
	infinitely often differentiable and in addition
	\begin{enumerate}
		\item $\mPot$ is subquadratic, i.e. $\nabla^k \mPot$ is bounded for all $k\geq 2$, and
		\item $\mgPot$ is sublinear, i.e. $\nabla^k \mgPot$ is bounded for all $k\geq 1$, and satisfies $\divergence \mgPot = 0$.
	\end{enumerate}
\end{assumption}

If in addition to \Cref{ass:onpotentials}, we assume that $\partial_t \mgPot$ is sublinear, then it can be shown that the initial value problem \eqref{eq:sproblem} 
%can be shown to 
is \wellpoAdj
for initial values in $L^2$, cf. \cite[sec.~4]{Yaj21} or the remarks after~\cite[Def.~1]{RobC21_book} or \cite[Rem.~5.14]{MasR17}.
% \begin{definition}\label[definition]{def:poly_obs}\
% Let \textcolor{blue}{$k \in \N_0$} and $\scp \in (0, \infty)$
% and let  $\varphi$ be a sufficiently regular function such that
%\begin{equation}\label{eq:weightedSobnorm}
%\textcolor{blue}{\normweiSobk{\varphi}{k} \coloneqq \sqrt{ 
%\normLtwobig{\varphi}^2 + \normLtwobig{(-\scp^2\Delta_x + |x|^2)^{k/2}\varphi}^2} < \infty.}
%\end{equation}
%Further, we define the space
%%
%\begin{equation}   \label{eq:weightedspace}
%\textcolor{blue}{\spaceWeiSobk{k} \coloneqq \bigl\{ 
%\varphi \in L^2(\R^\dim)  \bigm|\normweiSobk{\varphi}{k}
%  < \infty\nonumber \bigr\} }
%\end{equation}
%\end{definition}
%
% Note that the space \textcolor{blue}{$\spaceWeiSobk{k}$} is in fact independent of
% $\scp$ since the $\normweiSobk{\cdot}{k}$ norms for different values of $\scp$ are equivalent. 
% \textcolor{blue}{It can also be characterized according to
% \begin{align}
% \spaceWeiSobk{k} = \left\{ \varphi\in L^2(\R^\dim)\bigm| 
% x^\alpha (\scp\partial_x)^\beta \varphi\in L^2(\R^\dim) ,\quad |\alpha|+|\beta|\le k\right\}.
% \end{align}}
In particular, 
 the following \wellposedness result on the unitarity of the time evolution guarantees that the norm of the solution of \eqref{eq:sproblem} is the same as the one of the initial data. However, for our analysis here, only \Cref{ass:onpotentials} will be used.

%  The solution can be described by unitary evolution families $(\evof(t,s))_{t,s\in [0,\etime]}$ on $L^2(\R^\dim)$ and is given by 
%\begin{equation}\label{eq:sol}
%\sol(t) = \evof(t,0)\sol_0.
%\end{equation}
%\todoin{In which cases is \eqref{eq:sproblem} well-posed? Formulate criteria in some Theorem?}
%

\begin{theorem}\label{thm:wellposedness}
	Let \Cref{ass:onpotentials} hold  and assume that $\partial_t \mgPot$ is sublinear.
%	Furthermore, we  % \eqref{eq:weightedSobnorm}.
%	\begin{enumerate} %for initial data $\sol_0$ from \Cref{def:poly_obs} 
%\item 
 There exists a unitary evolution family $(\evof(t,s))_{t,s\in \R}$ on $L^2(\R^\dim)$ such that
		for all initial data $\sol_0\in L^2(\R^\dim)$ the solution $\sol$ of \eqref{eq:sproblem} is given by
		\begin{equation}\label{eq:sol}
		\sol(t) = \evof(t,0)\sol_0.
		\end{equation}
		%with initial value $\sol_0$ satisfying \Cref{defitem:weighted_sov} in \Cref{def:weighted_sobolev}.
		%
%		\item \textcolor{purple}{bound}
%		Let \textcolor{blue}{$\sol_0\in \spaceWeiSobk{k}$ for some $k \in \N_0$}
%		 and let $\sol$ be the solution of \eqref{eq:sproblem}. Then we have
%		%
%	\textcolor{blue}{
%		\begin{equation}\label{scw_sob}
%		\norm{k}{\sol(t)}
%		 C_{k}  \e^{C_{k} t} 
%		\norm{k}{\sol_0},
%		\end{equation}
%                where $C_{k}$ is independent of $\scp$.}
%	\end{enumerate}
\end{theorem}
%\begin{prop}[{\cite[Prop.~123]{RobC21_book}}]\label[prop]{prop:scw_sob}
%
%\end{prop}
%
%
%
%Under stronger conditions on the potentials $A$ and $V$ we can alternatively employ a perturbation theory on relatively bounded perturbations for selfadjoint operators and make use of evolution families in the hyperbolic case, see \cite{EngN00_book,Sch12_book,Kat53,Paz83_book}.
In the case of time-independent potentials the evolution family $(U(t,s))_{t,s \in\R}$ reduces to the unitary group $(e^{-it/\scp H})_{t \in \R}$
on $L^2(\R^\dim)$, which which is given by the spectral theorem and commutes with the Hamiltonian.
%Analysis:
%\begin{itemize}
%	\item Evolutionsfamilie (im zeitunabhängigen Fall Gruppe)
%	\item \cite{Kat53,Kat70,EngN00_book,Paz83_book,Sch12_book,RobC21_book}
%\end{itemize}

Following \cite[Chapter~3]{LasL20}, we approximate the solution $\sol$ of \eqref{eq:sproblem} in the manifold~$\Mf$ of \GWPs given by

\begin{align} 
\Mf = \Bigl\{ 
&g \in L^2(\R^\dim)  \bigm|
g(x) = \exp{\Bigl(\frac{\ii}{\scp}\bigl(\frac12(x-\pos)^T \wm (x-\pos)+(x-\pos)^T \mom +\pha\bigr)\Bigr) } , \nonumber
\\
&\pos,\,\mom\in \R^\dim, \,
 \wm = \wm^T \in \C^{\dim\times \dim}, \,
\im \wm 
\text{ positive definite}, \pha\in \C
\Bigr\} . \label{eq:gwp}
\end{align}

The approximating \GWP is characterized by
the Dirac--Frenkel variational formulation, cf.~\cite{LasL20,Lub08_book}:
seek $\vsol (t) \in \Mf$ such that for all $t \in\R$ it holds 
%
%\begin{equation}\label{eq:varskp}
%  \begin{aligned}
%    &\pt \vsol(t)  \in \Tst, \\
% &\bigl\langle \ii\scp \pt \vsol(t) - \Ham(t)\vsol(t)| v \bigr\rangle = 0\quad \text{for all}\quad v\in \Tst,
%\end{aligned}
%\end{equation}
\begin{equation}\label{eq:varskp}
\pt \vsol(t)  \in \Tst, \quad\bigl\langle \ii\scp \pt \vsol(t) - \Ham(t)\vsol(t)| v \bigr\rangle = 0
\quad
\text{for all}
\quad 
 v\in \Tst,
\end{equation}
with initial value $\vsol(0) = \vsol_0 \in \Mf$.  Using the orthogonal projection $\Pr : L^2(\R^\dim)\to\Ts$
% from $L^2$ onto the tangent space,
we can equivalently write %\eqref{eq:varskp} as
\begin{align} \label{eq:var}
\ii\varepsilon\pt \vsol(t) = \Prt\bigl(\Ham(t)\vsol(t) \bigr), \quad \vsol(0) = \vsol_0 \in \Mf.
\end{align}
We note that \eqref{eq:var} can also be stated in terms of the symplectic projection onto the tangent space, see C. Lubich's blue book \cite[II.1.3]{Lub08_book}.

\begin{remark} \label[remark]{rem:non_gaus_init}
	In the time-independent and non-magnetic case, one can also
	treat initial values $\sol_0 \notin \Mf$
	using continuous superpositions of thawed and frozen Gaussians, see \cite[Ch.~5]{LasL20}.
	The extension of these to the case \eqref{eq:ham}, however, is beyond  the scope of the present work.
\end{remark}
%with initial data $\vsol(0) = \vsol_0 \in \Mf$.

For the manifold $\Mf$ defined in \eqref{eq:gwp} the tangent space $\Ts$ takes the following simple form.
\begin{lemma}[{\cite[Lemma~3.1]{LasL20}}]\label[lemma]{lem:tangent}
 For $\vsol \in \Mf$ we have
\begin{equation}
\Ts = \left\{\varphi u\,|\, \varphi~\dim\text{-variate complex polynomial of degree at most }2\right\}.
\end{equation}
\end{lemma}
%%%% aus waves-abstract
The approximation by \GWPs seems appropriate due to the following exactness result, which is a consequence of \Cref{lem:tangent} together with \eqref{eq:var} and \Cref{thm:wellposedness}.
%the unique solvability of the initial value problem \eqref{eq:sproblem}.
%
\begin{prop}[{\cite[Prop.~3.2]{LasL20}}]\label[prop]{thm:exact}
  Let $\mltPot(t,\cdot)$ be quadratic and $\mgPot(t,\cdot)$ be linear in space for all ${t \in \R}$. If $\sol_0\in\Mf$, then the
  variational approximation $\vsol$ defined by  \eqref{eq:var} is exact, i.e.,
  $\vsol(t)=\sol(t)$, where $\sol$ denotes 
  the solution of \eqref{eq:sproblem}.
\end{prop}

In the next section we derive a system of ordinary differential equations to determine parameters of the variational solution $\vsol\in \Mf$ and present error bounds
for the variational approximation.

%% file: main_sublinear.tex
\section{Main results}\label{sec:mainres}

In the remaining paper we consider \eqref{eq:sproblem} and
\eqref{eq:var} for  initial data satisfying
\begin{equation}\label{eq:gauss_initialdata}
\sol_0 = \vsol_0\in \Mf\quad \text{ and } \quad \normLtwo{\vsol_0} =1.
\end{equation}
Our first step is to derive equations of motions for the
parameters defining the variational solution $\vsol$. Then we show that in the limit $\scp \to 0$, these equations tend to classical equations of motions.
Moreover, we study geometric properties of the solution and the variational approximation. Finally, we 
state error bounds for the solution in the $L^2$-norm and for averages of observables.
Our work generalizes the results in \cite{LasL20}
in the sense that we
treat time-dependent, magnetic Hamiltonians. We also generalize the results of \cite{KinO20,Ohs21} from the position and momentum operator to sublinear observables in the sense of \Cref{ass:onpotentials}.
For the sake of readability, we postpone the proofs to \Cref{sec:expv,sec:eqmo,sec:obs,sec:L2}.

\subsection{Variational equations of motion} \label{subsec:motion}

% By the Dirac--Frenkel variational formulation, we obtain the differential equation \eqref{eq:var} which leads to solution in the manifold $\Mf$. Since this solution is expressed in the parameters $\pos, \mom, \wm, \pha$, we derive
% in the following their corresponding ordinary differential equations. Often they are referred to as the \textit{equations of motion}.
% %

In order to write equations of motion for the parameters of a \GWP $u \in \Mf$ we use the short notation 
\begin{alignat}{5}
  \RC & = \re \wm, & \IC & = \im \wm, \\
v &= (v_j)_{j=1}^\dim,
\qquad &\mgPot &= (\mgPot_j)_{j=1}^\dim,
\\
%		&		&\text{ Jacobi\, matrix},\\
J_\mgPot &= \bigl(\partial_j \mgPot_k\bigr)_{j,k = 1}^\dim, \qquad&
(D^2_{\mgPot,v})_{k,l} &= {\textstyle\sum_{j=1}^\dim}{\partial_l\partial_k \mgPot_j v_j.}
\end{alignat}
We start by deriving two equivalent sets of equations for
$0 < \scp \ll 1$. In the following section, we discuss the limit $\scp \to 0$
and show that the two sets lead to the classical
equations of motion for charged particles in a magnetic field
given by  the \timedependent Hamiltonian function%corresponding to \eqref{eq:hamop}
\begin{align}
\cham(t,\varone,\vartwo) 
%&= \frac{1}{2}|\vartwo|^2-\mgPot(t,\varone)\cdot \vartwo +\frac{1}{2}|\mgPot(t,\varone)|^2 + \mltPot(t,\varone)\\
&= \frac{1}{2}|\vartwo|^2-\mgPot(t,\varone)\cdot \vartwo + \mPot(t,\varone),\quad 
(t,\varone,\vartwo)\in\R\times\R^{2\dim},\label{eq:clham}
\end{align}
cf.~\cite{GusS20_book, Hal13_book}.
The first set of equations of motion reads:

\begin{theorem}\label{thm:eqmo}
Let $\vsol_0$ satisfy \eqref{eq:gauss_initialdata} and  be given by its parameters $\pos_0, \mom_0, \wm_0, \pha_0$ defined in \eqref{eq:gwp}.
Then, the parameters of the solution $u\in \Mf$ of \eqref{eq:var} satisfy
\begin{subequations}\label{eq:eqmo} %\ICinv oder Im(C^{-1})
\begin{alignat}{5}
\dot{\pos} &= \mom - \langle \mgPot\rangle_\vsol,% &\pos(0)&=\pos_0
\label{eq:eqmo_q}\\
\dot{\mom} &=  \frac{\scp}{2}\l\langle\nabla \tr\l(J_A^T \RC \ICinv \r)\r\rangle_\vsol + \langle J_\mgPot \rangle_\vsol^T \mom - \langle \nabla \mPot \rangle_\vsol,
%\quad &\mom(0)&=\mom_0
\label{eq:eqmo_p}\\
\dot{\wm}  &= - \wm^2  + \langle  D^2_{\mgPot,\mom}\rangle_\vsol + \langle J_\mgPot\rangle_\vsol^T \wm + \wm \langle J_\mgPot\rangle_\vsol - \langle \nabla^2 \mPot \rangle_\vsol  . \label{eq:eqmo_C}\\ %+ \tr\l(J_\mgPot \RC \ICinv\r)
&\,\quad+ \frac{\scp}{2} \bigl\langle \nabla^2 \tr\l(J_\mgPot^T \RC \ICinv\r)\bigr\rangle_\vsol, \nonumber\\
\dot{\pha} &= \frac{1}{2} |\mom|^2 + \frac{\scp}{2} \bigl\langle  \tr\l(J_\mgPot^T \RC \ICinv\r)\bigr\rangle_\vsol + \frac{\ii \scp}{2} \tr (\wm) \label{eq:eqmo_pha} \\
&\quad 
- \frac{\scp}{4}\tr\Bigl(\ICinv\bigl(\frac{\scp}{2} \bigl\langle \nabla^2 \tr\l(J_\mgPot^T \RC \ICinv\r)\bigr\rangle_\vsol 
+
 \l\langle J_\mgPot\r\rangle_\vsol^T \RC +\RC \langle J_\mgPot\rangle_\vsol 
 +
  \langle  D^2_{\mgPot,\mom}\rangle_\vsol\bigr)\Bigr) \nonumber
\\
&\quad - \langle \mPot\rangle_\vsol + \frac{\scp}{4}\tr\bigl(\ICinv \bigl\langle \nabla^2 \mPot\bigr\rangle_u\bigr), \nonumber
\end{alignat}
with initial data $(\pos(0),\mom(0), \wm(0), \pha(0)) = (\pos_0, \mom_0, \wm_0, \pha_0)$.
%- \frac{1}{2}\l\langle |\mgPot|^2\r\rangle_\vsol + \frac{\scp}{8}\tr\l(\im \wm^{-1} \l\langle \nabla^2 |\mgPot|^2\r\rangle_\vsol\r)
\end{subequations}
\end{theorem}
The proof of \Cref{thm:eqmo} is given in \Cref{sec:eqmo}. We observe that in terms 
of the classical Hamiltonian function $h$ defined in \eqref{eq:clham}, the equations of motion \eqref{eq:eqmo} can be rewritten as 
\begin{subequations}\label{eq:eqmo-general}
\begin{align+}
\dot{\pos} &=\langle \nabla_\mom \cham\rangle_\vsol,\label{eq:eqmo-general-q}\\
\dot{\mom} &= -\langle \nabla_\pos \cham\rangle_\vsol, \label{eq:eqmo-general-p}\\
\dot{\wm} &=-\langle \nabla_{\pos\pos} \cham\rangle_\vsol - \langle\nabla_{\pos\mom} \cham\rangle_\vsol \wm - \wm \langle\nabla_{\mom\pos} \cham \rangle_\vsol -\wm \langle\nabla_{\mom\mom} \cham\rangle_\vsol \wm, \\
\dot{\pha} &= -\langle \cham\rangle_\vsol +\frac{\scp}{4}\, \tr(B \ \IC^{-1}) + \mom^T \langle\nabla_{\mom} \cham\rangle_\vsol. \label{eq:eqmo-general-pha}
\end{align+}
\end{subequations}
with the matrix $B\in\C^{d\times d}$ given by
\[
B = \begin{pmatrix}\Id,\wm \end{pmatrix}\langle \nabla^2 \cham\rangle_\vsol 
	\begin{pmatrix}\Id\\ \wm\end{pmatrix}. 
\]
Later on, in \Cref{theo:eqmo-general-ham} we extend these findings to the variational dynamics 
induced by a general a subquadratic Hamiltonian. 

\begin{remark}
In order to solve \eqref{eq:eqmo} numerically, one might adapt the
Boris algorithm originally proposed in \cite{Boris70} and recently analyzed 
in \cite{HaiL20,HaiLW20}. This algorithm is constructed for
the classical equations of motion for charged particle systems.
Details or an efficient numerical algorithm are ongoing work which will be presented elsewhere.
\end{remark}
%
%\todoin{\begin{itemize}
%		\item : Gleichungen passen zusammen
%\end{itemize}}

An alternative approach presented in \cite{LasL20} makes use of
a factorization of the width matrix $\wm$ due to Hagedorn.
For the magnetic \Schroed equation, it leads to differential equations for the factors of $\wm$ instead of \eqref{eq:eqmo_C}.
By \cite[Lemma~3.16]{LasL20}, we can write %factorize the width matrix by
\begin{equation} \label{eq:Hag_fac}
\wm = \fcP \fcQ^{-1} \quad \mathrm{and}\quad \im \wm = (\fcQ\fcQ^*)^{-1},
\end{equation}
with complex, invertible, and symplectic matrices $\fcP$ and $\fcQ$. The latter means that for
\begin{equation}\label{eq:PQSym}
Y \coloneqq
\begin{pmatrix}
\re Q\quad&  \im Q\\
\re P \quad&\im P
\end{pmatrix}
\qquad
\text{and}
\qquad
J \coloneqq
\begin{pmatrix}
0& -\Id\\
\Id&0\\
\end{pmatrix}
\in \R^{2d \times 2d} %\quad \mathrm{symplektisch}
\end{equation}  
%and $J$ defined in \eqref{eq:Jmatrix},
it holds $Y^T JY =J$, or equivalently 
\begin{subequations} \label{eq:PQ}
\begin{align+}
\fcQ^T\fcP - \fcP^T\fcQ &= 0, \label{eq:PQ_transpose}\\
\fcQ^*\fcP - \fcP^*\fcQ &= 2\ii\, \mathrm{Id}. \label{eq:PQ_conjugate}
\end{align+}
\end{subequations}
In fact, if $\fcQ$ and $\fcP$ are complex matrices satisfying \eqref{eq:PQ}, then $\fcQ$ and $\fcP$ are invertible and the matrix $\wm = \fcP \fcQ^{-1}$ is symmetric with positive definite imaginary part $(\fcQ \fcQ^*)^{-1}$.
% Contrariwise, each symmetric matrix $\wm \in \C^{d\times d}$ having positive definite imaginary part can be factorized into $\wm = \fcP\fcQ^{-1}$, where $\fcQ$ and $\fcP$ are symplectic.
%
This allows us to write the \GWP \eqref{eq:gwp} as
\begin{equation}\label{eq:GWP_factorization}
\vsol(\cdot,x) = \exp \Bigl(\frac{\ii}{\scp}\Bigl(\frac{1}{2}(x-\pos)^T \fcP\fcQ^{-1}(x-\pos)+ \mom^T (x-\pos)+ \pha\Bigr)\Bigr)
\end{equation}
and to derive equations of motion for the parameters $(\pos, \mom, Q,P,\pha)$.%can be derived.

\begin{corollary}\label{thm:hagmotion} 
Let $\vsol_0$ satisfy \eqref{eq:gauss_initialdata} and be
	  given by the parameters $\pos_0, \mom_0$, $\wm_0, \pha_0$. Then the \GWP \eqref{eq:GWP_factorization} with parameters $(\pos,\mom,\fcQ,\fcP,\pha)$ solving 
	\begin{subequations}\label{eq:hageqmo}
		\begin{align+}
		%	\dot{\pos} &= \mom - \langle \mgPot\rangle_\vsol,\\
		%	\dot{\mom} &= \frac{\scp}{2}\l\langle\nabla \tr\l(J_A \RC \ICinv\r)\r\rangle_\vsol  + \langle J_\mgPot \rangle_\vsol^T \mom  - \langle \nabla \mPot \rangle_\vsol,\\ 
		\dot{\fcQ} &= \fcP - \l\langle J_\mgPot\r\rangle_\vsol \fcQ, \label{eq:eqmo_Q}\\
		\dot{\fcP} &= \langle J_\mgPot\rangle_\vsol^T \fcP + \frac{\scp}{2} \bigl\langle \nabla^2 \tr\l(J_\mgPot \RC \ICinv\r)\bigr\rangle_\vsol \fcQ 
		+
		\l\langle  D^2_{\mgPot,\mom}\r\rangle_\vsol \fcQ 
		 -
		  \bigl\langle \nabla^2 \mPot \bigr\rangle_\vsol \fcQ, \label{eq:eqmo_P}
		\end{align+} % \l\langle D^2_{\mgPot,\re C(x-\pos)} \r\rangle_\vsol
		%- \frac{1}{2}\l\langle \nabla^2 |A|^2 \r\rangle_\vsol \fcQ
	\end{subequations}
	and \eqref{eq:eqmo_q}, \eqref{eq:eqmo_p}, and \eqref{eq:eqmo_pha} is the variational solution \eqref{eq:var} with initial data 
	\begin{equation}
	(\pos(0),\mom(0), \wm(0), \pha(0)) = (\pos_0, \mom_0, \wm_0, \pha_0).
	\end{equation}
	If the initial matrices $\fcQ_0$ and $\fcP_0$ are symplectic, then $\fcQ(t)$ and $\fcP(t)$ are symplectic for all times $t \in\R$.
	%	die Phasenfunktion $\zeta$ erf\"ulle die gleiche Differentialgleichung, wie in Satz $\ref{Satz mag Bewgl}$. Dann gilt:
	%	\begin{itemize}
	%		\item[(1)] Das Gau\ss'sche Wellenpaket $u$ in $\eqref{eq:GWP_Faktor}$ mit den Parametern $(q,p,Q,P,\zeta)$ ist die variationelle L\"osung $\eqref{mag_var_Formul}$.
	%		\item[(2)]  Falls die Anfangsdaten $Q(0)$ und $P(0)$ die Symplektizit\"atsbedingungen $\eqref{PQ_transponiert}$ und $\eqref{PQ_konjugiert}$ erf\"ullen, dann erf\"ullen die Faktoren $Q(t), P(t)$ diese ebenfalls f\"ur jedes $t$.
	%	\end{itemize}
\end{corollary}
%In \cite{KinO20}, the same equations as in \Cref{thm:hagmotion} are derived, using a different approach and the %authors show that the flow of \eqref{eq:hageqmo} is symplectic. 
The proof of \Cref{thm:hagmotion} is given in \Cref{sec:eqmo}.

\subsection{Equations of motion in the limit $\scp\to 0$}\label{subsec:eqmo}

%The classical, \timedependent Hamiltonian function for charged particles in a magnetic field is given by  %corresponding to \eqref{eq:hamop}
%%
%\begin{equation}
%\cham(t,q,p) = \frac{1}{2}|p|^2-\mgPot(t,q)\cdot p +\frac{1}{2}|\mgPot(t,q)|^2 + \mltPot(t,q), \label{eq:clham}
%\end{equation}
%%
%cf.~\cite{GusS20_book, Hal13_book}.
The classical Hamiltonian function \eqref{eq:clham} induces the non-autonomous classical Hamiltonian system
%
%
%\todoin{vektorwerting mit $J^{-1}$}
\begin{equation}\label{eq:cldiff}
\begin{aligned}
\begin{pmatrix}
\dot{\varone}(t)\\
\dot{\vartwo}(t)
\end{pmatrix}
&= J^{-1}\nabla \cham(t,\varone(t), \vartwo(t))
\\
&=
\begin{pmatrix}
\vartwo(t) -\mgPot(t,\varone(t))\\
J_\mgPot^T(t,\varone(t))\vartwo(t)
%-\frac12\nabla|\mgPot(t,\varone(t))|^2
-\nabla \mPot(t,\varone(t))
\end{pmatrix}%\\
% \begin{pmatrix}
% \pos(s)\\
% \mom(s)
% \end{pmatrix}
% &=
% \begin{pmatrix}
% \pos_s\\
% \mom_s
% \end{pmatrix}
\end{aligned}
%
%
%\begin{aligned}
%\dot{q}(t)&= \nabla_p \cham(q,p,t)= p -\mgPot(q,t),		&		 q(s) &= q_s,\\
%\dot{p}(t)&= -\nabla_q \cham(q,p,t)= J_A^T(q,t)p-\frac{1}{2}\nabla |\mgPot(q,t)|^2-\nabla \mltPot(q,t),	&		 p(s) &=p_s,
%\end{aligned}
\end{equation}
with initial data $\varone(s)=\varone_s,\vartwo(s)= \vartwo_s$ and with $J$ defined in \eqref{eq:PQSym}. 
\alt{Since $\mgPot(t,\varone)$ and $\mPot(t,\varone)$ are sublinear and subquadratic with respect to $\varone$, the classical Hamiltonian function $\cham(t,\varone,\vartwo)$ is subquadratic with respect to $(\varone,\vartwo)$. This provides a globally Lip\-schitz continuous 
	right-hand side for the ordinary differential equation \eqref{eq:cldiff}, and the Picard--Lindel\"of theorem guarantees a unique global solution.}
%\slin{Throughout this paper, we assume that the ordinary differential equation \eqref{eq:cldiff} has unique solution on a fixed time interval $[0,\etime]$.}
Since $\mgPot(t,\varone)$ and $\mPot(t,\varone)$ are sublinear and subquadratic with respect to $\varone$,
the right-hand side for the ordinary differential equation \eqref{eq:cldiff} is locally Lip\-schitz continuous. There is no \blowup, since
\begin{align}
\frac12\partial_t \big(\abs{\varone}^2 + \abs{\vartwo}^2\big) &= \varone^T (\vartwo - \mgPot(\varone)) + \vartwo^T (J_\mgPot^T(\varone)\vartwo -\nabla \mPot(\varone))\\
&\le C\big(1+\abs{\varone}^2 + \abs{\vartwo}^2\big),
\end{align}
where the constant $C>0$ depends on bounds of the potentials. 
By Gronwall's lemma, there is no finite time \blowup.
%\begin{equation}\label{eq:Jmatrix}
%J \coloneqq
%\begin{pmatrix}
%0& -\Id\\
%\Id&0\\
%\end{pmatrix}.
%\end{equation}
This provides the existence of a unique global solution.
The bound in \cite[Lemma~3.15]{LasL20} states that $\langle \cdot\rangle_\vsol$ tend to point evaluations at $\pos$ as $\scp \to 0$, i.e.,
% \textcolor{red}{e.g.},
$\langle \mgPot \rangle_\vsol \to \mgPot(\pos)$.
%These convergence properties are due to the fact that in the same manner as in \cite[Lemma~3.15]{LasL20} it can be shown that averages tend to point evaluations at $\pos$ as $\scp \to 0$.
Hence, we observe that the magnetic equations of motion \eqref{eq:eqmo_q} and \eqref{eq:eqmo_p} tend to classical equations \eqref{eq:cldiff} as ${\scp \to 0}$ and \eqref{eq:eqmo_pha} to
\begin{align}
\dot{\pha}= \frac{1}{2}|\mom|^2  -  \mPot(\cdot,\pos).
\end{align}
In order to link the set of equations \eqref{eq:hageqmo} to classical mechanics, we consider the linearization of \eqref{eq:clham} along the position and momentum parameters $(\pos, \mom)$, i.e.,
\begin{equation} \label{eq:eqmo_lin_classic}
\begin{aligned}
\begin{pmatrix}
\dot{\fcQ}\\
\dot{\fcP}
\end{pmatrix}
&= J^{-1}\nabla^2 \cham(\cdot,\varone,\vartwo)
\begin{pmatrix}
\fcQ\\
\fcP
\end{pmatrix} \\
&=\begin{pmatrix}
\fcP - J_\mgPot(\cdot,\varone)\fcQ\\
\left(D^2_{\mgPot(\cdot,\varone), \vartwo}  - \nabla^2 \mPot(\cdot,\varone)\right)\fcQ + J_\mgPot(\cdot,\varone)^T \fcP
\end{pmatrix}.
%&\fcP - J_\mgPot(\pos, \cdot)\fcQ,\\
%&\left(D^2_{\mgPot(\pos, \cdot), \mom}  - \nabla^2 \mPot(\pos,\cdot)\right)\fcQ + J_\mgPot(\pos,\cdot)^T \fcP.
\end{aligned}
\end{equation}
By the same reasoning, we observe that the
equations \eqref{eq:hageqmo} tend to the linearized equations classical equations \eqref{eq:eqmo_lin_classic} as ${\scp \to 0}$. %which shows a correspondence between classical and quantum mechanics.

%
%\textcolor{red}{
%It is known that in the \timeindependent case, the classical flow \eqref{eq:clflow} is symplectic, i.e.,
%\begin{equation}\label{eq:symplecticflow}
%(D\flow^t)^T J D\flow^t = J,
%\end{equation}
%where $D\flow^t$ denotes time derivative with respect to the initial data, cf.~\cite{LasL20,Wal07_book}. Property \eqref{eq:symplecticflow} also holds true in the \timedependent case which motivates us to consider a symplectic factorization \eqref{eq:Hag_fac} for the \timedependent case.
%}
%
%\todoin{ \begin{itemize}
%		\item Bewegungsgleichungen \eqref{eq:eqmo}
%		\item $\scp \to 0$ klassischen BwgGl. mit symplektischen Fluss \eqref{eq:symplecticflow}
%		\item $D \Phi$ loest \eqref{eq:eqmo_lin_classic}% und \eqref{eq:linearized_eqmo}
%		
%		\item Hagedorn Zerlegung und $\scp \to 0$ liefert das gleiche Ergebnis
%		
%\end{itemize} 
%Im wesentlichen eine weitere Einsicht in den Zusammenhang von klass. und Quantenmechanik, noch plausibler Ansatz/weitere Rechtfertigung?
%}
%\todoin{
%	Brauchen fuer spaeter nur \eqref{eq:eqmo} und den Fluss zu den klassischen Gleichungen \eqref{eq:cldiff}, Hagedorn nicht mehr
%}

% %
%##########################

\subsection{Averages} \label{subsec:averages}

A further remarkable property of \GWPs is the conservation of several physical quantities. In the following, we recall the definitions of the linear and angular momentum for quantum dynamical systems.

Let $x = (x_1, \ldots, x_N)$, where $x_k \in \R^3,\, k = 1,\ldots, N$
and $d = 3N$, be position variables. We recall the follwoing definition given in \cite[Chapter~3]{LasL20}.

\begin{definition} \label{def:momentum_op}
	\begin{enumerateletters}%[leftmargin=0.65cm]
		\item The quantum mechanical total linear momentum operator is given by
		\begin{equation}%\label{eq:lin_Mom}
		\linMom := -\ii \scp\sum_{k=1}^N{ \nabla_{x_k}}.
		\end{equation}
		\item The quantum mechanical total angular momentum operator is given by
		\begin{equation}%\label{eq:ang_Mom}
		\angMom := \sum_{k=1}^N{x_k \times \left( -\ii \scp \nabla_{x_k}\right)} = -\ii \scp \sum_{k=1}^N\begin{pmatrix}
		x_{k_2}\partial_{k_3}-x_{k_3}\partial_{k_2}\\
		x_{k_3}\partial_{k_1} - x_{k_1}\partial_{k_3}\\
		x_{k_1}\partial_{k_2} -x_{k_2}\partial_{k_1}
		\end{pmatrix}. %\,\mathrm{Hertischer\, Teil\, von\,}
		\end{equation}
%		
%		\item We call a potential $V$ translation invariant, if % f\"ur alle $r\in \R^3$ gilt
%		\begin{equation}
%		V(x_1, \ldots ,x_N) = V(x_1 + r, \ldots , x_N + r),
%		\end{equation}
%		for all $r\in \R^3$. We call $V$ rotation invariant if for all orthogonal matrices $R \in \R^{3\times 3}$ with $\det R =1$ it holds true that
%		\begin{equation}
%		V(x_1, \ldots , x_N) = V(R x_1 , \ldots , R x_N ).
%		\end{equation}
	\end{enumerateletters}
\end{definition}
%\noindent\textbf{Averages}
%\todoin{Überleitungssatz; \Cref{def:pot_trans_rot} direkt für vektorwertige Potentiale (nach $\R^d$), ersetze $V$ durch $W$}
%
%\textcolor{red}{Next, we state properties for the potentials $\mgPot$ and $\mltPot$, which are sufficient for averages of the observables from \Cref{def:momentum_op} to be conserved. 
%}
%\textcolor{blue}{
	Next, we state sufficient conditions on the potentials $\mgPot$ and $\mltPot$, which lead to the conservation of
averages of the observables from \Cref{def:momentum_op}. 
%}

\begin{definition}\label[definition]{def:pot_trans_rot}
We call a potential $W = (W_j)_{j=1, \ldots, \dim}: (\R^3)^N \to \R^\dim$
\begin{enumerateletters}
	\item translation invariant, if % f\"ur alle $r\in \R^3$ gilt
	\begin{equation}
	W_j(x_1, \ldots ,x_N) = W_j(x_1 + r, \ldots , x_N + r),
	\end{equation}
	for all $r\in \R^3$ and $j= 1, \ldots, \dim$,
	
	\item rotation invariant if for all orthogonal matrices $R \in \R^{3\times 3}$ with $\det R =1$ it holds
	\begin{equation}
	W_j(x_1, \ldots , x_N) = W_j(R x_1 , \ldots , R x_N ),
	\end{equation}
	where $j=1, \ldots, \dim$.
\end{enumerateletters}
\end{definition}

In the next lemma we provide a representation for the energy and state conservation properties of the momenta.

\begin{lemma}\label[lemma]{lem:norm_energy_mom}
%	Let $\sol, \vsol$ be the solution of \eqref{eq:sproblem} and \eqref{eq:var}, respectively. 
	The following assertions hold.
	\begin{enumerateletters}%[leftmargin=0.6cm]
		\item We have $\normLtwo{\sol(t)} = \normLtwo{\vsol(t)}  = \normLtwo{\vsol_0}$ for all $t\in\R$.
		\item If the potentials $\mgPot$ and $\mltPot$ are both \timeindependent, then% the energy conservation holds with
		\begin{equation}
		\langle \Ham \rangle_{\sol(t)} = \langle \Ham \rangle_{\sol_0}\quad\text{ and }\quad\langle \Ham \rangle_{\vsol(t)} = \langle \Ham 	\rangle_{\vsol_0}.
		\end{equation}
		%		 the energy $\langle \Ham \rangle_\sol$ of the exact solution and the energy $\langle \Ham \rangle_\vsol$ of the variational solution is conserved.
		%
		%
		%
		%
		\item For $\varphi = \sol, \vsol$ the energy $\langle \Ham \rangle_\varphi$ is given by
		\begin{align}
		\langle \Ham(t)\rangle_{\varphi(t)} &= \langle  \Ham(0)\rangle_{\varphi(0)} + \int_0^t{ \bigl\langle  \ii\scp \partial_s \mgPot(s)\cdot \nabla \bigr\rangle_{\varphi(s)} + \bigl\langle\partial_s\mPot(s)\bigr\rangle_{\varphi(s)} }\,\dd s .
		\end{align}
		\item For $\linMom$ and $\angMom$ from \Cref{def:momentum_op} we have:
		%For the linear momentum operator $\linMom$ and the angular momentum operator $\angMom$ from \Cref{def:momentum_op} we have:
		\begin{itemize}
			\item[(i)] If $\mltPot$ and $\mgPot = (\mgPot_j)_{j=1}^\dim$ given in  \Cref{ass:onpotentials} are invariant under translations
%			 for all $j =1,\ldots,d$, then% the total linear momentum $\langle \linMom \rangle_{\vsol(t)}$ is conserved.
			%
			\begin{equation}
			\langle \linMom \rangle_{\sol(t)} = \langle \linMom \rangle_{\sol_0}
			\quad\text{ and }\quad		
			\langle \linMom \rangle_{\vsol(t)} = \langle \linMom \rangle_{\vsol_0}.
			\end{equation}
			
			\item[(ii)] If $\mPot$ defined in \eqref{eq:hamop} is invariant under rotations and $\mgPot(\cdot,x)= \alpha(\cdot) x$ for some $\alpha(\cdot) \in \R$, then
			\begin{equation}
						\langle \angMom \rangle_{\sol(t)} = \langle \angMom \rangle_{\sol_0}
						\quad\text{ and }\quad
						\langle \angMom \rangle_{\vsol(t)} = \langle \angMom \rangle_{\vsol_0}.
			\end{equation}
			% the total angular momentum $\langle \angMom \rangle_{\vsol(t)}$ is conserved.
			%\item Falls $\|u(t)\|_{L^2} =1$ ist, gilt (klassisches Variante). (drin lassen? klassische Sachen werden erst später eingeführt)
		\end{itemize}
	\end{enumerateletters}
\end{lemma}
The proof of \Cref{lem:norm_energy_mom} is given in \Cref{sec:expv}.

\subsection{\texorpdfstring{$L^2$-error bound}{L2-error bound}}

In this section, we present the approximation property of the Gaussian wave packet \wrt the $L^2$-norm.
Since our error bounds depend on parameters characterizing the
%	 the approximating
	  Gaussian wave packet in \eqref{eq:gwp}, we first consider
%	   \Cref{lem:bddparam} provides 
	   the boundedness of these parameters up to a fixed but arbitrary finite time $\etime>0$ specified by \ODE-theory.

\begin{lemma}\label[lemma]{lem:bddparam}
%\begin{lemma} \label{lem:bddparam}
For all times $\etime >0$, the set of equations \eqref{eq:eqmo} is \wellpoAdj on $[0,\etime]$ independently of $\scp$. Furthermore, the solution parameters are bounded independently of $\scp$, i.e.
\begin{equation}
|\param|\leq c_{\param_0}, \qquad \text{ for all }~ \param \in \{\pos, \mom, \wm, \pha\},
\end{equation}
uniformly on $[0,\etime]$, where $c_{\param_0}$ depends on the parameters of the initial Gaussian $\vsol_0$, on the potentials $\mltPot, \mgPot$, and on $\etime$.
%the equations \eqref{eq:eqmo}.
%
%
%	
%	There exists a time $\etime >0$, such that the set of equations \eqref{eq:eqmo} is \wellposed independent of $\scp$, i.e.,
%	there is a lower bound, independent of $\scp$, on the time $\etime$ and the solution parameters are bounded independently of $\scp$ uniformly on $[0,\etime]$.
	%
	%
	%We assume that the parameters solving \eqref{eq:eqmo} are bounded uniformly on $[0,\etime]$.
\end{lemma}
%
%\erkl{There is no finite time blow up: we use energy estimates and 
%similiarly as for the classical equations, we calculate
%\begin{align}
%\langle\pos, \dot{\pos}\rangle &= \langle \pos, \mom\rangle - \Big\langle \pos,\langle\mgPot(\pos) + J_\mgPot(\pos)(\pos + \theta(x-\pos))\rangle_\vsol \Big\rangle,
%\end{align}
%and 
%\begin{equation}
%\mgPot(\pos) = \mgPot(\pos_0) + J_\mgPot\big(\pos_0)(\pos + \theta (\pos_0 -\pos)\big).
%\end{equation}
%Since $\mgPot$ is sublinear, $J_\mgPot$ is bounded and we obtain bounds with respect to the initial data and derivatives of $\mgPot$, which read
%\begin{align}
%\langle\pos, \dot{\pos}\rangle & \le C(q_0, \mgPot)\big(\|\pos\|^2 + \|\mom\|^2\big).
%\end{align}
%Similarly for $\langle\mom, \dot{\mom}\rangle $, since $\mPot$ is subquadratic.}

We note that by \Cref{thm:hagmotion} the matrix $\IC$ is real symmetric, positive definite for all times $t$. To formulate the following results, we denote by $\rho >0$ a lower bound on the smallest eigenvalue of $\IC$ on the finite time horizon $[0,\etime]$. For a discussion of relevant time scales on which $\rho$ is sufficiently large compared to $\eps$, called the Ehrenfest time, we refer to \cite[Sec.~3.6]{LasL20}. 
With this, 
%Having \Cref{lem:bddparam}
we can state our approximation result.
%\todoin{Brauchen wir hier schon \Cref{lem:bddparam}? }
%
\begin{theorem} \label{thm:L2} %%aus waves-abstract
	Let $\sol, \vsol$ be the solution of \eqref{eq:sproblem} and \eqref{eq:var}, respectively,
	and let $\vsol_0$ satisfy \eqref{eq:gauss_initialdata}. 
	% If $\sol_0 \in \Mf$,
	Then the error bound 
	\begin{equation}
	\normLtwo{\sol(t) -\vsol(t)} \leq t c \sqrt{\varepsilon}, \qquad t  \in [0,T],
	\end{equation}
	holds with a constant $c$ which depends on $\rho$, the bounds on the parameters from
	 \Cref{lem:bddparam}
	 and on the potentials, but is independent of $\varepsilon$ and $t$.
\end{theorem}

	We provide the 
	%	The 
	details and the proof
	of the theorem
	%	 of the approximation of the wave function are provided
	in \Cref{sec:L2}.

\subsection{Observable error bound}

In classical mechanics physical states are described by the position and momentum parameters $\varone, \vartwo \in \R^{\dim}$. Observables are functions depending smoothly on $(\varone, \vartwo ) \in \R^{\dim}\times\R^{\dim}$, see, for example, \cite{Hal13_book,Wal07_book}. Classical mechanics can be linked to quantum mechanics via \Weylquantization, which asigns a classical observable to a quantum mechanical one using \semiclassical Fourier transformation, cf.~\cite[Thm.~4.14]{EvaZ07_book} or \cite{Mar02_book,Hal13_book}. Formally, for $ \varphi \in \SchF(\R^\dim)$ and an observable $\clobs$, we define
\begin{align}\label{eq:Weyl_def}
\op(\clobs)\varphi(x) &\coloneqq \frac{1}{(2\pi\scp)^{d}}\int_{\R^{2\dim}}{\clobs\Bigl(\frac{x+\varone}{2},\vartwo\Bigr)\e^{\ii \vartwo \cdot(x-\varone)/\scp}\varphi(\varone)}\,\dd(\varone,\vartwo).
%\\
%
%\text{ for all } \varphi &\in \SchF(\R^\dim)
%\nonumber
 %\label{eq_Def_Weyl_Intgr}\\
%&= \frac{1}{(2\pi\eps)^\frac{d}{2}}\int{K_b(x,\varone)\varphi(\varone)}\,\dd \varone,\nonumber
\end{align}
 The \Weylquantization of the projections to the first or second component of the classical variables are
\begin{align}
\op(\vartwo)\varphi = -\ii\varepsilon \nabla \varphi\quad\text{ and } \quad\op(\varone)\varphi = x\varphi.
\end{align}
Further examples of physically relevant observables stemming from classical symbols are 
\begin{align}
\op(|\vartwo|^2)\sol(x) &= -\scp^2\Delta \sol(x)
\end{align}
and, due to $\divergence \mgPot = 0$, 
\begin{align}
\op(\mgPot(\varone)\cdot \vartwo)\sol(x) 
&= \tfrac12\left(\mgPot(x) \cdot(-\ii\scp \nabla) + (-\ii\scp\nabla)\cdot \mgPot(x)\right) \sol(x)\\
&= \left(\mgPot(x) \cdot(-\ii\scp \nabla)\right) \sol(x),
\end{align}
and, of course, 
\[
\op(\cham(t))\sol(x) = \Ham(t)\sol(x)
\]
for the Hamiltonian function \eqref{eq:clham} and the magnetic Schr\"odinger operator \eqref{eq:ham}.
An observable $\obs= \op(\clobs)$ defines for an $L^2$-normalised function 
$\varphi\in\SchF(\R^\dim)$ an expectation value,
\[
\bigl\langle \varphi| \obs \varphi \bigr\rangle = \int_{\R^\dim} \overline{\varphi(x)} (\obs\varphi)(x) \dd x,
\]
and we investigate how expectation values issued by the variational approximation~$\vsol(t)$ differ from the ones of the true solution $\sol(t)$. For an error estimate relying on $L^2$ bounds, we have to restrict ourselves to sublinear classical observables.

%In this case, we also incorporate suitable seminorms from \Cref{def:poly_obs}, see also \cite{RobC21_book}.
%\todoin{seminorms jetzt Verweis nach oben?}

\begin{definition}\label{def:admissible-symbols}
	The class of sublinear classical symbols is defined as smooth functions $\clobs : \R^{2\dim} \to \R$ such that for $\alpha \in \N^{2\dim}_0$ with $\abs{\alpha} \ge 1$ there exists $C_\alpha>0$
\begin{equation}
\abs{\partial^\alpha \clobs(\varone,\vartwo)} \le C_\alpha
\end{equation}
for all $(\varone,\vartwo)\in\R^\dim\times\R^\dim$.
\end{definition}

For the expectation values of classical sublinear observables, we obtain the following error 
estimate that generalizes and improves the findings of N.~King and T.~Ohsawa \cite{KinO20,Ohs21}, where asymptotic accuracy of the order $\scp^{3/2}$ has been observed and proved for the variational position and momentum expectation value.
\begin{theorem} \label{thm:obs} %%aus waves-abstract
	Let $\sol, \vsol$ be the solution of \eqref{eq:sproblem} and \eqref{eq:var}, respectively,
	and let $\vsol_0$ satisfy \eqref{eq:gauss_initialdata}.
%	 and $\sol_0 = \vsol_0\in \Mf$ such that $\normLtwo{\vsol_0}=1$.
%
Moreover, let $\obs= \op(\clobs)$ be an observable stemming from a classical sublinear observable $\clobs$ in the sense of \Cref{def:admissible-symbols} 
%Further, assume that % for $j,m,l,k \in \N_0$ 
such that $\egor{t}{s}$ is sublinear.
%	\begin{enumerate}[leftmargin=0.65cm]
%		\item 
%
%%		}
%%		\clobs \circ \flow^{t,s}
%%		 \in \symbclass{j,m}$,
%		\item $\mgPot \in \symbclass{k,0}$, and $\mltPot \in \symbclass{l,0}$.
%	\end{enumerate}
	%\todoin{Hier wirklich $A$ gemeint, nicht $a$?}
	Then we have the error bound %for $n := \max\{2k+3, l+3, k+4\}$
	\begin{equation}
	\bigl|\bigl\langle \sol(t)| \obs \sol(t) \bigr\rangle - \bigl\langle \vsol(t)| \obs \vsol(t) \bigr\rangle \bigr|\leq 
	t\,c\,\scp^2, %{\vaiindex{m}}{n+j},
	\end{equation}
	%
%	\textcolor{rosa}{where $k=1$ for $\mgPot\neq0$ and $k=2$ for $\mgPot=0$.}
for all $t\in[0,T]$. The error constant $c$ depends on the parameter bounds of \Cref{lem:bddparam} for the time-interval $[0,T]$, in particular on 
the bounds for the width matrix $\wm$, on the potentials, and on $\clobs$, but is independent of $\scp$ and~$t$. %The index $\vaiindex{m}$ only depends on $m$ and $\dim$.
\end{theorem}
%\car{Hier war eine Bemerkung zur Sublinearit\"at von $\egor{t}{s}$ , die ich auskommentiert habe, weil sie mir an der Stelle unpassend erschien. Mir erscheint es schwierig Zeit-Kontrolle \"uber die Fehlerkonstanten hier untechnische anzusprechen. Da wir hinten beim Egorov-Satz dies alles klar benennen, w\"are mein Vorschlag es hier so wie im  Theorem zu belassen.}
%\car{Ich hab die Bemerkung editiert, denke aber dass wir sie weglassen sollten.\begin{remark}
%		The assumption on sublinearity of $\egor{t}{s}$ in \Cref{thm:obs} concerns the growth of the bounding constants with respect to time, see the precise bounds given in \eqref{eq:sublinear_time} later on. 
%		Sublinearity per se can be proven using Gronwall's lemma without control on the error growth with respect to time. 
%\end{remark}}

Note that the convergence in the observables is of order $\eps^2$, 
%with $k=1,2$ depending on the presence of a magnetic field, 
while the convergence in the $L^2$-norm presented in \Cref{thm:L2} is of order $\sqrt{\eps}$. This is
an improvement of the results obtained in \cite[Theorem~3.5]{LasL20}, where $\mathcal O(\sqrt\scp)$ norm accuracy and an $\mathcal O(\scp)$ estimate for the non-magnetic observable error were proved.
The rest of the paper is devoted to the proofs of the equations of motion and the error estimates presented in this section.

%\todoin{Texte zum Ueberleiten darauf, dass nur noch Beweise kommen.}

%% file: equations_motion_sublinear.tex
\section{Equations of motions: proof of \Cref{thm:eqmo,thm:hagmotion}}\label{sec:eqmo}

In this section we derive equations of motion for the parameters $(\pos, \mom, \wm, \pha)$ as well as for the factorization matrices $\fcQ$ and $\fcP$.
To do so, we compute both sides of \eqref{eq:var} and compare the coefficients.

% First, we state an auxiliary lemma which gives an explicit representation for the projection of a potential times a Gaussian wave packet to the tangent space.
%\begin{prop}[{\cite[Prop.~3.14]{LasL20}}] \label[prop]{lem:proj_pot_gwp}
%Let $W : \R^\dim \times \R \to \R $ be a smooth potential and $\vsol$ a Gaussian wave packet such that ${\|\vsol\|_{L^2} =1}$. Then we have
%%For a Gaussian wave packet $\vsol$ with ${\|\vsol\|_{L^2} =1}$ and a scalar smooth potential $W$ we have
%%
%\begin{align}
%\Pr\bigl(W\vsol\bigr) =\bigl(\alpha + v^T(x-\pos) + \tfrac{1}{2}(x-\pos)^TB(x-\pos)\bigr) \vsol,
%%\Pr\bigl(i\varepsilon A\cdot\nabla\fsol\bigr) =\bigl(\beta + w^Tx_q + \tfrac{1}{2}x_q^TDx_q\bigr) u
%\end{align}
%where $\alpha, v, B$ are given by
%\begin{align}
%\alpha &=\langle W\rangle_u - \frac{\varepsilon}{4} \tr\left(\im \wm^{-1}\left\langle \nabla^2 W\right\rangle_\vsol\right),
%\quad v =\langle \nabla W\rangle_u, \quad B = \langle \nabla^2 W\rangle_u.
%\end{align}
%Here we used the notation $\langle W\rangle_\vsol = \langle \vsol |W\vsol\rangle$.
%\end{prop}
%Using \Cref{lem:proj_pot_gwp} we can give a proof of \Cref{thm:eqmo}.

\begin{proof}[Proof of \Cref{thm:eqmo}]
	In order to use the formula for the orthogonal projection derived in {\cite[Prop.~3.14]{LasL20}} for \eqref{eq:var}, we observe that derivatives \wrt $x$ of a \GWP turn into
	scalar functions of 
%	the spatial variable
	 $x$ % \in \R^\dim$
%	 potential
	  times $\vsol$. For notational simplicity, we omit the \timedependence and in the potentials $\mgPot$ and $\mltPot$ we omit the space variable $x$. In particular, we have
	  \begin{subequations}\label{eq:derivative-gwp-all}
	  	\begin{align+}
	  \ii\scp \mgPot \cdot \nabla \vsol &= -\mgPot\cdot\bigl(\wm (x-\pos)+\mom\bigr) \vsol, \label{eq:advection-gwp}\\
	  -\frac{\scp^2}{2}\Delta \vsol &= \Bigl(\frac{1}{2}(x-\pos)^T\wm^2(x-\pos) + \mom^T\wm (x-\pos)+\frac{1}{2}|\mom|^2 - \frac{\ii\scp}{2}\tr(\wm)  \Bigr) \vsol, \label{eq:laplace-gwp}
	  \end{align+}
	  \end{subequations}
	and for the time derivative it holds that
	\begin{equation}\label{eq:formula-dt-vsol}
	\ii\scp \pt \vsol(\cdot,x) = 
	\Bigl(-\frac{1}{2}(x-\pos)^T\dot{\wm}(x-\pos) + \dot{\pos}^T\wm(x-\pos) - \dot{\mom}^T(x-\pos) + \mom^T \dot{\pos} -\dot{\zeta}\Bigr) \vsol.
	\end{equation}
	%
	%\todoin{Referenz zu klassisch (im Beweis?)} 
	%
	Motivated by the classical magnetic Hamiltonian system \eqref{eq:cldiff},
	we eliminate one degree of freedom by setting $\dot{\pos} = \mom - \langle \mgPot\rangle_\vsol$, see \cite{Hal13_book, GusS20_book}. 
%\textcolor{red}{degree of freedom from $x^3$}
	%
%	 and \Cref{sec:obs}.
%	\todoin{Was macht \Cref{sec:obs} hier?}
%	
	Incorporating the above formulas, we compare the coefficients in $x$ on both sides of \eqref{eq:var} and arrive at equations of motions of the form
	\begin{align}
	\dot{\pos} &= \mom - \langle \mgPot\rangle_\vsol,
	\\
	\dot{\mom} &= \bigl\langle J_\mgPot^T \RC (x-\pos)\bigr\rangle_\vsol + \bigl\langle J_\mgPot \bigr\rangle_\vsol^T \mom  - \bigl\langle \nabla \mPot \bigr\rangle_\vsol,\\
	%
	%%%%%%%%%%%%%%%%%%%%%%%%%%%%%%%%%%%%%%%%
	%
	\dot{\wm} &= - \wm^2 + \bigl\langle D^2_{\mgPot,\RC(x-\pos)} \bigr\rangle_\vsol + \bigl\langle  D^2_{\mgPot,\mom}\bigr\rangle_\vsol + \bigl\langle J_\mgPot\bigr\rangle_\vsol^T \wm + \wm \bigl\langle J_\mgPot\bigr\rangle_\vsol  - \bigl\langle \nabla^2 \mPot \bigr\rangle_\vsol,\\
	%
	%%%%%%%%%%%%%%%%%%%
	%
	\dot{\pha} &= \frac12 |\mom|^2
	 +
	  \l\langle \mgPot^T \RC(x-\pos) \r\rangle_\vsol 
	  + 
	  \frac{\ii \eps}{2} \tr (\wm)  \\
	&\quad - \frac{\eps}{4}\tr
	\bigl(\ICinv
	\bigl( \bigl\langle D^2_{\mgPot,\RC(x-\pos)} \bigr\rangle_\vsol 
	+
	 \l\langle J_\mgPot\r\rangle_\vsol^T \RC 
	 +
	 \RC \langle J_\mgPot\rangle_\vsol 
	 +
	  \langle  D^2_{\mgPot,\mom}\rangle_\vsol\bigr)
	\bigr)\\
	&\quad  - \langle \mPot\rangle_\vsol + \frac{\eps}{4}\tr\bigl(\ICinv \bigl\langle \nabla^2 \mPot\bigr\rangle_u \bigr).
	%- \frac12\l\langle |\mgPot|^2\r\rangle_\vsol + \frac{\eps}{8}\tr\l(\ICinv \l\langle \nabla^2 |\mgPot|^2\r\rangle_\vsol\r)
	\end{align}
	It remains to extract the additional power of $\scp$ from the terms that contain the difference $x-\pos$.
	From
	\begin{equation}
	|\vsol(x)|^2 = \exp\Bigl(-\frac{1}{\scp} (x-\pos)^T \IC (x-\pos) - \frac{2}{\scp} \im \zeta\Bigr)
	\end{equation}
	we obtain the derivative
	\begin{align}
	%|\vsol(x)|^2 &= \exp\left(-\frac{1}{\scp} (x-\pos)^T \IC (x-\pos) - \frac{2}{\scp} \im \zeta\right),\\ %\label{eq:Betrag u}
	\nabla |\vsol(x)|^2 &= -\frac{2}{\scp}\IC(x-\pos)|\vsol(x)|^2,\label{eq:nabla_abs_u} %\\
%	\nabla^2 |\vsol(x)|^2 &=  \frac{1}{\scp}\left(\frac{4}{\scp} \IC (x-\pos) (x-\pos)^T\IC -2 \IC \right)|\vsol(x)|^2,\label{eq:nabla2_abs_u}
	\end{align}
	and apply integration by parts to obtain
	\begin{align}
	\bigl\langle \mgPot^T \RC (x-\pos)\bigr\rangle_\vsol &= \bigl\langle \mgPot^T \RC \ICinv \IC (x-\pos)\bigr\rangle _\vsol
	\\
	&= \int_{\R^\dim} \mgPot^T \RC \ICinv \IC (x-\pos) |\vsol(x)|^2 \dd x
	\\
%	&= \frac{\eps}{2}\Bigl\langle \sum_{j,k,l =1}^\dim \partial_j \mgPot_k \RCkl{k}{l} \ICinvkl{lj} \Bigr\rangle\\%_{{\mkern-10mu}\vsol} \\
	%
	&= \frac{\eps}{2} \bigl\langle \tr\l(J^T_A \RC \ICinv\r) \bigr\rangle_\vsol .
	\end{align}
	Similarly, we gain an order of $\scp$ for
	\begin{align}
	\bigl( \bigl\langle J_\mgPot^T \RC (x-\pos)\bigr\rangle_\vsol \bigr)_{i}
	= \bigl(\bigl\langle J_\mgPot^T \RC \ICinv \IC (x-\pos)\bigr\rangle_{\vsol} \bigr)_i
	=\frac{\eps}{2} \bigl\langle\partial_i \tr\l(J^T_A \RC \ICinv\r) \bigr\rangle_\vsol,
	%
	%\\
%	&= \int J_\mgPot^T \RC \ICinv \IC (x-\pos) |\vsol(x)|^2 \dd x\\
%	&= \frac{\eps}{2}\l\langle \sum_{j,k,l =1}^\dim \partial_i \partial_j \mgPot_k \RCkl{k}{l} \ICinvkl{lj} \r\rangle\\%_{{\mkern-10mu}\vsol} \\
%	&= \frac{\eps}{2}\partial_i \tr\l(J_A \RC \ICinv\r).
	\end{align}
	as well as for
	\begin{align}
	\bigl(\bigl\langle D^2_{A,\RC(x-q)} \bigr\rangle_\vsol\bigr)_{ij} 
	%
%	&=  \bigl(\bigl\langle D^2_{A,\RC \IC \ICinv(x-q)} \bigr\rangle_\vsol\bigr)_{ij} \\
	%
%	&= \Bigl\langle \sum_{k,l,m =1}^\dim \partial_i \partial_j \mgPot_k \RCkl{k l}\ICinvkl{lm} \sum_{n=1}^d \ICkl{m n}(x_n-q_n) \Bigr\rangle_\vsol\\
	%
	&= \frac{\eps}{2} \Bigl\langle \sum_{k,l,m =1}^d \partial_m\partial_i \partial_j \mgPot_k \RCkl{k l}\ICinvkl{lm}  \Bigr\rangle_\vsol.
	\end{align} 
	By the identity 
	%
%	\begin{align}
%	\tr\l(J_\mgPot \RC \ICinv\r) &= \tr\l(\RC J_\mgPot  \ICinv\r)  =\sum_{k,m,l=1}^\dim  \RCkl{k}{l} \partial_m \mgPot_l \ICinvkl{mk}
%	\end{align}
%	and
	\begin{align}
	\partial_{ij} \tr\l(J_\mgPot^T \RC \ICinv\r) & = \sum_{k,m,l=1}^\dim \partial_{ij} \partial_m \mgPot_k\RCkl{k l} \ICinvkl{lm},
	\end{align}
	we conclude the equations of motion stated in \eqref{eq:eqmo}.
%	\qed
\end{proof}

We now turn to the equations of motion for the Hagedorn factorization \eqref{eq:Hag_fac}.
The idea is to show that the product $\fcP \fcQ^{-1}$ solves the same differential equation as $\wm$ and conclude with the uniqueness of the variational solution $\vsol$.
%\textcolor{red}{
%Starting from the equation of motion \eqref{eq:eqmo_C} we can now derive equations of motion for the factorization matrices $\fcP$ and $\fcQ$, respectively. 
%}

\begin{proof}[Proof of \Cref{thm:hagmotion}]
		We employ the differential identity
		\begin{equation}
		\pt(\fcQ^{-1} ) = -\fcQ^{-1}\pt{\fcQ}\fcQ^{-1} ,
%		 \quad \text{ where } \quad
%		 \pt \fcQ = \dot{\fcQ}, 
		\end{equation}
and the product rule to find that $\wm = \fcP\fcQ^{-1}$ satisfies the differential equation 
	\begin{equation}
	\dot{\wm} = -\fcP\fcQ^{-1}\dot{\fcQ}\fcQ^{-1} +\dot{\fcP}\fcQ^{-1}
	\end{equation}
	with $\pt \fcQ = \dot{\fcQ}$. Then, using \eqref{eq:hageqmo}, we see that this is the differential equation for $\wm$ in 
	\eqref{eq:eqmo_C}.
%% ausführlicher Rest
%		and use
%		\eqref{eq:hageqmo}
%		to find that  $\wm = \fcP\fcQ^{-1}$ satisfies the differential equation with $\pt \fcQ = \dot{\fcQ}$
%		%
%		\begin{equation}
%		\dot{\wm} = -\fcP\fcQ^{-1}\dot{\fcQ}\fcQ^{-1} +\dot{\fcP}\fcQ^{-1}
%		\end{equation}
%		\begin{align}
%		\dot{\wm} &= \fcP \pt(\fcQ^{-1}) +\dot{\fcP}\fcQ^{-1} \\
%%		& = -\fcP\fcQ^{-1}\dot{\fcQ}\fcQ^{-1} +\dot{\fcP}\fcQ^{-1} \\
%		%
%		&= -\fcP\fcQ^{-1}\l(\fcP - \langle J_\mgPot\rangle_u \fcQ\r)\fcQ^{-1}\\
%		%
%		&\quad  +\l(\bigl\langle D^2_{\mgPot,\RC(x-\pos)} \bigr\rangle_u \fcQ + \l\langle  D^2_{\mgPot,\mom}\r\rangle_\vsol \fcQ +  \langle J_\mgPot\rangle_\vsol^T \fcP - \bigl\langle \nabla^2 \mPot \bigr\rangle_\vsol \fcQ\r)\fcQ^{-1} \\  %- \frac{1}{2}\l\langle \nabla^2 |\mgPot|^2 \r\rangle_\vsol \fcQ
%		%
%		&=-\wm^2 + \wm\langle J_\mgPot\rangle_\vsol + \bigl\langle D^2_{\mgPot,\RC(x-\pos)} \bigr\rangle_\vsol + \bigl\langle  D^2_{\mgPot,\mom}\bigr\rangle_\vsol  +  \langle J_\mgPot\rangle_\vsol^T \wm  - \bigl\langle \nabla^2 \mPot \bigr\rangle_\vsol,
%		\end{align} %- \frac{1}{2}\l\langle \nabla^2 |A|^2 \r\rangle_u
%		which are the differential equations for $\wm$ in 
%		\eqref{eq:eqmo_C}.
%
%
% Ende differentialgleichung für C ausführlich	

		Concerning the symplectic relation in \eqref{eq:PQ}, we have %ist $C$ symmetrisch, denn
		\begin{equation}
		\pt (\fcQ^T \fcP - \fcP^T \fcQ) = \dot{\fcQ}^T \fcP + \fcQ^T \dot{\fcP} - \dot{\fcP}^T \fcQ - \fcP^T \dot{\fcQ},
		\end{equation}
		and by inserting the differential equations of $\fcP, \fcQ$ given in \eqref{eq:hageqmo}, we see that $\fcQ^T \fcP - \fcP^T \fcQ$ is constant.
		%
%		\begin{align}
%		\dot{\fcQ}^T \fcP &=  \fcP^T \fcP - \fcQ^T\langle J_\mgPot\rangle_u^T  \fcP\\
%		%
%		\fcQ^T \dot{\fcP} &= \fcQ^T \bigl(\bigl\langle D^2_{\mgPot,\RC(x-\pos)} \bigr\rangle_\vsol 
%		+
%		\l\langle  D^2_{\mgPot,\mom}\r\rangle_\vsol 
%		-
%		\bigl\langle \nabla^2 \mPot \bigr\rangle_\vsol  \bigr)\fcQ + \fcQ^T \langle J_\mgPot\rangle_\vsol^T \fcP\\ %- \frac{1}{2}\l\langle \nabla^2 |A|^2 \r\rangle_u 
%		%
%		\dot{\fcP}^T \fcQ &=  \fcQ^T\bigl(\bigl\langle D^2_{\mgPot,\RC(x-\pos)} \bigr\rangle_\vsol^T  
%		+
%		 \bigl\langle  D^2_{\mgPot,\mom}\bigr\rangle_\vsol^T  
%		 -
%		  \bigl\langle \nabla^2 \mPot \bigr\rangle_\vsol^T \bigr)\fcQ 
%		  + 
%		  \fcP^T\langle J_\mgPot\rangle_\vsol \fcQ\\
%		% - \frac{1}{2}\l\langle \nabla^2 |\mgPot|^2 \r\rangle_\vsol^T
%		%
%		\fcP^T \dot{\fcQ} &= \fcP^T\fcP - \fcP^T \langle J_\mgPot\rangle_\vsol \fcQ.
%		\end{align}
		The same calculation holds for $\pt(\fcQ^* \fcP - \fcP^* \fcQ)$ with $^*$ replaced by $^T$. Since $p,q,A$ and $V$ are real valued, we conclude
		\begin{equation}
		\pt(\fcQ^* \fcP - \fcP^* \fcQ) = \dot{\fcQ}^* \fcP + \fcQ^* \dot{\fcP} - \dot{\fcP}^* \fcQ - \fcP^* \dot{\fcQ} =0,
		\end{equation}
		which means that \eqref{eq:PQ} holds true for all times.	
%\qed
\end{proof}

%%%%
%%%% ab hier orthogonale Projektion von Caroline

\subsection{Equations of motion for a general Hamiltonian}

The findings of \Cref{thm:eqmo} for the magnetic  Schr\"odinger operator $\Ham(t)$ extend to the dynamics for general Hamiltonian operators that are the Weyl quantization of a smooth function 
$\cham:\R\times\R^{2\dim}\to\R$ of subquadratic growth, that is, for all 
$\alpha\in\N_0^{2\dim}$ with $|\alpha|\ge 2$ there exists $C_\alpha>0$ such that 
\begin{equation}\label{eq:general-ham-subquadr}
|\partial^\alpha \cham(t,\varone,\vartwo)| \le C_\alpha
\end{equation}
for all $t\in\R$ and $(\varone,\vartwo)\in\R^{2\dim}$.
Note that the classical magnetic Hamiltonian function \eqref{eq:clham} is not subquadratic, but our analysis works for both cases.
%\unklar{In this section, the subquadratic growth conditionis only needed for the existence of a unitary propagator. %The equations of motion can be obtained for a general Hamiltonian.}

A first step for the generalization is the construction of a suitable orthonormal basis of the tangent space of a Gaussian wave packet, which is done in \cite[Lemma~3.12 and Theorem~4.1]{LasL20} for the non-magnetic case, where only the modulus squared of the wave packet matters. For convenience, we state the representation formulas of the basis functions that we use. Consider a Gaussian wave packet $\vsol\in\Mf$ of unit norm, $\|\vsol\| =1$.  The family $\{\varphi_n\}_{|n|\le 2}$ with
	\begin{subequations}\label{eq:onb-all}
			\begin{align+}
		\varphi_0 &= \vsol,\\
		\varphi_{e_j} &= \sqrt{\frac{2}{\scp}}\ \left(Q^{-1}(x-\pscone)\right)_j u,\\
		\varphi_{e_j+e_k} &= \frac{1}{\sqrt{\delta_{kj}+1}}\,\left(\frac{2}{\scp}\,\left(Q^{-1}(x-\pscone)\right)_j \,\left(Q^{-1}(x-\pscone)\right)_k  - (Q^*Q^{-T})_{j,k} \right) \vsol ,
		%			&\varphi_{e_j+e_k} = -\frac{1}{\sqrt{\delta_{kj}+1}}\, (Q^*Q^{-T})_{j,k}\, u  + \frac{2}{\scp\sqrt{\delta_{kj}+1}}\,\left(Q^{-1}(x-\pscone)\right)_j \,\left(Q^{-1}(x-\pscone)\right)_k u,
	\end{align+}
	\end{subequations}
		is an orthonormal basis of the tangent space 
		$\mathcal T_u\mathcal M$ of $\mathcal M$ at $u$.
For calculating the orthogonal projection to the tangent space, we make use of another representation via the raising and lowering operators $\mathcal A_j^\dagger$ and $\mathcal A_j$. These are the $j$th component of the vector-valued operators
	\begin{align}
		\mathcal A^\dagger &= \frac{\ii}{\sqrt{2\scp}}\left( P^*\op(\varone-\pscone) -Q^*\op(\vartwo-\psctwo)\right),\\
		\mathcal A &= -\frac{\ii}{\sqrt{2\scp}}\left( P^T\op(\varone-\pscone) -Q^T\op(\vartwo-\psctwo)\right),
	\end{align}
	respectively.
%}
%\textcolor{purple}{
%	\begin{lemma}[Orthonormal basis for the tangent space]\label{lem:onb}
%		Consider a Gaussian wave packet $\vsol\in\Mf$ of unit norm, $\|\vsol\| =1$. 
Using the complete family of Hagedorn functions constructed by the infinite ladder process, we obtain that $\{\varphi_n\}_{|n|\le 2}$  with
%		The complete family of Hagedorn functions $\{\varphi_n\mid n\in\N_0^d\}$ is constructed by the infinite ladder process
		\begin{align}\label{eq:onb-ladder-all}
			\varphi_0 &= \vsol, \quad
			\varphi_{e_j} = \mathcal A_j^\dagger \vsol, \quad
			\varphi_{e_k + e_j} = \frac{1}{\sqrt{\delta_{kj} + 1}}\ \mathcal A_j^\dagger \ \mathcal A_k^\dagger \vsol,
		\end{align}
%		with $n\in\N_0^d$ and $j=1,\ldots,d$, where the lowering relation only applies if $n_j\ge1$.
%		
%		, and define the functions
%		\begin{align}
%			&\varphi_0 = \vsol,\\
%			&\varphi_{e_j} = \sqrt{\frac{2}{\scp}}\ \left(Q^{-1}(x-\pscone)\right)_j u,\\
%			&\varphi_{e_j+e_k} = -\frac{1}{\sqrt{\delta_{kj}+1}}\, (Q^*Q^{-T})_{j,k}\, u \\
%			&\qquad + \frac{2}{\scp\sqrt{\delta_{kj}+1}}\,\left(Q^{-1}(x-\pscone)\right)_j \,\left(Q^{-1}(x-\pscone)\right)_k u,
%		\end{align}
%		for $j,k=1,\ldots,d$.
%		 Then, the family $\{\varphi_n\}_{|n|\le 2}$ 
%		 is an orthonormal basis of the tangent space 
%		$\mathcal T_u\mathcal M$ of $\mathcal M$ at $u$, 
see also \cite[Chapter~V.2]{Lub08_book} or \cite[Theorem~3.3]{Hag98}.

Equipped with the orthonormal basis \eqref{eq:onb-all} and \eqref{eq:onb-ladder-all}, we can give an explicit formula for the quadratic polynomial generated by the orthogonal projection when acting on a general Hamiltonian operator.
\begin{prop}[Orthogonal projection]\label{prop:orthogonal-proj}
	Let $\cham:\R^{2\dim}\to\R$ be smooth and of growth. Let $\vsol\in\Mf$ be a Gaussian wave packet 
	of unit norm, $\|\vsol\| = 1$, with phase space center $z_0=(\pscone,\psctwo)\in\R^{\dim}\times\R^\dim$. Then, 
	\begin{equation}
	P_\vsol (\op(\cham)\vsol) = p_2 \vsol,
	\end{equation}
	where $p_2$ is the quadratic polynomial
	\begin{align}
	& p_2:\R^{\dim}\to\C,\\
	& p_2(x) = \beta + b^T(x-\pscone) + \frac12 (x-\pscone)^T B(x-\pscone)
	\end{align} 
	given by the complex coefficients
	\begin{align}
	\beta &= \langle \cham\rangle_\vsol -\frac{\scp}{4}\, \tr(B \ \IC^{-1}),\\
	b &= \begin{pmatrix}\Id & \wm\end{pmatrix} \langle \nabla \cham\rangle_\vsol\in\C^d,\\
	B &= \begin{pmatrix}\Id & \wm \end{pmatrix}\langle \nabla^2 \cham\rangle_\vsol 
	\begin{pmatrix}\Id\\ \wm\end{pmatrix}\in\C^{\dim\times \dim}. 
	\end{align}
	The notation $\langle \clobs\rangle_\vsol = \langle \vsol\mid\op(\clobs)\vsol\rangle$ refers to the expectation value of a quantized smooth observable $\clobs:\R^{2\dim}\to\R^L$ with respect to the Gaussian state $\vsol$.
\end{prop}

\begin{proof}
	We use the Hagedorn wave packets $\{\varphi_n\}_{|n|\le 2}$ associated with the Gaussian wave packet $u$ as an orthonormal basis of the tangent space $\mathcal T_u\mathcal M$, see \eqref{eq:onb-all} and \eqref{eq:onb-ladder-all}, and write the orthogonal projection as
	\[
	P_u (\op(h)u) = \sum_{|n|\le 2} \langle\varphi_n\mid\op(h)u\rangle\ \varphi_n.
	\]  
	Starting with the contribution for $n=0$, we have
	\begin{align}
	\langle\varphi_0\mid\op(h)u\rangle &= \langle u\mid\op(h)u\rangle =\langle h\rangle_u.
	\end{align}
	For the following, it will be useful to introduce the slim rectangular matrix 
	$Z = (Q;P)\in\C^{2d\times d}$ with column vectors $Z_1,\ldots,Z_d\in\C^{2d}$ and to write the ladder operators more compactly as
	\begin{align}
		\mathcal A^\dagger &= \frac{\ii}{\sqrt{2\scp}}\ Z^* J\op(\pstup-\psc) ,\qquad
		\mathcal A = -\frac{\ii}{\sqrt{2\scp}}\ Z^TJ\op(\pstup-\psc).
\end{align}
	For $n=e_j$ we have by \eqref{eq:onb-all}, \cite[Lemmas~4.1, and~4.2]{LasL20} that
	\begin{align}
	\langle\varphi_{e_j}\mid\op(h)u\rangle &= 
	\langle u\mid \mathcal A_j\op(h)u\rangle 
	=\langle u\mid [\mathcal A_j,\op(h)]u\rangle.
	\end{align}
	Since the symbol of $\mathcal A_j$ is linear, we can use pseudodifferential calculus without remainders and obtain that the commutator satisfies
%	\begin{equation}\label{eq:com}
%		\begin{aligned}
%	[\mathcal A_j,\op(h)] &= -\frac{\ii}{\sqrt{2\scp}} \  \left[\op(Z_j^TJ(\pstup-\psc)),\op(\cham)\right]
%	\\
%	&
%	= -\frac{\ii}{\sqrt{2\scp}} \frac{\scp}{\ii}  \op(\{Z_j^TJ(\pstup-\psc),\cham\})\\ 
%	&= \sqrt{\frac{\scp}{2}} \ \op(Z_j^T\nabla h),
%	\end{aligned}
%	\end{equation}
	\begin{align}
	[\mathcal A_j,\op(h)] &= -\frac{\ii}{\sqrt{2\scp}} \  \left[\op(Z_j^TJ(\pstup-\psc)),\op(\cham)\right]
	\\
	&
	= -\frac{\ii}{\sqrt{2\scp}} \frac{\scp}{\ii}  \op(\{Z_j^TJ(\pstup-\psc),\cham\})\\ 
	&= \sqrt{\frac{\scp}{2}} \ \op(Z_j^T\nabla h), \label{eq:com}
	\end{align}
	where we have calculated the Poisson bracket according to 
	\begin{align}
	\{Z_j^TJ z,h\} &= \nabla (Z_j^TJz) \cdot J\nabla h = -Z_j^T\nabla h.
	\end{align}
	Therefore, 
	\[
	\langle\varphi_{e_j}\mid\op(h)u\rangle = \sqrt{\frac{\scp}{2}} Z_j^T\langle\nabla h\rangle_u.
	\]
	After summation, we therefore obtain that
	\begin{align}
	\sum_{j=1}^d \langle\varphi_{e_j}\mid\op(h)u\rangle \,\varphi_{e_j} &= 
	\sum_{j=1}^d \langle\nabla h\rangle_u^T Z e_j e_j^T Q^{-1}(x-\pscone)u\\
	&=\langle\nabla h\rangle_u^T Z Q^{-1}(x-\pscone)u\\
	&=\langle\nabla h\rangle_u^T \begin{pmatrix}\Id\\ \wm\end{pmatrix}(x-\pscone)u,
	\end{align}
	which concludes the computation of the first order contributions.
	For the second order wave packets, we analogously compute the projection coefficient as
	\begin{align}
	\langle \varphi_{e_j+e_k}\mid\op(h)u\rangle 
%	&= \frac{1}{\sqrt{\delta_{kj}+1}} \langle \mathcal A_k^\dagger \mathcal A_j^\dagger u\mid\op(h)u\rangle\\
%	&= \frac{1}{\sqrt{\delta_{kj}+1}} \langle \mathcal A_j^\dagger u\mid[\mathcal A_k,\op(h)]u\rangle\\
	&= \frac{1}{\sqrt{\delta_{kj}+1}} \langle u\mid[\mathcal A_j,[\mathcal A_k,\op(h)]]u\rangle.
	\end{align}
	Using \eqref{eq:com} twice, we obtain that the double commutator satisfies
	\begin{align}
	\left[\mathcal A_j,[\mathcal A_k,\op(h)]\right] 
	&= \sqrt{\frac{\scp}{2}}\left[\mathcal A_j,\op(Z_k^T\nabla h)\right]
%	\\
%	&= \frac{\scp}{2}\,\op(Z_j^T\nabla(Z_k^T\nabla h))\\
%	&
	= \frac{\scp}{2}\,\op(Z_j^T\nabla^2 h Z_k).
	\end{align} 
	This implies for the coefficient that
	\begin{equation}
	\langle \varphi_{e_j+e_k}\mid\op(h)u\rangle  =  \frac{\scp}{2\sqrt{\delta_{kj}+1}}\,Z_j^T\langle\nabla^2 h\rangle_u Z_k.
	\end{equation}
	We now calculate the sum of all the second order contributions. We have 
	\begin{align}
%	&
	\sum_{|n|=2} \langle \varphi_n\mid\op(h)u\rangle \varphi_n 
%	\\
	&= \sum_{j=1}^d \sum_{k=1}^j \langle \varphi_{e_j+e_k}\mid\op(h)u\rangle \varphi_{e_j+e_k}\\
	&= \sum_{j,k=1}^d \frac{\scp}{2\sqrt{2}}\,Z_j^T\langle\nabla^2 h\rangle_u Z_k \ \frac{1}{\sqrt2} \mathcal A_j^\dagger \mathcal A_k^\dagger u,
	\end{align}
	where the complete summation over the full square of indices is compensated by a change in 
	normalisation of the contributions for $j\neq k$. For the part of the sum that generates a constant prefactor for the Gaussian, we have
	\begin{align}
	-\frac{\eps}{4}\sum_{j,k=1}^d \,Z_j^T\langle\nabla^2 h\rangle_u Z_k \ (Q^*Q^{-T})_{j,k}
	&=-\frac{\eps}{4}\tr( Q^*Q^{-T} Z^T \langle\nabla^2 h\rangle_u Z)\\
	&=-\frac{\eps}{4}\tr(\begin{pmatrix}\Id & \mathcal C\end{pmatrix} \langle\nabla^2 h\rangle_u \begin{pmatrix}\Id\\ \mathcal C\end{pmatrix} QQ^*).
	\end{align}
	For the quadratic prefactor, we similarly obtain
	\begin{align}
	&
	\frac{1}{2}\sum_{j,k=1}^d \,Z_j^T\langle\nabla^2 h\rangle_u Z_k \ 
	\left(Q^{-1}(x-\pscone)\right)_j \,\left(Q^{-1}(x-\pscone)\right)_k 
	\\
%	&=\frac12 (x-\pscone)^T Q^{-T} Z^T \langle\nabla^2 h\rangle_u Z Q^{-1} (x-\pscone)\\
	&=\frac12 (x-\pscone)^T \begin{pmatrix}\Id & \mathcal C\end{pmatrix} \langle\nabla^2 h\rangle_u \begin{pmatrix}\Id\\ \mathcal C\end{pmatrix}
	(x-\pscone).
	\end{align}
\end{proof}

% Ende orthogonale Projektion Caroline
Let $\cham:\R\times\R^{2\dim}\to\R$ be continuous with respect to time $t\in\R$, and smooth, and of subquadratic growth in the sense of \eqref{eq:general-ham-subquadr}. Denote $H(t) = \op(\cham(t))$. Then, the time-dependent Schr\"odinger equation
\[
i\scp\partial_t\sol(t) = H(t)\sol(t),\quad \psi(0) = \psi_0
\]
has a unique solution $\sol(t) = U(t,0)\psi_0$ for all times $t\in\R$ for all square integrable initial data $\psi_0\in L^2(\R^\dim)$, see \cite{MasR17} or \cite[Def.~1]{RobC21_book}. The corresponding variational Gaussian wave packet obeys the following equations of motion.

\begin{theorem}[Equations of motion for a general Hamiltonian]\label{theo:eqmo-general-ham}
Let $\vsol_0\in\Mf$ satisfy \eqref{eq:gauss_initialdata} and  be given by its parameters $\pos_0, \mom_0, \wm_0, \pha_0$ defined in \eqref{eq:gwp}. Then, the parameters of the variational approximation 
	\begin{equation}
	\ii\scp \partial_t \vsol(t) = P_{\vsol(t)} (H(t)\vsol(t)),\quad u(0) = u_0
	\end{equation}
	satisfy the following set of ordinary differential equations \eqref{eq:eqmo-general}
%\begin{subequations}\label{eq:eqmo-general}
%\begin{align+}
%\dot{\pos} &=\langle \nabla_\mom \cham\rangle_\vsol,\label{eq:eqmo-general-q}\\
%\dot{\mom} &= -\langle \nabla_\pos \cham\rangle_\vsol, \label{eq:eqmo-general-p}\\
%\dot{\wm} &=-\langle \nabla_{\pos\pos} \cham\rangle_\vsol - \langle\nabla_{\pos\mom} \cham\rangle_\vsol \wm - \wm \langle\nabla_{\mom\pos} \cham \rangle_\vsol -\wm \langle\nabla_{\mom\mom} \cham\rangle_\vsol \wm, \\
%\dot{\pha} &= -\langle \cham\rangle_\vsol +\frac{\scp}{4}\, \tr(B \ \IC^{-1}) + \mom^T \langle\nabla_{\mom} \cham\rangle_\vsol\label{eq:eqmo-general-pha}
%\end{align+}
%\end{subequations}
subject to initial data $(\pos(0),\mom(0), \wm(0), \pha(0)) = (\pos_0, \mom_0, \wm_0, \pha_0)$, where $\cham$ is now the given general subquadratic classical Hamiltonian function.
The Hagedorn parameter matrices of the variational wave packet 
satisfy:
	\begin{subequations}%\label{eq:hag-eqmo-general}
		\begin{align}
			\dot{\fcP} &= -\langle \nabla_{\pos\pos} \cham\rangle_\vsol \fcQ -  \langle\nabla_{\pos\mom} \cham\rangle_\vsol\fcP, \qquad
%			\label{eq:hag-eqmo-general-P}\\
			\dot{\fcQ} =   \langle\nabla_{\mom\pos} \cham \rangle_\vsol \fcQ + \langle\nabla_{\mom\mom} \cham\rangle_\vsol \fcP. \label{eq:hag-eqmo-general-Q}
		\end{align}
	\end{subequations}
Moreover, the matrix factors $\fcQ$, $\fcP$ are symplectic, provided that the initial matrices $\fcQ_0$, $\fcP_0$ of the factorization $\wm_0 = \fcP_0 \fcQ_0^{-1}$ are symplectic.
\end{theorem}

\begin{proof}
We again use \eqref{eq:formula-dt-vsol} and \Cref{prop:orthogonal-proj} and compare the coefficients with respect to the spatial variable $x$. We have one degree of freedom and set, inspired by  \eqref{eq:eqmo_q},
\begin{equation}
\dot{\pos} =\langle \nabla_\mom \cham\rangle_\vsol.
\end{equation}
Now, the claim follows by a direct calculation.
\end{proof}

The equations of motion given in \Cref{theo:eqmo-general-ham} are indeed a generalization of the magnetic ones derived in \Cref{thm:eqmo} as we verify next.

%\begin{prop}[Hagedorn equations of motion for a general Hamiltonian]\label{prop:hag-eqmo-general-ham}
%	Let $\cham:\R^{2\dim}\to\R$ be smooth and of subquadratic growth. Let $\vsol\in\Mf$ be a Gaussian wave packet 
%	of unit norm, $\|\vsol\| = 1$, with phase space center $z=(\pos,\mom)\in\R^{\dim}\times\R^\dim$. Then, the  of the variational approximation 
%	\begin{equation}
%	\ii\scp \partial_t \vsol = P_\vsol (\op(\cham)\vsol)
%	\end{equation}
%	satisfy the following set of ordinary differential equations:
%	\begin{subequations}\label{eq:hag-eqmo-general}
%		\begin{align+}
%%
%			\dot{\fcP} &= -\langle \nabla_{\pos\pos} \cham\rangle_\vsol \fcQ - \langle\nabla_{\pos\mom} \cham\rangle_\vsol\fcP, \label{eq:hag-eqmo-general-P}\\
%			\dot{\fcQ} &=   \langle\nabla_{\mom\pos} \cham \rangle_\vsol \fcQ + \langle\nabla_{\mom\mom} \cham\rangle_\vsol \fcP, \label{eq:hag-eqmo-general-Q}
%%
%		\end{align+}
%	\end{subequations}
%and $\pos, \mom, \pha$ satisfy the equations  \eqref{eq:eqmo-general-q}, \eqref{eq:eqmo-general-p},  \eqref{eq:eqmo-general-pha}, respectively.
%%
%Moreover, the factors $\fcQ$, $\fcP$ are symplectic, provided that the initial matrices $\fcQ_0$, $\fcP_0$ are symplectic.
%\end{prop}

\begin{corollary} In the special space of the magnetic Hamiltonian given in \eqref{eq:clham} we rediscover the equations of motion \eqref{eq:eqmo}.
Moreover, if $\scp \to 0$ and averages tend to point evaluations at the center point $\pos$, then the equations \eqref{eq:eqmo-general-q} and \eqref{eq:eqmo-general-p} tend to classical equations of motion for a general classical Hamiltonian function $\cham$.
\end{corollary}

\begin{proof}
	We have that
	\begin{align}
	\langle \nabla_\mom \cham\rangle_\vsol &= \mom -\langle \mgPot\rangle_\vsol, \quad\text{ and }\quad
	-\langle \nabla_\pos \cham\rangle_\vsol 
%	&
	= -\ii\scp\langle J_A^T\nabla \rangle_\vsol - \langle \nabla \mPot \rangle_\vsol.
	\end{align}
Furthermore, it is
\begin{equation}
\nabla^2 \cham(\cdot,\pos, \mom) =
\begin{pmatrix}
 \nabla^2 \mPot(\cdot,\pos) - D^2_{\mgPot(\cdot,\pos),\mom}  & -J_A^T\\
-J_A & \Id
\end{pmatrix}
,
\end{equation}
such that the trace part appearing in \eqref{eq:eqmo-general-pha} contains the terms
\begin{alignat}{7}
-\langle \nabla_{\pos\pos} \cham\rangle_\vsol ~&=&&~ \langle D^2_{\mgPot,-\ii\scp \nabla}\rangle_\vsol -\langle  \nabla^2 \mPot\rangle_\vsol,& \qquad
- \langle\nabla_{\pos\mom} \cham\rangle_\vsol \wm ~&=&&~ \langle J_A^T\rangle_\vsol \wm,\\
- \wm \langle\nabla_{\mom\pos} \cham \rangle_\vsol ~&=&&~  \wm\langle J_A\rangle_\vsol,&
\qquad
-\wm \langle\nabla_{\mom\mom} \cham\rangle_\vsol \wm ~&=&&~ -\wm^2.
\end{alignat}
%\begin{subequations}\label{eq:av-hessian-general-ham}
%\begin{align+}
%-\langle \nabla_{\pos\pos} \cham\rangle_\vsol &= \langle D^2_{\mgPot,-\ii\scp \nabla}\rangle_\vsol -\langle  \nabla^2 \mPot\rangle_\vsol,\\
%%
%- \langle\nabla_{\pos\mom} \cham\rangle_\vsol \wm &= \langle J_A^T\rangle_\vsol \wm,\\
%%
%- \wm \langle\nabla_{\mom\pos} \cham \rangle_\vsol &=  \wm\langle J_A\rangle_\vsol,\\
%%
%-\wm \langle\nabla_{\mom\mom} \cham\rangle_\vsol \wm &= -\wm^2.
%\end{align+}
%\end{subequations}
%
For the scalar contribution of the projection \Cref{prop:orthogonal-proj} we observe by \eqref{eq:derivative-gwp-all},
\begin{align}
\langle \cham \rangle_\vsol &= -\frac{\scp^2}{2} \langle \Delta \rangle_\vsol +\ii\scp\langle \mgPot\cdot  \nabla \rangle_\vsol +  \langle \mPot \rangle_\vsol\\
%
%&= \frac12\abs{\mom}^2 +\frac{\eps}{4}\tr\Big((\RC^2 + \IC^2)\ICinv\Big)
%-\langle \mgPot^T \big((\RC+ \ii\IC)((x-\pos)+\mom)\big) \rangle_\vsol+  \langle \mPot \rangle_\vsol\\
&=\frac12\abs{\mom}^2 +\frac{\eps}{4}\tr\Big((\RC^2 + \IC^2)\ICinv\Big) 
-\langle \mgPot^T \big(\RC(x-\pos)+\mom\big) \rangle_\vsol+  \langle \mPot \rangle_\vsol.
\end{align}
Finally, we calculate the following trace, appearing in \eqref{eq:eqmo-general-pha}, as
\begin{align}
-\tr\Big((\RC^2 + \IC^2)\ICinv\Big) + \tr\Big(\wm^2 \ICinv\Big) 
%&= \tr\Big(\big(\RC^2 - \IC^2 + \ii(\IC \RC + \RC\IC)\big)\ICinv\Big)\\
&= -2\,\tr(\IC) + 2\ii \,\tr(\RC)
%\\
%&
= 2\ii\, \tr (\wm),
\end{align}
such that, together with
\begin{align}
 \mom^T \langle\nabla_{\mom} \cham\rangle_\vsol &= \abs{\mom}^2 - \mom^T \langle \mgPot \rangle_\vsol,
\end{align}
%and the identities in \eqref{eq:av-hessian-general-ham} of $B$, 
we obtain the differential equation \eqref{eq:eqmo_pha}.
\end{proof}

%\dummyp{\Cref{thm:eqmo}}{
%We can use \Cref{lem:proj_pot_gwp} in \eqref{eq:var} by observing that derivatives of a Gaussian wave packet with respect to $x$ turn into a potential times $\vsol$. In particular, we have
%\begin{align}
%\ii\scp \mgPot \cdot \nabla \vsol &= -\mgPot\cdot\bigl(\wm (x-\pos)+\mom\bigr) \vsol,\\
%-\frac{\scp^2}{2}\Delta \vsol &= \left(\frac{1}{2}(x-\pos)^T\wm^2(x-\pos) + \mom^T\wm (x-\pos)+\frac{1}{2}|\mom|^2 - \frac{\ii\scp}{2}\tr(\wm)  \right) \vsol. 
%\end{align}
%\begin{equation}
%\dot{u}(x,\cdot) = \frac{\ii}{\varepsilon}\left(\frac{1}{2}(x-q)^T\dot{C}(x-q) - \dot{q}^TC(x-q) + \dot{p}^T(x-q) - p^T \dot{q} +\dot{\zeta}\right) u 
%\end{equation}
%Comparing both sides of \eqref{eq:var} and integrating by parts gives the equations of motion in \eqref{eq:eqmo}.
%\qed
%}
%\begin{align}
%\dot{\pos} &= \mom - \langle \mgPot\rangle_\vsol,\\
%\dot{\mom} &= \l\langle J_\mgPot^T \RC (x-\pos)\r\rangle_\vsol + \langle J_\mgPot \rangle_\vsol^T \mom - \langle \nabla \mPot \rangle_\vsol, \\
%\dot{\wm}  &= - \wm^2 + \langle D^2_{\mgPot,\RC(x-\pos)} \rangle_\vsol + \langle  D^2_{\mgPot,\mom}\rangle_\vsol + \langle J_\mgPot\rangle_\vsol^T \wm + \wm \langle J_\mgPot\rangle_\vsol - \langle \nabla^2 \mPot \rangle_\vsol,
%\end{align}

%% file: L2_eb_sublinear.tex
\section{\texorpdfstring{$L^2$-error bound: proof of
		\Cref{lem:bddparam} and \Cref{thm:L2}}{L2-error bound}}\label{sec:L2}

This section is devoted to the \wellposedness of the equations of motion \eqref{eq:eqmo}
and the approximation quality of the variational solution in the $L^2$-norm.

We first state
%The estimate of
the following lemma which will be used frequently to obtain error bound \wrt $\scp$. 
We recall that the lower bound on the eigenvalues of $\IC$ was denoted by $\rho >0$.
%
%We denote by $\lambda_1, \ldots, \lambda_\dim > 0$ the
%(time-dependent) eigenvalues of $\IC$
%{\color{red} which satisfy $\lambda_i \geq \rho$. Brauchen die einen Namen, tauchen danach nicht mehr auf. }

\begin{lemma}[{\cite[Lemma~3.8]{LasL20}}]\label{lem:moments}
%\begin{lemma}[{moments \cite[Lemma~3.8]{LasL20}}]\label{lem:moments}
%	Let ${\lambda_1, \ldots, \lambda_\dim > 0}$ be bounded from below by $\rho > 0$ uniformly on $[0,\etime]$. Then f
	For any $m \geq 0$ there exists a constant $c_{m}$ such that for all $\scp > 0$ it holds
	\begin{align}
	&(\pi\scp)^{-\frac{\dim}{4}}\det(\IC)^\frac{1}{4}\Bigl(\int{|x|^{2m} \exp\bigr(-\frac{1}{\scp}x^T \IC x\bigr)}\,\dx\Bigr)^{\frac{1}{2}}\leq c_{m}\Bigl(\frac{\scp}{\rho}\Bigr)^\frac{m}{2} ,
	\end{align}
	where $c_m$ is independent of $\scp$ and $\rho$.
\end{lemma} 

%{\color{red}Proof of \Cref{lem:bddparam} here}

We now prove the \wellposedness result for \eqref{eq:eqmo} and show the boundedness of the parameters solving \eqref{eq:eqmo}. 

\begin{proof}[Proof of \Cref{lem:bddparam}]
	We show that the right-hand side of \eqref{eq:eqmo} satisfies a local Lipschitz condition with Lipschitz constant independent of $\scp$. To this end it is sufficient if the derivatives \wrt parameters $\pos, \mom, \RC, \IC,\pha$ are bounded on a bounded domain. Then, we obtain a local solution and, as in \Cref{subsec:eqmo}, we can show that there is no \blowup.
	
	%$\refPot$ be a real smooth potential and
	The potentials in the averages of the equations of motion in \eqref{eq:eqmo} do not depend on $\scp$. 
	However, we need to carefully treat
%	Some care needs to be taken for 
	the absolute values of the \GWP, since they contain $\scp$ in the denominator.
	By the chain rule, it is sufficient to {first} calculate the derivatives of averages of some arbitrary potential $\refPot$, which is independent of the parameters. Then, the average has the form
	\begin{align}
	\langle \refPot\rangle_\vsol 
%	&= \int \refPot(x)|\vsol(x)|^2\,\dd x\\
	&= \frac{\sqrt{\det(\IC)}}{(\pi \scp)^\frac{\dim}{2}} \int \refPot(x) \exp\l(-\frac{1}{\scp}(x-\pos)^T \IC (x-\pos)\r) \,\dd x,
	\end{align}
	from which we see that, in this case, the average only depends on $\pos$ and $\IC$. Let $\vsol$ be a \GWP with $\normLtwo{\vsol} =1$. By \eqref{eq:nabla_abs_u} we obtain 
	\begin{align}
	 & \frac{\sqrt{\det(\IC)}}{(\pi \scp)^\frac{\dim}{2}} \ \partial_{\pos} \exp\l(-\frac{1}{\scp}(x-\pos)^T \IC (x-\pos)\r) \\
	 &=  \frac{\sqrt{\det(\IC)}}{(\pi \scp)^\frac{\dim}{2}} \ \frac{2}{\scp} \IC(x-\pos) \exp\l(-\frac{1}{\scp}(x-\pos)^T \IC (x-\pos)\r)\\
	&=- \nabla |\vsol(x)|^2,
	\end{align} 
	thus, using integration by parts, the derivative of the average \wrt to $\pos$ is given by% F\"ur die Parameter $q$ und $\IC$ eines Gau\ss'schen Wellenpaketes mit $\|u\|=1$ gilt
	\begin{align}%\label{Lemma_Abl_nach_q}
	\partial_\pos \langle \refPot(x) \rangle_\vsol &= \langle \nabla \refPot(x) \rangle_\vsol.
	\end{align}

	We continue with derivatives \wrt $\IC$. For a differentiable matrix function $F:\R^{\dim\times\dim}\rightarrow\R$ and a general invertible, symmetric matrix $\matrix = (m_{ij})_{i,j=1,\ldots,\dim}$ we define the componentwise derivation matrix
	\begin{equation}
	\partial_\matrix F(\matrix) \coloneqq (\partial_{m_{ij}} F(\matrix))_{i,j=1,\ldots,\dim} \in \R^{\dim\times\dim}.
	\end{equation}
	By \cite[Part 0.8.10]{HorJ13} we have
	\begin{align}
	\partial_{\matrix} a^T \matrix b = a b^T \quad \text{ and }\quad\partial_{\matrix} \det(\matrix)  ={\det(\matrix)} \matrix^{-1}
	\end{align} %= \mathrm{adj}(\matrix)
	%where $\mathrm{adj}(\matrix)$ denotes the adjugate matrix of $\matrix$.
	and consequently, 
	\begin{align}
	\partial_{\matrix} \sqrt{\det(\matrix)} = \frac{1}{2} \sqrt{\det(\matrix)} M^{-1}.
	\end{align}
	Hence, it follows that
	\begin{equation}
	\partial_\IC\exp\left(-\frac{1}{\scp}(x-\pos)^T\IC(x-\pos)\right) = -\frac{1}{\scp}(x-\pos)(x-\pos)^T \exp\left(-\frac{1}{\scp}(x-\pos)^T\IC(x-\pos)\right)
	\end{equation}
	and 
	\begin{align}
	\partial_{\IC} \langle \refPot(x) \rangle_\vsol &= -\frac{1}{\scp} \bigl\langle (x-\pos)(x-\pos)^T \refPot(x) \bigr\rangle_\vsol
	 + 
	 \frac{1}{2}I^{-1}\bigl\langle  \refPot(x) \bigr\rangle_\vsol.
	\end{align}
	By \Cref{lem:moments} we have $\bigl|\bigl\langle (x-\pos)(x-\pos)^T \refPot(x) \bigr\rangle_\vsol\bigr| \leq C\scp$ for parameters on a bounded domain.
	
	For potentials depending on the parameters, we use again that we are on a bounded domain and that the dependence on $\scp$ of the potentials in \eqref{eq:eqmo} is such that $\scp$ does not enter the denominator.
	%\todoin{Warum sind wir hier fertig?}
%	\qed
\end{proof}

We now turn to the $L^2$-error bound
and adapt the
 proof of \cite[Theorem~3.5]{LasL20} to the magnetic case
 and note that the multiplication potential $\mPot$ is already covered. In order to demonstrate the dependence of the constant in the error bound, we carry out the proof for the advection term.
% and note that  
% is covered by \cite[Lemma~3.5]{LasL20}.
%
%\todoin{Vielleicht noch einen Satz zu $\mPot$?}
%
\begin{proof}[Proof of \Cref{thm:L2}]
	From the proof of \cite[Theorem~3.5]{LasL20} we know that
	\begin{equation}
	\normLtwobig{\sol(t)-\vsol(t)} \leq \int_0^t \frac{1}{\scp}\,\normLtwobig{ Hu - P_{u}(Hu)}\,\dd s.
	\end{equation}
\smallskip 
\noindent (a)
	We write the action of the magnetic Schr\"odinger operator $H$ on a Gaussian $u$ with width $\wm$ and phase space center $(\pos,\mom)$ as
	\[
	H\vsol = -\frac{\scp^2}{2}\Delta \vsol + Y \vsol + \mPot \vsol
	\]
    with 
	\begin{equation}\label{eq:advPot}
	\advPot \coloneqq -\mgPot \cdot (\wm (x-\pos) + \mom).
	\end{equation}
	We perform a second order Taylor expansion of the potentials $\advPot$ and $\mPot$ around 
	the point $\pos$ and denote by $\advPotTaylor$ and $\PotTaylor$ the respective remainders. Then,
	\begin{align}
%	H\vsol-P_\vsol(H\vsol) &= 
    (\Id - P_\vsol)(H\vsol) = (\Id-P_\vsol)(\advPotTaylor\vsol + \PotTaylor\vsol)
	\end{align}
	and
	\begin{equation}
	\normLtwobig{\sol(t)-\vsol(t)} \leq \int_0^t{\frac{1}{\scp}\,\normLtwobig{ \advPotTaylor \vsol + \PotTaylor \vsol}}\,\dd s.
	\end{equation}
	%
%	\begin{equation}
%	\|\sol(t)-\vsol(t)\| \leq \int_0^t{\frac{1}{\scp}\,\l\|\Pr^\perp\l(\advPot\vsol+ {\frac{1}{2}}|\mgPot|^2 u + \mltPot\vsol\r)\r\|}\,\dd s.
%	\end{equation}
%	For $\Pr^\perp\l(\frac{1}{2}|\mgPot|^2 \vsol + \mltPot \vsol\r)$ we can use \cite[Lemma~3.]{LasL20}, hence we only focus on $\Pr^\perp(\advPot \vsol)$. 
%	\begin{equation}
%	\advPot  = X_q+W_q,
%	\end{equation}
	Since
	\begin{equation}\label{eq:Wq_taylor_remainder}
	\advPotTaylor = \frac{1}{2}\sum_{|\alpha|=3}{(x-\pos)^\alpha\int_0^1{(1-\theta)^2 \partial^\alpha \advPot\bigl(\pos+\theta (x-\pos)\bigr)}\dd \theta},
	\end{equation}
	%
%	Due to \Cref{thm:exact} we obtain the bound
%	\begin{equation}
%	\bigl\|P^\perp_u(\advPot u)\bigr\|  \leq \bigl\|W_q u\bigr\|. %= \mathrm{dist}\l(\advPot u, \Ts\r)
%	\end{equation}
	we bound $\normLtwobig{\advPotTaylor u}$ by finding a bound on ${\partial^\alpha \advPot \bigl(q+\theta (x-q)\bigr)}$, which then leads us to
	\begin{align} 
	|\advPotTaylor(x)|^2 \leq C|x-q|^6.
	\end{align}
	By norm conservation and \Cref{lem:moments} the claim 
	that $\normLtwobig{\advPotTaylor\vsol} = \mathcal O(\scp^{3/2})$ follows.
	%Dann folgt die Behauptung aus der Normerhaltung und \Cref{lem:moments} mit $m=3$ und $n=0$.
	%
	%
	For the third derivative of $\partial_{lmn} \advPot$ where $l,m,n = 1,\ldots,d$, we have
		\begin{equation}\label{eq:thirdderivY}
	\begin{aligned}
	\partial_{lmn} \advPot  &= (\partial_{lmn} \mgPot)^T \wm({x}-\pos)+ (\partial_{lmn} \mgPot)^T \mom \\
	&\quad + \big((\partial_{lm} \mgPot)^T \wm\big)_n +\big((\partial_{ln} \mgPot)^T \wm\big)_m + \big((\partial_{mn} \mgPot)^T \wm\big)_l,
	%\\
%	&\quad + \sum_{j=1}^d{\partial_{lm} \mgPot_j \wm_{jn}} + \sum_{j=1}^d{\partial_{ln} \mgPot_j \wm_{jm}} + \sum_{j=1}^d{\partial_{mn} \mgPot_j \wm_{jl}}. 
	\end{aligned}
	\end{equation}
%	\begin{equation}\label{eq:thirdderivY}
%	\begin{aligned}
%	\partial_{lmn} \advPot  &= \sum_{j,k=1}^d{\partial_{lmn} \mgPot_j \bigl(\wm_{jk}({x}_k-\pos_k)+\mom_j\bigr)} \\
%	&\quad + \sum_{j=1}^d{\partial_{lm} \mgPot_j \wm_{jn}} + \sum_{j=1}^d{\partial_{ln} \mgPot_j \wm_{jm}} + \sum_{j=1}^d{\partial_{mn} \mgPot_j \wm_{jl}}. 
%	\end{aligned}
%	\end{equation}
	where $\partial_{lmn} \mgPot$ is meant component wise. The term $x-\pos$ in \eqref{eq:thirdderivY} evaluated at $x = \pos+\theta(x-\pos)$ has the form 
	\begin{equation}
	 \theta (\partial_{lmn} \mgPot)^T \wm({x}-\pos).
%	\sum_{j,k=1}^d{\partial_{lmn} \mgPot_j \wm_{jk}\theta(x_k-\pos_k)}.
	\end{equation}
	By \Cref{lem:moments} we gain additional orders of $\scp$, and we thus neglect the first summand in \eqref{eq:thirdderivY}. The remaining terms are bounded again using \Cref{lem:moments}.
	
	\smallskip 
	\noindent (b) In the general subquadratic case, we use that the action of a semiclassical pseudodifferential operator on a Gaussian wave packet can be approximated by a polynomial prefactor, see 
	\cite[Lemma~14 in \S2.3]{RobC21_book}. For any $\ell \in \N$ there exists a polynomial $\mathcal{Q}_\ell$ of degree $\ell$, such that
	%$\remPot_{\vsol, \ell}$ such that
		\begin{equation}
		Hu  = \mathcal{Q}_\ell \, \vsol + \O(\scp^{(\ell+1)/2}),
		\end{equation} 
		i.e., we have
		\begin{equation}\label{eq:remainder-pot-sq}
		Hu - P_{u}(Hu) = \remPot_{\vsol,\ell}\, \vsol + \O(\scp^{(\ell+1)/2}),
		\end{equation} 
		with a remainder potential $\remPot_{\vsol, \ell}$.
		%		the main contribution of the semiclassical expansion of the orthogonal projection of $\Ham\vsol$ can be written as a subquadratic remainder potential $\remPot_\vsol$ multiplied to $\vsol$. 
		We now fix $\ell=2$ and denote the corresponding cubic remainder potential $\remPot_\vsol =\remPot_{\vsol, 2}$. The proof then works along the lines of the magnetic case.%, which we carry out next.}
\end{proof}

%% file: expectation_values_sublinear.tex
\section{Expectation values: proof of \Cref{lem:norm_energy_mom}}\label{sec:expv}

In this section we adapt the proofs of \cite[Section~3.2]{LasL20} on conservation properties to the \timedependent, magnetic case. Due to \timedependence, the energy will not be a conserved quantity.

%\begin{itemize}
%	\item norm conservation
%	\item Energie (im zeitunabhängigen Fall)
%	\item linearer Impuls, Drehimpuls unter weiterer Voraussetzung
%\end{itemize}
Let $\sol$ be the exact solution of \eqref{eq:sproblem} and $\vsol$ the variational solution \eqref{eq:var} such that \eqref{eq:gauss_initialdata} holds.
%
%
%	\todoin{Hier fehlt nur noch ein einleitender Satz, der den obendran ersetzt.}
%
%
%
%\begin{lemma}[linear and angular momentum]\label[lemma]{lem:momentum}
%	Seien $\mltPot: \R^d \times \R \rightarrow \R$ ein reelles und ${\mgPot :\R^d \times \R \rightarrow \R^d }$ ein reelles, vektorwertiges  Potential, sowie $u$ die variationelle L\"osung \eqref{eq:var}. Weiter seien $P$ und $L$ die Operatoren in \Cref{def:momentum_op} .
%	\begin{enumerate}
%		\item 	Falls $\mltPot$ und $\mgPot_j$ f\"ur alle $j =1,\ldots,d$ translationsinvariant ist, dann ist der totale lineare Impuls $\langle P \rangle_{\vsol(t)}$ konstant.
%		\item Falls $\mltPot$ rotationsinvariant ist und $\mgPot(x,\cdot)= \alpha(\cdot) x$ f\"ur ein $\alpha(\cdot) \in \R$, dann ist der totale Drehimpuls $\langle L \rangle_{u(t)}$ konstant.
%		\item Falls $\|u(t)\|_{L^2} =1$ ist, gilt (klassisches Variante). (drin lassen? klassische Sachen werden erst später eingeführt)
%	\end{enumerate}
%\end{lemma}

\begin{proof}[Proof of \Cref{lem:norm_energy_mom}]
	The proof of norm conservation and the energy formula can be done in the same way as in \cite{LasL20}.
	 We only show the conservation of total linear and angular momentum. 
	 
	By \cite[Theorem~1.3]{Lub08_book} or \cite[Lemma~4.1]{FaoL06} it is sufficient to show that $\Ham(t)$ commutes with $P$ and $L$, respectively, for each $t \in [0,\etime]$.
	By \cite{LasL20} it follows that $\linMom \mgPot_{k_j} = 0$ for all $k \in \{1,\ldots,N\}$ and $j\in \{1,2,3\}$. We further calculate
	\begin{align}
	\linMom(\mgPot\cdot \nabla)\psi	&= \sum_{k=1}^N{\sum_{j=1}^3{(\linMom \mgPot_{k_j}) \partial_{k_j}\psi +  \mgPot_{k_j} \linMom \partial_{k_j}\psi}}= (\mgPot\cdot \nabla)\linMom\psi.
	\end{align}
	Furthermore, a tedious calculation shows that $ (\mgPot\cdot \nabla)\angMom\psi = \angMom(\mgPot\cdot \nabla)\psi$ if and only if
	\begin{align} %\label{A_und_L}
	\sum_{l=1}^N{
		\begin{pmatrix}
		\mgPot_{{l_2}}\partial_{l_3} -\mgPot_{{l_3}}\partial_{l_2} \\
		\mgPot_{{l_3}}\partial_{l_1} - \mgPot_{{l_1}}\partial_{l_3}\\
		\mgPot_{{l_1}}\partial_{l_2} - \mgPot_{{l_2}}\partial_{l_1}
		\end{pmatrix}} \psi =
	\sum_{k=1}^N{\sum_{j=1}^3{\sum_{l=1}^N{\begin{pmatrix}
				x_{l_2}(\partial_{l_3}\mgPot_{k_j})\partial_{k_j} -x_{l_3}(\partial_{l_2}\mgPot_{k_j})\partial_{k_j} \\
				x_{l_3}(\partial_{l_1}\mgPot_{k_j})\partial_{k_j}  - x_{l_1}(\partial_{l_3}\mgPot_{k_j})\partial_{k_j} \\
				x_{l_1}(\partial_{l_2}\mgPot_{k_j})\partial_{k_j} -x_{l_2}(\partial_{l_1}\mgPot_{k_j})\partial_{k_j}
				\end{pmatrix}}}}\psi
	\end{align} 
	holds true. This condition 
%	$\eqref{A_und_L}$
	 is fulfilled if
%	\todoin{can we find more?}
	%
	\begin{equation}
	\partial_{l_m}\mgPot_{k_j} = \alpha \delta_{{l_m},{k_j}}\quad\mathrm{und} \quad \mgPot_{{l_n}} = \alpha x_{l_n}
	\end{equation}
	holds true for some $\alpha \in \R$, $j,n,m \in \{1,2,3\}$, and $k,l\in \{1,\ldots,N\}$ and thus, if $\mgPot(\cdot,x)= \alpha(\cdot) x$ holds.
%	\qed
\end{proof}

%% file: obs_errorbound4_sublinear.tex
% !TeX root = magn_gauss_phys_a_sublinear.tex

\section{Error bound for averages of observables: proof of \Cref{thm:obs} }\label{sec:obs}
In this section we give the proof of \Cref{thm:obs}. We proceed in three steps: 
First, we follow \cite[Section~6.7]{LasL20} and establish an integral representation for the error 
that involves a commutator with 
the time-evolved observable. Second, we prove Egorov's theorem for the time-evolution 
of observables in the general context of magnetic Schr\"odinger operators. Third, we derive 
a \semiclassical expansion of averages with respect to Gaussian wave packets. The combination 
of these steps then allows us to prove \Cref{thm:obs}. We  note that the \semiclassical expansion 
of the averages is crucial for improving the observable estimate in \cite[Theorem~3.5]{LasL20}.
This section applies for both the magnetic and the general subquadratic hamiltonian case. For better readability, some arguments will be provided for the magnetic case only, but with natural slight modifications they also apply for the general subquadratic case.

\subsection{Error representation}
We start with a useful a posteriori representation for the observable error. To this end,
let $\evof(t,s)$ be the evolution family given by \Cref{thm:wellposedness} and $\obs$ an  observable. We introduce the notation
\begin{equation}\label{eq:evof_obs}
\hobs(t,s) \coloneqq \evof(s,t)\obs\evof(t,s), \qquad t, s \in\R.
\end{equation}

\begin{lemma}\label[lemma]{lem:obs_bound_aux}
	Let $\sol$ be the solution of \eqref{eq:sproblem} and $\vsol$ the solution of \eqref{eq:var}. If the initial value $\sol_0 = \vsol_0\in \Mf$ is a \GWP with $\normLtwo{\vsol_0}=1$, then the error of the observables takes the form
	\begin{align}\label{eq:obs_bound_aux}
	&\bigl\langle \sol(t)|\obs\sol(t)\bigr\rangle
	-
	\bigl\langle \vsol(t)|\obs \vsol(t)\bigr\rangle 
	\\
	=~& 
	\int_0^t{\frac{1}{\ii\eps}\bigl\langle u(s) \Big| 
	\left( \overline\remPot_{\vsol(s)}\hobs(t,s)-\hobs(t,s)\remPot_{\vsol(s)}\right)\vsol(s)\bigr\rangle}\,\dd s,
	\end{align}
	where the remainder potential $\remPot_\vsol:\R^\dim\to\C$ depends on the Gaussian wave 
	packet~$\vsol$. In the general subquadratic case, it has been previously defined in \eqref{eq:remainder-pot-sq}. 
	In the magnetic Schr\"odinger case, it satisfies
	\begin{equation}\label{eq:remPot}
	\begin{aligned}
	\remPot_\vsol &= \auxPot(\pos) -\langle \auxPot\rangle_\vsol 
	+ \frac{\scp}{4} \tr(\IC^{-1}\langle \nabla^2 \auxPot\rangle_\vsol) 
	+ \left(\nabla \auxPot(\pos) - \langle \nabla \auxPot\rangle_\vsol\right)^T(x-\pos)\\
	&\quad + \frac{1}{2}(x-\pos)^T\left(\nabla^2\auxPot(\pos)-\langle \nabla^2 \auxPot\rangle_\vsol\right)(x-\pos) 
	%&\quad +\frac{1}{3!}\sum_{\abs{k}=3} (x-\pos)^k \big(\partial^k \auxPot(\pos) - \langle\partial^k \auxPot \rangle_\vsol\big)   
	+ R(\auxPot) ,
	\end{aligned} 
	\end{equation}
	with $\auxPot = \advPot + \mPot$ defined in \eqref{eq:advPot} and \eqref{eq:hamop}, respectively, and $R(\auxPot)$ being the remainder potential of the 
	quadratic Taylor expansion of $\auxPot$ around the point~$\pos$.
For the non-magnetic Schr\"odinger case $\mgPot = 0$, we have $\advPot = 0$ and 
$\remPot_\vsol:\R^\dim\to\R$. 
\end{lemma}

\begin{proof}
	Let $\evof(t,s)$ be the evolution family, such that the exact solution of \eqref{eq:sproblem} is given by \eqref{eq:sol}. Using $\sol_0 = \vsol_0$ and $\evof(t,t) = \Id$ we calculate
	\begin{align}
	\langle u(t)|\obs u(t)\rangle &- \langle \sol(t)|\obs \psi(t)\rangle\\*[1ex]
	&= \langle u(t)|\evof(t,t)\obs\evof(t,t) u(t)\rangle - \langle \evof(t,0)u(0)|\obs \evof(t,0)u(0)\rangle 
	\\*[1ex]
	&= \langle u(t)|U(t,t)\obs \evof(t,t) u(t)\rangle - \langle u(0)|U(0,t)\obs \evof(t,0)u(0)\rangle \\*[1ex]
	&= \int_0^t{\frac{\partial}{\partial s}\langle u(s)|\underbrace{\evof(s,t)\obs\evof(t,s)}_{=\hobs(t,s)}u(s)\rangle}\,\dd s.
	\end{align}
	Employing the differential properties of the evolution family, that is, 
	$\ii\scp\partial_t \evof(t,s) = H(t)\evof(t,s)$ and $-\ii\scp\partial_t \evof(s,t) = \evof(s,t)H(t)$, 
	we obtain 
	\begin{align}
	\frac{\partial}{\partial s} \hobs(t,s)&= \frac{1}{\ii\eps}\bigl(\Ham(s)\evof(s,t)\obs\evof(t,s)-\evof(s,t)\obs\evof(t,s)\Ham(s)\bigr)\\
	&= \frac{1}{\ii\eps}\bigl(\Ham(s)\hobs(t,s)-\hobs(t,s)\Ham(s)\bigr).
	\label{eq:Heisenberg-A-tilde}
	\end{align}
	Since the variational evolution satisfies $\ii\scp\partial_t \vsol(t) = P_{\vsol(t)}\Ham(t)\vsol(t)$, we then have
	\begin{align}
	\frac{\partial}{\partial s}\langle u(s)|\hobs(t,s)u(s)\rangle &= 
	\frac{1}{\ii\scp} \left( \langle  (\Id-P_{\vsol(s)})\Ham(s)\vsol(s)| \hobs(t,s) \vsol(s)\rangle \right)\\
	 &\left.- \langle \vsol(s)| \hobs(t,s) (\Id - P_{\vsol(s)})\Ham(s) \vsol(s)\rangle\right).
	\end{align}	
	 We arrive at \eqref{eq:obs_bound_aux}, using that
	 \begin{align}
	 (\Id-P_{\vsol(s)})\Ham(s)\vsol(s) &=
	 X_{\vsol(s)}\vsol(s) - P_{\vsol(s)}(X_{\vsol(s)}\vsol(s))  = W_{u(s)}u(s).
%	  \\
%	 & =:
%	 W_{u(s)}u(s).
	 \end{align}
	 The claimed form of the remainder potential $W_{\vsol(s)}:\R^\dim\to\C$ follows 
	 from \cite[Proposition 3.14]{LasL20}, 
	 since the proof of the projection formula there also applies for 
	 the potential function $X_{\vsol(s)}$ even though it is complex-valued. 
\end{proof}

\subsection{Egorov's theorem}
Further, to prove \Cref{thm:obs} we have to establish a variant of Egorov's theorem, which connects the time-evolved quantum observable~$\hobs(t,s)$, in case it originates from 
a \Weylquantized $\obs = \op(\clobs)$, with the evolution map 
of the classical Hamiltonian system. 
\alt{Since the classical Hamiltonian function $\cham(t)$ is subquadratic, the ordinary differential equation \eqref{eq:cldiff} has a globally Lip\-schitz continuous 
right-hand side, and the Picard--Lindel\"of theorem provides a unique global solution.}
%
%\slin{Recall that we require the existence of a unique solution to the ordinary differential equation \eqref{eq:cldiff} on a fixed time interal $[0,\etime]$. In case of a subquadratic classical Hamiltonian function $\cham(t)$, the Picard--Lindel\"of theorem provides a unique global solution.}
Recall that since $\mgPot$ and $\mPot$ are sublinear and subquadratic, respectively, we obtain a unique global solution to the ordinary differential equation \eqref{eq:cldiff}.
We denote by 
\begin{align}
\flow^{t,s} : \R^{2\dim} \to \R^{2\dim},\quad (\varone_s,\vartwo_s) \mapsto (\varone_s(t), \vartwo_s(t))
\end{align}
the classical propagator, which maps initial values at time $s$ to the solution of \eqref{eq:cldiff} at time $t$.
For any $\pstup=(\varone,\vartwo)$, it satisfies the evolution equation 
\begin{align}\label{eq:reform_cham}
\pt\Phi^{t,s}(\pstup) &= -J(\nabla_\pstup \cham)(t,\Phi^{t,s}(\pstup)),\\
\Phi^{s,s}(\pstup) &= \pstup.
\end{align}
Both in the magnetic and the general subquadratic case, the classical propagator  $\Phi^{t,\tau}$ is a diffeomorphism with inverse $(\Phi^{t,\tau})^{-1} = \Phi^{\tau,t}$.
For time-independent, subquadratic Hamiltonians it is well-established that 
\[
\hobs(t,0) = \op\l(\egor{t}{0}\r) + \mathcal O(\scp^2).
\]
However, to the best of our knowledge, in the literature a proof of the Egorov approximation for the non-autonomous case is not available, and the proofs presented for example in 
\cite{BouR02}, \cite[Chapter~11]{Zwo12_book}, or \cite[Thm.~12]{RobC21_book} assume time-independent or compactly supported Hamiltonians and thus do not cover our more general situation. 
The main difficulties are the \timedependence of the Hamiltonian operator $\Ham(t)$, which prevents energy conservation, 
%and commutation properties, 
and the allowed sublinear growth of the observables.
\begin{prop}[\timedependent Egorov--theorem] \label{prop:timdep_egorov} 
	Let $\obs = \op(\clobs)$ be a quantum observable stemming from a smooth, sublinear classical observable $\clobs$ in the sense of \Cref{def:admissible-symbols}. Further, let $\auxobs:\R\times\R\times\R^{2d}\to\R$, $(t,s,\pstup)\mapsto \auxobs(t,s,\pstup)$ be defined by
		\begin{equation}\label{eq:egorov-obs}
		\auxobs(t,s,\pstup) = \egor{t}{s}(\pstup).
		\end{equation}
		We consider two cases.
%For all $\varphi\in L^2(\R^\dim)$ we then have
	\begin{enumerate}
	\item The Hamiltonian operator stems from a classical, subquadratic  function $\cham$. Then, the observable given in \eqref{eq:egorov-obs} is sublinear and for all $\varphi\in L^2(\R^\dim)$ we have 
	\begin{align}\label{eq:egorov-estimate-h-sq}
	\normLtwobig{\bigl(\hobs(t,s) - \op\l(\auxobs(t,s)\r)\bigr)\varphi} \leq 
	C\, \scp^2\, \e^{C|t-s|}\, \normLtwo{\varphi}
	\end{align}
	for all $s,t\in\R$.
	\item The Hamiltonian operator is a magnetic Schr\"odinger operator. We assume that the observable given in \eqref{eq:egorov-obs} is of time-exponential growth in the following sense. There exists a smooth nonnegative function $\Gamma(t,s)\ge0$ such that for any $\alpha\in\N^{2\dim}$ there exists $C_\alpha>0$ with
	\begin{equation}\label{eq:sublinear_time}
	|\partial^{\alpha}_\pstup\auxobs(t,s,\pstup)| \le C_{\clobs,\alpha}\ \exp(|\alpha|\,\Gamma(t,s))
	\end{equation}
	for all $\pstup\in\R^{2\dim}$ and all $t,s\in\R$. Then, for any $\varphi\in L^2(\R^\dim)$ such that $\op(\pstup)\varphi\in L^2(\R^\dim)$, we then have
		\begin{align}
		\normLtwoBig{\Big(\hobs(t,s)- \op\l(\auxobs(t,s)\r)\varphi\Big)} \le 
		C\, \scp^2\, \e^{C|t-s|}\, \normLtwo{\op(\pstup)\varphi},
		\end{align}
	for all $s,t\in \R$.
%	\begin{align}
%	\normLtwobig{\bigl(\hobs(t,s) - \op\l(\egor{t}{s}\r)\bigr)\vsol} \leq 
%	C\, \scp^2\, \e^{C|t-s|}\, \normLtwo{\vsol}.
%	\end{align}
	\end{enumerate}
%For all $\varphi\in L^2(\R^\dim)$ we then have
%	\begin{align}
%	\normLtwobig{\bigl(\hobs(t,s) - \op\l(\egor{t}{s}\r)\bigr)\varphi} \leq 
%	 C\, \scp^2\, \e^{C|t-s|}\, \normLtwo{\varphi}
%	\end{align}
	%
	\alt{for all $s,t\in\R$.} 
	The constant $C>0$ depends on derivative bounds of the potentials $\mgPot$, $\mltPot$ and the observable $\clobs$, 
	but not on $\scp,t,s$. In particular, $C=0$ for $\mgPot$ linear and $\mltPot$ quadratic. 
	%\slin{In case (b), the constant $C$ additionally depends on parameters of the Gaussian wave packet $\vsol$.}
\end{prop}

\begin{proof}
%Let $\auxobs:\R\times\R\times\R^{2d}\to\R$, $(t,s,\pstup)\mapsto \auxobs(t,s,\pstup)$ 
%be defined by 
%%
%\begin{equation}\label{eq:egorov-obs}
%	\auxobs(t,s,\pstup) = \egor{t}{s}(\pstup) .
%\end{equation}
(1) We start by discussing the growth of the function $\auxobs(t,s,\pstup)$ for case (a). 
For first order derivatives \wrt $(\varone, \vartwo)$ of the classical propagator we have 
\begin{align}
D\flowts{t}{s} &= \Id + J^{-1} \int_s^t \nabla^2 \cham(\tau, \flowts{\tau}{s})\, D \flowts{\tau}{s}\,\dd \tau,
\end{align}
and thus 
\begin{align}
\norm{\infty}{D\flowts{t}{s}} &\le 1 + \int_s^t \sup_{\pstup\in\R^{2\dim}}\|\nabla_\pstup^2 \cham(\tau,\pstup)\| \norm{\infty}{D \flowts{\tau}{s}}\,\dd \tau.
\end{align}
Since the Hamiltonian function $\cham(t,\cdot)$ is subquadratic, we have 
	\begin{equation}
	\Gamma(t,s) \coloneqq \int_s^t \sup_{\pstup\in\R^{2\dim}}\|\nabla_\pstup^2 \cham(\tau,\pstup)\| \dd \tau<\infty,
	\end{equation}
	and by Gronwall's lemma
	\[
	\norm{\infty}{D\Phi^{t,s}}\le  \exp(\Gamma(t,s)). 
	\]
	Moreover, for any $\alpha\in\N^{2\dim}$ with $|\alpha|\ge 1$ there exists a constant $C_\alpha>0$ such that
	\[
	|\partial^{\alpha}_\pstup\Phi^{t,s}(\pstup)| \le C_\alpha \exp(|\alpha|\,\Gamma(t,s))
	\]
	for all $t,s\in\R$ and all $\pstup\in\R^{2\dim}$, see \cite[Lemma~2.2]{BouR02} for 
	a proof that literally applies to the non-autonomous case. Then, the same argument as for 
	\cite[Lemma~2.4]{BouR02} yields that for every $\alpha\in\N^{2\dim}$ with $|\alpha|\ge 1$ there exists a constant $C_{\clobs,\alpha}>0$ such that
	\[
	|\partial^{\alpha}_\pstup\auxobs(t,s,\pstup)| \le C_{\clobs,\alpha}\ \exp(|\alpha|\,\Gamma(t,s))
	\]
	for all $t,s\in\R$ and all $\pstup\in\R^{2\dim}$. In particular, $\auxobs(t,s,\cdot)$ is sublinear.

\smallskip
(2) Next we compare the operators $\op(\auxobs(t,s))$ and $\hobs(t,s)= \evof(s,t)\obs\evof(t,s)$. Since on the diagonal 
$\auxobs(t,t,\cdot) = \clobs$ and $\evof(s,s) = \Id$, we obtain similarly as for \eqref{eq:Heisenberg-A-tilde} 
\begin{align}
&  \hobs(t,s) - \op\l(\auxegor{t}{s}\r)\\
=&\int_s^t \evof(s,\tau) 
\Big(\frac{\ii}{\scp} \big[\Ham(\tau), \op\l(\auxegor{t}{\tau}\r)\big] +\op\l(\partial_\tau\auxegor{t}{\tau}\r) \Big)
\evof(\tau,s)
\,\dd \tau\\
=&\int_s^t \evof(s,\tau) 
\Big(\op(\{\cham(\tau), \auxegor{t}{\tau}\})  +\op\l(\partial_\tau\auxegor{t}{\tau}\r) \Big)
\evof(\tau,s)
\,\dd \tau 
+
\remEgorov(t,s) ,
% \O(\scp^2),
\end{align}
where the last equation relies on the product rule of Weyl quantization \cite[Theorem]{RobC21_book}. Here,
\begin{align}
\{\cham(\tau), \auxegor{t}{\tau}\} &=
% \nabla_\vartwo\cham(\tau)\cdot\nabla_\varone\auxegor{t}{\tau}-
%\nabla_\varone\cham(\tau)\cdot\nabla_\vartwo\auxegor{t}{\tau}
%\\
%&=
 \nabla_\pstup \cham(\tau)\cdot J\nabla_\pstup \auxegor{t}{\tau}
\end{align}
denotes the Poisson bracket of $\cham(\tau)$ and $\auxegor{t}{\tau}$. It remains to show that the integral vanishes and that the remainder $\remEgorov(t,s)$ is of order~$\scp^2$.
\smallskip

(3)  For the estimation of the remainder, we use that
\[
\remEgorov(t,s) = \scp^2 \int_s^t \evof(s,\tau) \op(\bm{r}(t,\tau)) \evof(\tau,s)\,\dd \tau,
\]
where $\bm{r}(t,\tau,\cdot)$ is a smooth function depending on the derivatives of the order $\ge 3$ of 
the function $\cham(\tau,\cdot)$ and of the sublinear $\auxobs(t,\tau,\cdot)$. 
In order to estimate
	\begin{align}
	\normLtwobig{\rho(t,s)\varphi} &\le  \scp^2 \int_s^t 
	\normLtwobig{\op(\bm{r}(t,\tau)) \evof(\tau,s)\varphi}\,\dd \tau,
	\end{align}
we investigate the above integrand.%\\[0.5em] 

$(i)$
%\begin{itemize}
%\item
If $\cham$ is subquadratic, then, due to the estimates given in (a), for all $\alpha\in\N_0^{2d}$ there exist 
$c_{1,\alpha},c_{2,\alpha}>0$ such that
\[
|\partial_\pstup^\alpha \bm{r}(t,\tau,\pstup)| \le c_{1,\alpha}\exp(c_{2,\alpha}|t-\tau|)
\]
for all $t,\tau\in\R$ and $\pstup\in\R^{2\dim}$, and the Calder\'on--Vaillancourt Theorem, see e.g. \cite[Theorem~4]{RobC21_book}, provides the claimed constant $C>0$ for part (a).

%
%\item 
%\\[0.5em] 
%\smallskip
$(ii)$
In the magnetic case, we rewrite the remainder function 
\begin{equation}
\bm{r}(\varone,\vartwo)= \bm{r}(\cdot, \cdot, \varone,\vartwo) = \bm{b}_0(\varone,\vartwo) + \bm{b}(\varone,\vartwo)^T\vartwo,
\end{equation}
where $\bm{b}_0 : \R \times \R \times \R^{2\dim} \to \R$ and $\bm{b} : \R \times \R \times \R^{2\dim} \to \R^\dim$ are bounded with all their derivatives. For the first summand, we proceed as in the subquadratic case, using Calder\'on--Vaillancourt. For the second summand containing an unbounded linearity in $\vartwo$,
% For estimating the remainder part $\bm{r}_2(\varone, \vartwo) = \bm{b}(\varone,\vartwo)^T\vartwo$, 
 we use the product rule and obtain that 
\begin{align}
\op\big(\bm{r}_2(\varone, \vartwo)\big) = -\op\big(\bm{b}(\varone,\vartwo)\big)\cdot \ii \scp \nabla + \O(\scp).
\end{align}
Then, the boundedness of $\bm{b}$ provides $C_{\bm{b}}>0$ such that
\begin{align}
\normLtwobig{\op(\bm{b}(t,\tau)) \scp\nabla \big(\evof(\tau,s)\varphi\big)} \le C_{\bm{b}}\,\normLtwobig{ \scp\nabla \big(\evof(\tau,s)\varphi\big)}.
\end{align}
In the next step, we analyse $\normLtwobig{ \scp\nabla \big(\evof(\tau,s)\varphi\big)}$.
%
%\end{itemize}
\alt{In particular, due to the estimates given in (a), for all $\alpha\in\N_0^{2d}$ there exist 
$c_{1,\alpha},c_{2,\alpha}>0$ such that
\[
|\partial_\pstup^\alpha \bm{r}(t,\tau,\pstup)| \le c_{1,\alpha}\exp(c_{2,\alpha}|t-\tau|)
\]
for all $t,\tau\in\R$ and $\pstup\in\R^{2\dim}$. Therefore, \Cref{thm:wellposedness} in combination with the Calder\'on--Vaillancourt Theorem, see e.g. \cite[Theorem~4]{RobC21_book}, provides the claimed constant $C>0$ such that
\begin{align}
\normLtwobig{\rho(t,s)\varphi} &\le  \scp^2 \int_s^t 
\normLtwobig{\op(\bm{r}(t,\tau)) \evof(\tau,s)\varphi}\,\dd \tau\\
&\le C\, \scp^2\, \e^{C|t-s|}\, \normLtwo{\varphi}.
\end{align}}

\smallskip 
(4) Let $t\ge s$ and set $f(t) =  \op(\pstup) \evof(t,s)\varphi$. We argue as in the proof for 
\cite[Lemma~10.4]{Car21} and observe that $f(t)$ solves the perturbed magnetic Schr\"o\-dinger equation
	\begin{align}
	\ii\scp \pt f(t) &=  \op(\pstup)\Ham(t)\evof(t,s)\varphi =  \Ham(t)f(t) + \delta(t)
	\end{align}
	with source term
	\begin{align}
	\delta(t) &= [\op(\pstup), \Ham(t)]\evof(t,s)\varphi = \frac{\scp}{\ii}\begin{pmatrix}
	\op(\vartwo - \mgPot(\varone))\\
	\op\big(\nabla \mPot(\varone)-\nabla \mgPot(\varone) \cdot\vartwo \big)
	\end{pmatrix}
	\evof(t,s)\varphi,
	\end{align}
	where we used the product rule for the second equation. In the same spirit as in step~(3), we estimate
	\begin{align}
	\normLtwobig{\delta(t)} \le C \scp \normLtwobig{\op(\pstup)\evof(t,s)\varphi} = C\scp\normLtwobig{f(t)},
	\end{align}
	where we exploited the sublinearity of $\mgPot$ and that $\mPot$ is subquadratic. 
	%This means that we have
	%\begin{equation}
	%\normLtwobig{\delta(t)} \le 2\scp \normLtwobig{f(t)}.
	%\end{equation}
	By the variation of constants formula
%	\[
%	\normLtwobig{f(t)} \le \normLtwobig{f(s)} + C \int_s^t \normLtwobig{f(\tau)} \d\tau,
%	\]	
	 followed by Gronwall's lemma, we obtain that
%	\begin{equation}
%	\normLtwobig{f(t)} \le \normLtwobig{f(0)} +2 \int_0^t 	\normLtwobig{f(\sigma)} \, \dd \sigma,
%	\end{equation}
%	such that Gronwall's lemma gives us the estimate
	\begin{equation}
	\normLtwobig{f(t)} \le \e^{C(t-s)}\normLtwobig{f(s)} = \e^{C(t-s)}\normLtwobig{\op(\pstup)\varphi}.
	\end{equation}
%	where $\normLtwobig{f(0)} = \normLtwobig{\op(\pstup)\vsol(0)}$ 
%	is bounded in terms of parameters if the Gaussian wave packet $\vsol(0)$ but does not depend on $\scp$.

\smallskip

(5) In the following step we show that 	$\auxobs$
satisfies the transport equation
\begin{subequations} \label{eq:transport}
\begin{align+}
\partial_\tau \auxegor{t}{\tau} &= -\{\cham(\tau),\auxegor{t}{\tau}\},\\
\auxegor{t}{t} &= \clobs
\end{align+}
\end{subequations}
for $\tau\in[s,t]$. Then the integrand in question indeed vanishes, and we obtain
\begin{equation}
\hobs(t,s) = \op(\auxegor{t}{s}) + \mathcal O(\eps^2),
\end{equation}
as claimed. We rewrite the transport equation \eqref{eq:transport} as 
\begin{align}\label{eq:reform_transport}
\partial_\tau \auxegor{t}{\tau} &= J\nabla_\pstup \cham(\tau) \cdot\nabla_\pstup\auxegor{t}{\tau},\\
\auxegor{t}{t} &= \clobs.
\end{align}
The following argument crucially uses that $\Phi^{t,\tau}$ is a diffeomorphism with inverse $(\Phi^{t,\tau})^{-1} = \Phi^{\tau,t}$. We observe that 
$\auxobs(t,\tau,\flow^{\tau,t}(\pstup)) = \clobs(\pstup)$ for all $\pstup\in\R^{2\dim}$ and calculate
\begin{align}
0 &= \partial_\tau \clobs(\pstup)\\
&= \partial_\tau \auxobs(t,\tau,\flow^{\tau,t}(\pstup))\\
%&=(\partial_\tau \auxobs)(t,\tau,\flow^{\tau,t}(\pstup)) 
%+ \partial_\tau\flow^{\tau,s}(\pstup)\cdot(\nabla_\pstup \auxobs)(t,\tau,\flow^{\tau,t}(\pstup))\\
&=(\partial_\tau \auxobs)(t,\tau,\flow^{\tau,t}(\pstup)) 
-J (\nabla_\pstup \cham)(\tau,\flow^{\tau,t}(\pstup))\cdot(\nabla_\pstup \auxobs)(t,\tau,\flow^{\tau,t}(\pstup)),
%&=\left(\partial_\tau \tilde a -J\nabla_\pstup h\right)(\tau,s,\Phi^{\tau,s}(\pstup)).
\end{align}
where, we used in the last step, the chain rule and  
%that $\flowts{\tau}{t}$ solves 
\eqref{eq:reform_cham}.
Since $\Phi^{\tau,t}$ is a diffeomorphism, this proves that 
$\auxobs(t,\tau)$ indeed solves the transport equation \eqref{eq:reform_transport}.
\end{proof}

%%%%%%%%%%%%%%%%%%%%%%%%%%%%%%%%
\subsection{Averages with respect to Gaussian wave packets}
The a posteriori error representation of \Cref{lem:obs_bound_aux} involves an average with 
respect to the variational solution. By Egorov's theorem, \Cref{prop:timdep_egorov}, the time-evolved quantum observable can be approximated by the \Weylquantized classical observable evolved along the classical flow. We therefore derive an asymptotic expansion of averages of \Weylquantized operators with respect to Gaussian wave packets. For obtaining this expansion, the following phase space moments will be useful.
\begin{lemma}[Gaussian moments]\label{lem:mom_gauss}
We consider a Gaussian $u\in\mathcal M$ of unit norm, $\|\vsol\|=1$, with phase space center $\psc=(\pscone,\psctwo)\in\R^{2\dim}$ and width matrix 
	$\mathcal C\in\C^{\dim\times\dim}$. We denote by 
\begin{equation}\label{eq:rho-ell-wm}
%	\covm = \begin{pmatrix} \IC + \RC \ICinv \RC & -\RC \ICinv\\ \ICinv\RC & \ICinv 
%\end{pmatrix}
%\quad  \text{ and }\quad
\rho_\ell(\wm) =\pi^{-d} 
\int_{\R^{2\dim}} \pstup^\ell \,\exp(-\pstup\cdot \covm \pstup) \, \dd \pstup
\quad  \text{ with }\quad
	\covm = \begin{pmatrix} \IC + \RC \ICinv \RC & -\RC \ICinv\\ \ICinv\RC & \ICinv 
\end{pmatrix},
\end{equation}	
where $G \in \R^{2\dim\times 2\dim}$ is symmetric, positive definite and symplectic. Then, for any multi-index $\ell=(\ell_1,\ldots,\ell_{2\dim})\in\N_0^{2\dim}$, we have
	\begin{equation}
	\langle (\pstup-\psc)^\ell\rangle_\vsol  = \scp^{|\ell|/2} \rho_\ell(\wm).
	\end{equation}
%	\[
%	\langle (\pstup-\psc)^\ell\rangle_\vsol = \scp^{|\ell|/2}\, \pi^{-d} 
%	\int_{\R^{2\dim}} \pstup^\ell \,\exp(-\pstup\cdot \covm \pstup) \, \dd \pstup,
%	\]
%where the matrix 
%	\[
%	\covm = \begin{pmatrix} \IC + \RC \ICinv \RC & -\RC \ICinv\\ \ICinv\RC & \ICinv 
%	\end{pmatrix}\in\R^{2\dim\times 2\dim}
%	\]
%	is symmetric, positive definite and symplectic.
If the length $|\ell|$ of the multi-index is odd, then we have $\langle (\pstup-\psc)^\ell\rangle_\vsol = 0$.
\end{lemma}
\begin{proof}
The claimed representation becomes evident, when using the Wigner function of the 
Gaussian wave packet $\vsol$. The Wigner function of a Gaussian wave packet centered in $\psc$ satisfies
	\[
	\mathcal W_\vsol(\pstup) = (\pi\scp)^{-\dim} \exp(-\tfrac1\scp (\pstup-\psc)\cdot \covm(\pstup-\psc)),
	\] 
	where the matrix $\covm$
	%, that is build of the real and imaginary parts of $\wm$, 
	is symplectic, symmetric, positive definite, 
	see \cite[Proposition~6.15]{LasL20}. The average of any \Weylquantized observable can be written as the phase space integral of the symbol versus the Wigner function, 
	see for example \cite[Theorem~6.5]{LasL20}. In particular, 
	\begin{align}
	\langle (\pstup-\psc)^\ell\rangle_\vsol 
	&= \int_{\R^{2\dim}} (\pstup-\psc)^\ell  \,\mathcal W_u(\pstup) \, \dd \pstup\\\label{eq:gauss_mom}
	&= (\pi\scp)^{-\dim} \int_{\R^{2\dim}} (\pstup-\psc)^\ell  \exp(-\tfrac1\scp (\pstup-\psc)\cdot \covm(\pstup-\psc)) \, \dd \pstup\\
	&=  \pi^{-\dim} \scp^{|\ell|/2} \int_{\R^{2\dim}} \pstup^\ell  \exp(-\pstup\cdot \covm\pstup) \, \dd \pstup, 
	\end{align}
	where we have used that symplecticity implies $\det(\covm) = 1$.
	 We observe, that if the length $|\ell|$ of the multi-index is odd, then the above integral vanishes,  and consequently $\langle (\pstup-\psc)^\ell\rangle_\vsol = 0$ as well. 
\end{proof}

%For any complex symmetric matrix $\wm\in\C^{\dim\times\dim}$ with positive definite imaginary part 
%and any multi-index $\ell\in\N_0^{2\dim}$ we denote
%\[
%\rho_\ell(\wm) =\pi^{-d} 
%	\int_{\R^{2\dim}} \pstup^\ell \,\exp(-\pstup\cdot \covm \pstup) \, \dd \pstup.
%\]
%In this notation, the result of \Cref{lem:mom_gauss} reads
%\[
%\langle (\pstup-\psc)^\ell\rangle_\vsol  = \scp^{|\ell|/2} \rho_\ell(\wm).
%\]
%We also observe, that if the length $|\ell|$ of the multi-index is odd, then $\rho_\ell(\wm)=0$,  and consequently $\langle (\pstup-\psc)^\ell\rangle_\vsol = 0$ as well.
We now use these 
moments for expanding Gaussian averages with respect to general observables. 
\begin{prop}[Gaussian averages]\label{prop:av_gauss}
	We consider a Gaussian $u\in\mathcal M$ of unit norm, $\|\vsol\|=1$, with phase space center $\psc=(\pscone,\psctwo)\in\R^{2\dim}$ and complex width matrix 
	$\mathcal C\in\C^{\dim\times\dim}$. Then, for any smooth function 
	$\clobs:\R^{2\dim}\to\R$ with bounded sixth order derivatives, 
	\[
	\langle \clobs\rangle_\vsol = \clobs(\psc) + \scp f_2(\clobs,\wm) + \scp^2 f_4(\clobs,\wm) + \rho^\scp(\clobs,\wm),
	\]
	where
	\begin{equation}
	f_k(\clobs,\wm) = \sum_{|\ell|=k} \frac{1}{\ell !}\, \partial^\ell\clobs(\psc) \,\rho_\ell(\wm),
	\qquad k=2,4. 
	\end{equation}
	The second order contribution satisfies
	\begin{equation}
	f_2(\clobs,\wm) = \tfrac14 \tr(\nabla^2\clobs(\psc)_\wm \,\IC^{-1})
	\end{equation}
	with
	\begin{equation}\label{eq:hessian-symbol-wm}
	\nabla^2\clobs(\psc)_\wm = 
	\begin{pmatrix}\Id & \wm^*\end{pmatrix} \nabla^2\clobs(\psc) 
	\begin{pmatrix}\Id\\ \wm\end{pmatrix}\in\C^{d\times d}.
	\end{equation}
	The remainder satisfies $|\rho^\scp(\clobs,\wm)|\le C \scp^3$ with a constant $C>0$ that only depends on sixth order derivatives of $\clobs$ as well as on the width matrix $\wm$.
\end{prop}
\begin{proof}
	We start by Taylor expanding the symbol around $\psc$ with sixth order remainder,
	\begin{equation}		
	\clobs(\pstup) = \sum_{|k|\le 5} \frac{1}{k!} \partial^k \clobs(\psc) (\pstup-\psc)^k + \bm r_6(\pstup;\psc),
	\end{equation}
	where
	\begin{equation}
	\bm r_6(\pstup;\psc) = \sum_{|k|=6} r_k(\pstup;\psc) (\pstup-\psc)^k, \qquad  r_k(\pstup;\psc) = \frac{6}{k!}\int_0^1 (1-\vartheta)^5 \partial^k \clobs(\psc + \vartheta(\pstup-\psc)) \ \dd \vartheta.
	\end{equation}
%	\begin{equation}
%	\bm r_6(\pstup;\psc) = 6 \sum_{|k|=6} \frac{1}{k!} (\pstup-\psc)^k \int_0^1 (1-\vartheta)^5 \partial^k \clobs(\psc + \vartheta(\pstup-\psc)) \ \d\vartheta.
%	\end{equation}
We have
\begin{equation}
\langle \bm r_6(\pstup;\psc)\rangle_u = \sum_{|k|=6} \int_{\R^{2\dim}} r_k(\pstup;\psc) \,(\pstup-\psc)^k\, \mathcal W_\vsol(\pstup) \,\dd \pstup.
\end{equation}
	Therefore, using \Cref{lem:mom_gauss}, 
	\begin{align}
	\langle \clobs\rangle_u 
%	&= \clobs(\psc) + \sum_{|k|=1}^5 \frac{1}{k!} \partial^k \clobs(\psc) \langle(\pstup-\psc)^k\rangle_u + \langle \bm r_6(\pstup;\psc)\rangle_u\\
%	&= \clobs(\psc) + \sum_{|k|=2,4} \frac{1}{k!} \partial^k \clobs(\psc) \langle(\pstup-\psc)^k\rangle_u + \langle \bm r_6(\pstup;\psc)\rangle_u\\*[1ex]
	&=\clobs(\psc) + \scp f_2(\clobs,\wm) + \scp^2 f_4(\clobs,\wm) + \langle \bm r_6(\pstup;\psc)\rangle_u,
	\end{align}
%	Writing the remainder as
%	\[
%	\bm r_6(\pstup;\psc) = \sum_{|k|=6} r_k(\pstup;\psc) (\pstup-\psc)^k,
%	\]
%	we have
%	\[
%	\langle \bm r_6(\pstup;\psc)\rangle_u = \sum_{|k|=6} \int_{\R^{2\dim}} r_k(\pstup;\psc) \,(\pstup-\psc)^k\, \mathcal W_\vsol(\pstup) \,\dd \pstup.
%	\]
%	Therefore, 
	and, with the same substitution as in the proof of \Cref{lem:mom_gauss}, we bound
	\begin{align}
	\left|\langle \bm r_6(\pstup;\psc)\rangle_u\right| \le C(\clobs,\wm) \, \scp^3 \quad\text{ with }\quad  C(\clobs,\wm) =
		\sum_{|k|=6} \|r_k(\cdot;\psc)\|_\infty\
		\abs{\rho_k(\wm)}.
%		\pi^{-\dim}
%		 \int_{\R^{2\dim}} |\pstup|^{|k|} \ \exp(-\pstup\cdot \covm \pstup) \ \dd \pstup,
	\end{align}
	The constant $C(\clobs,\wm)>0$  depends on fourth order derivatives of $\clobs$ and on the width matrix $\wm$. It remains to rewrite the second order contribution as
	\begin{align}
	f_2(\clobs,\wm) &= 
	\pi^{-\dim} \sum_{|\ell|=2} \frac{1}{\ell !} \partial^\ell \clobs(\psc)  \int_{\R^{2d}} (G^{-1/2}\pstup)^\ell  
	\exp(-|\pstup|^2) \, \dd \pstup \\
	&= \frac{1}{2}\ \pi^{-\dim} \int_{\R^{2\dim}} \pstup \cdot G^{-1/2}\nabla^2\clobs(\psc)G^{-1/2}\pstup\   
	\exp(-|\pstup|^2) \, \dd \pstup \\*[1ex]
	&= \frac{1}{4}\, \tr(\nabla^2 \clobs(\psc) G^{-1}).
	\end{align}
	Since $G$ is symplectic and symmetric, its inverse satisfies 
	\begin{align}
	G^{-1} &= JGJ^{-1}
%	\\
%	&
	= \begin{pmatrix}  \ICinv & \ICinv\RC \\ \RC\ICinv &  \IC + \RC \ICinv \RC
	\end{pmatrix}.
	\end{align} 
	We decompose the Hessian $\nabla^2 \clobs(\psc)$ in block form as
	\begin{equation}\label{eq:hess-block-form}
	\nabla^2\clobs(\psc) = \begin{pmatrix}A & B\\ B^T & D\end{pmatrix}.
	\end{equation}
%	\[
%	\nabla^2\clobs(\psc) = \begin{pmatrix}A & B\\ B^T & D\end{pmatrix}.
%	\] 
	Using the cyclicity of the trace, we calculate that
	\begin{align}
%	&
	\tr(\nabla^2 \clobs(\psc) G^{-1}) 
%	 \\
%	&= \tr(A\ICinv) + \tr(B\RC\ICinv) + \tr(B^T\ICinv\RC) + \tr(D(\IC + \RC\ICinv\RC))\\
%	&
	&= \tr((A+B\RC + \RC B^T + \IC D\IC + \RC D\RC)\ICinv)\\
	&= \tr(\nabla^2\clobs(\psc) _{\wm} \, \ICinv),
	\end{align}
	where $\nabla^2\clobs(\psc) _{\wm}$ was defined in \eqref{eq:hessian-symbol-wm} and has the form
	\begin{align}
	\nabla^2\clobs(\psc) _{\wm} 
	%&\coloneqq \begin{pmatrix}\Id & \wm^*\end{pmatrix}\nabla^2 \clobs(z_0)
	%\begin{pmatrix}\Id\\ \wm\end{pmatrix}\\
	&= A+B\wm+ \wm^*B^T + \wm^*D\wm,
	\end{align}
	which gives the claim.
\end{proof}

For our analysis of the observable error, we will use \Cref{prop:av_gauss} also for observables 
that are products of two functions. One of the factors will have a controlled \semiclassical 
expansion, when evaluated in the position center of the variational solution. 
%\textcolor{purple}{Note that $\mgPot(x)\cdot(x-\pos)$ is not subquadratic, but the factor $x-\pos$ behaves well, %as mentioned in the proof of \Cref{thm:L2}. For simplicity, we only state the next result for a subquadratic factor %$b^\scp$.}

%
\begin{corollary}[Gaussian averages] \label{cor:av_gauss} 
In the situation of \Cref{prop:av_gauss} applied to a sublinear classical observable 
$\clobs:\R^{2\dim}\to\R$, we additionally consider a smooth and subquadratic function $b^\scp:\R^\dim\to\R$, $x\mapsto b^\scp(x)$. Then, 
\begin{equation}
f_2(b^\scp,\wm)  = \tfrac14 \tr(\nabla^{2} b^\scp(\pos)\IC^{-1}).
\end{equation}

\begin{enumerateletters}
\item
If the function satisfies 
\[
b^\scp(\pos), \nabla b^\scp(\pos) = \mathcal O(\scp),
\]  
then 
\[
	\langle \clobs b^\scp\rangle_u = \clobs(\psc) \left( b^\scp(\pos) + 
	\scp f_2(b^\scp,\wm)\right) + 
	\mathcal O(\scp^2).
\]
\item
If the function satisfies 
\[
b^\scp(\pos) = \mathcal O(\scp^2),\quad
\nabla b^\scp(\pos),\nabla^2 b^\scp(\pos) = \mathcal O(\scp),
\]  
then 
\begin{align}
\langle \clobs b^\scp\rangle_\vsol &=
\clobs(\psc)\left( b^\scp(\pos) + \scp f_2(b^\scp,\wm) + \scp^2 f_4(b^\scp,\wm)\right)\\*[1ex]
&+\scp F_{1,1}(\clobs,b^\scp,\wm) + \scp^2 F_{1,3}(\clobs,b^\scp,\wm) + \mathcal O(\scp^3)
\end{align}
with 
\begin{align}
\quad
F_{1,n}(\clobs,b^\scp,\wm) =
\sum_{|\ell|=n+1} \sum_{\beta\le\ell, |\beta|=1} \frac{1}{(\ell-\beta)!}\ \partial^\beta\clobs(\psc)\ \partial^{\ell-\beta} b^\scp(\pos)\ \rho_\ell(\wm), \quad n=1,3.
\end{align}
%\begin{align}
%F_{1,1}(\clobs,b,\wm) &=
%\sum_{|\ell|=2} \sum_{\beta\le\ell, |\beta|=1} \partial^\beta\clobs(\psc)\ \partial^{\ell-\beta} b^\scp(\pos) \ \rho_\ell(\wm)\\
%F_{1,3}(\clobs,b,\wm) &=
%\sum_{|\ell|=4} \sum_{\beta\le\ell, |\beta|=1} \frac{1}{(\ell-\beta)!}\ \partial^\beta\clobs(\psc)\ \partial^{\ell-\beta} b^\scp(\pos)\ \rho_\ell(\wm).
%\end{align}
\end{enumerateletters}
\end{corollary}
\begin{proof}
For the trace formula, it is enough to observe that the matrices $B$ and $D$ in the block matrix \eqref{eq:hess-block-form} vanish, since $b^\scp$ only depends on $x$.
%\[
%\begin{pmatrix}\nabla b^\scp(\pos) & 0\\ 0 & 0\end{pmatrix}_\wm = \nabla b^\scp(\pos).  
%\]

 For proving the expansions of the averages, we crucially use the Leibniz formula for the $\ell$th derivative of the product, that is, 
\[
\partial^\ell(\clobs b^\scp)(\psc) = 
\sum_{\beta\le\ell} \binom{\ell}{\beta}\ \partial^\beta\clobs(\psc)\ \partial^{\ell-\beta} b^\scp(\pos)
\]
for any multi-index $\ell\in\N_0^{2\dim}$. 
\begin{enumerateletters}
\item
In the situation of statement (a), we only consider $|\ell| =2$ and obtain
\[
\scp \partial^\ell(\clobs b^\scp)(\psc) = \scp \clobs(\psc) \partial^\ell b^\scp(\pos) + \mathcal O(\scp^2).
\]
Then, \Cref{prop:av_gauss} implies
\begin{align}
\langle \clobs b^\scp\rangle_u &=  \clobs(\psc) b^\scp(\pos) + \scp f_2(\clobs b^\eps,\wm) 
+ \mathcal O(\scp^2)\\
&=\clobs(\psc) \left( b^\scp(\pos) + \scp f_2(b^\scp,\wm)\right) + 
	\mathcal O(\scp^2).
\end{align}
\item In the situation of statement (b), we aim for a higher order expansion and need to consider 
second and fourth derivatives.
In the same spirit as the proof of part (a), \Cref{prop:av_gauss} implies the claimed expansion of the average 
$\langle \clobs b^\scp\rangle_\vsol$.
%
%jetzt ausführliche Version:
%\item
%In the situation of statement (b), we aim for a higher order expansion and need to consider 
%second and fourth derivatives. For $|\ell| = 2$, we obtain that
%\begin{align}
%&\frac{\scp}{\ell !}\, \partial^\ell(\clobs b^\scp)(\psc) = \\
%&\frac{\scp}{\ell !}\, \clobs(\psc) \partial^\ell b^\scp(\pos) 
%+ \scp \sum_{\beta\le\ell, |\beta|=1}\ \partial^\beta\clobs(\psc)\ \partial^{\ell-\beta} b^\scp(\pos) + \mathcal O(\scp^3),
%\end{align} 
%where we have used that
%\[
%\frac{1}{\ell !} \binom{\ell}{\beta} = \frac{1}{\beta !(\ell-\beta) !} = 1. 
%\]
%Similiary, for $|\ell|=4$, 
%\begin{align}
%&\scp^2 \partial^\ell(\clobs b^\scp)(\psc) =\\ 
%&\scp^2 \clobs(\psc) \partial^\ell b^\scp(\pos) 
%+ \scp^2 \sum_{\beta\le\ell, |\beta|=1} \binom{\ell}{\beta}\ \partial^\beta\clobs(\psc)\ \partial^{\ell-\beta} b^\scp(\pos) + \mathcal O(\scp^3).
%\end{align}
%Then, \Cref{prop:av_gauss} implies the claimed expansion of the average 
%$\langle \clobs b^\scp\rangle_\vsol$. 
\end{enumerateletters}
\end{proof}

\begin{remark}\label{rem:av_gauss} 
The estimates of \Cref{cor:av_gauss} also apply to functions $b^\eps(x)$ of the form 
$b^\scp(x) = B^\scp(x)\cdot (x-q)$, where $B^\scp(x)\in\R^\dim$ is sublinear with uniform bounds in $\scp$. 
A derivative $\partial^\alpha b^\scp(x)$  
additively decomposes into a bounded function and the function $\partial^\alpha B^\scp(x)\cdot (x-q)$, 
which can be controlled by the arguments used in the proof of \Cref{thm:L2} (a).
\end{remark}
%%%%%%%%%%%%%%%%%%%%%%%%%%%%%%

\subsection{Proof of \Cref{thm:obs}}
We now have everything at hand to estimate the error of observables
and to conclude our final main result. In the following proof, we use 
\Cref{ass:onpotentials} on the potentials and the representation \eqref{eq:remPot} 
of the remainder potential $\remPot_\vsol$
only to the extent that the arguments literally also apply to the dynamics induced by general subquadratic hamiltonians. 
%The remainder potential $\remPot_\vsol$ is not subquadratic for the magnetic case, but the advection term 
%$\mgPot(x)\cdot (x-\pos)$ is nice, such that the proof works the same. 
 Thus, the proof improves known observable error estimates in full generality.
 
\begin{proof}[Proof of \Cref{thm:obs}]
By \Cref{lem:obs_bound_aux} we only have to bound the commutator in
the representation formula \eqref{eq:obs_bound_aux}. 
%		\begin{enumerateletters}%[leftmargin=0.65cm]
			%

\smallskip 
\noindent(a) We start by recalling, that in the proof of \Cref{thm:L2}, we have estimated 
			\begin{align}\label{eq:Wu_L2}
			\normLtwo{\remPot_{\vsol(s)}\vsol(s)} = \normLtwo{(\Id-P_{\vsol(s)})\Ham(s)\vsol(s)} 
			\le C\scp^{3/2}.
			\end{align}
%			\item  The remainder potential $\remPot_\vsol$ in \eqref{eq:remPot} is the sum of a quadratic polynomial developed in powers of $x-\pos$  plus the term $X_q$. By \cite[Lemma~3.15]{LasL20}, the coefficients of the 
%			quadratic polynomial are of the order $\scp$. In the proof of \Cref{thm:L2}, 
%			we have estimated $\normLtwo{X_q}$ as being of the order $\scp^{3/2}$. Therefore,
%			 \Cref{lem:average_pointeval} 			 
%			 to the remainder potential $\remPot_\vsol$ , which leads to
%			Let $\remPot$ be a smooth scalar potential, such that 
%			\begin{equation}%\label{eq:sec_deriv}
%			\l\|\nabla^2 \remPot(x,t)\r\| \leq c_\remPot \l(1+|x|^2\r)^n%\e^{M|x|^2}
%			\end{equation}
%			for some $n\in \N$ and all $x\in \R^\dim$, $t \in [0,\etime]$. Then a similar calculation as in \cite[Lemma~3.15]{LasL20} together with \Cref{lem:moments} gives the estimate	
%			\begin{equation}\label{eq:average_pointeval}
%			|\langle \remPot(\cdot,t)\rangle_\vsol - \remPot(\pos(t),t)| \leq c_{\remPot} C \scp %c^2_{1,0}c_{\remPot}c_{\pos}\frac{1}{\rho}
%			\end{equation}		
%			where the constants depend on the parameters and on $\remPot$. This leads to % und $\rho$ und mit Konstante $c_{1,0}$ aus \Cref{lem:moments} $\remPot, \pos$
%				\begin{equation}\label{eq:Wu_L2}
%				\normLtwo{\remPot_\vsol \vsol} \leq C\scp. %\quad\text{ and }\quad \normLtwo{\nabla \remPot_\vsol \vsol} \leq C\scp,
%				\end{equation}
%			where 
%			$\remPot_\vsol$ is the potential in \eqref{eq:remPot} and
%			the constant depends on parameters and on potentials. 
%
%

\smallskip
\noindent(b) We denote $\auxobs(t,s) = \egor{t}{s}$ and expand
		\begin{align}
		&\frac{1}{\ii\eps}\big\langle  \overline\remPot_{\vsol(s)}\hobs(t,s)-\hobs(t,s)\remPot_{\vsol(s)}\big\rangle_{\vsol(s)}\\
		=~&	\frac{1}{\ii\eps}\big\langle  \overline\remPot_{\vsol(s)}\op\l(\auxobs(t,s)\r)-\op\l(\auxobs(t,s)\r)\remPot_{\vsol(s)}\big\rangle_{\vsol(s)} +r_1(s,t).
			\end{align}
		%
%		\begin{equation}
%		\Bigl\langle \frac{1}{\ii\scp}\bigl[\remPot_{\vsol(s)}, \hobs(t,s)\bigr] \Bigr\rangle_{\vsol(s)} 
%		=
%		 \Bigl\langle \frac{1}{\ii\scp}\left[\remPot_{\vsol(s)}, \op\l(\clobs \circ \flow^{t,s}\r)\right] \Bigr\rangle_{\vsol(s)} +r_1(s,t),% \O(\scp^2),
%		\end{equation}
		%
%		where the remainder is given by
%		\begin{align}
%			r_1(s,t) &= \frac{1}{\ii\eps}\big\langle  \overline\remPot_{\vsol(s)}\big(\hobs(t,s) -\op\l(\auxobs(t,s)\r)\big)\big\rangle_{\vsol(s)}\\
%			&-\frac{1}{\ii\eps}\big\langle \big(\hobs(t,s) -\op\l(\auxobs(t,s)\r)\big)\remPot_{\vsol(s)}\big\rangle_{\vsol(s)}.
%			\end{align}		
		Using first Cauchy-Schwarz and then \eqref{eq:Wu_L2} together with \Cref{prop:timdep_egorov} and norm conservation, we bound the remainder by %obtain%%%%%%%%  Formel (6.6) in LasL20
		\begin{equation}%\label{eq_Rechnung_r_1}
		\begin{aligned}
		|r_1(s,t)| &\leq \frac{2}{\scp}	\normLtwo{\remPot_{\vsol(s)}\vsol(s)} \normLtwobig{\bigl(\hobs(t,s) -\op\l(\auxobs(t,s)\r)\bigr)\vsol(s)}%\\
		%
		%&\qquad + \frac{1}{\scp}\normLtwobig{\bigl(\hobs(t,s) -\op\l(\egor{t}{s}\r)\bigr)\remPot_{\vsol(s)} \vsol(s)}\\
		%
		%&\leq c\,\normLtwobig{\bigl(\hobs(t,s) -\op\l(\egor{t}{s}\r)\bigr)\vsol(s)} \\
		%&\qquad+ c\scp \normweiSob{\remPot_{\vsol(s)} \vsol(s)}{\vaiindex{m}}{j}\\
		%
		%&\leq c \scp^2, %\int_0^t\|\psi_{t-s}(s)\|_{L^2, m+d+1, l+j, \eps} \,\dd s.
%		&
		\leq c\, \scp^{5/2}.
		\end{aligned}
		\end{equation}
		%
%		\todoin{
%			\begin{itemize}
%		\item 	If \Cref{rem on weighted norms} enters, we do not need \eqref{eq:Wu_L2} in the last inequality? 
%			\item check polynomial degree of $W$
%			\item $L^2$-norm
%			\end{itemize}
%		}
%		where we used \eqref{eq:Wu_L2}		for the last inequality.
%
%

\smallskip
\noindent(c)		%\input{quantum-proof-mag.tex}
	    As in the proof of \Cref{prop:timdep_egorov}, we use the product rule of Weyl calculus and expand the commutator. For notational simplicity, we suppress the dependence on $t$ and $s$. We have by symmetry of the real part \wrt the $L^2$ scalar product and anti-symmetry of the Poisson bracket
	    % (reicht diese Erklärung, um die zweite Gleichung wegzulassen?)
%
	\begin{align}
	&\big\langle\overline\remPot_{\vsol}\,\op\l(\auxobs\r)-\op\l(\auxobs\r)\remOp_{\vsol}\big\rangle_{\vsol}\\
%	&=\big\langle\overline\remPot_{\vsol}\auxobs-\auxobs \remPot_{\vsol}\big\rangle_{\vsol} +
%	\frac{\scp}{2\ii} \big\langle\{\overline\remPot_{\vsol},\auxobs\}-\{\auxobs, \remPot_{\vsol}\}\big\rangle_{\vsol}\\
%	 &\quad+\frac{\scp^2}{8} \big\langle\nabla^2\overline\remPot_{\vsol}J\nabla^2\auxobs J
%	 -\nabla^2\auxobs J \nabla^2\remPot_{\vsol}J\big\rangle_{\vsol}
%	 +	\mathcal O(\scp^3)\\
	~=~&\frac{2}{\ii} \,\big\langle\im\remPot_{\vsol}\auxobs\big\rangle_{\vsol} 
	+ \frac{\scp}{\ii} \big\langle\{\re\remPot_{\vsol},\auxobs\}\big\rangle_{\vsol}
	+\frac{\scp^2}{4\ii} \big\langle\nabla^2\im\remPot_{\vsol}J\nabla^2\auxobs J\big\rangle_{\vsol}
	 +	\mathcal O(\scp^3),
	\end{align}
where the constant in $\O(\scp^3)$ depends on phase space derivatives of the remainder potential $\remPot_{\vsol(s)}$ and of $\auxobs(t,s)$ of the order $\ge 3$. 	Since 
	%$Y(t)$ and 
	$\auxobs(t,s)$
%\unklar{is sublinear and the classical Hamiltonian $\cham(s)$ subquadratic,} 
is sublinear and $\remPot_{\vsol(s)}$ consists of subquadratic summands and a non-subquadratic summand which can be handled by \Cref{rem:av_gauss},
%$\bm r(t,s)$ is bounded together with all its derivatives. 
	the Calder\'on--Vaillancourt Theorem applies for the remainder term.
	We will prove below that
	\begin{subequations}
	\begin{align}
	&\big\langle\op\l(\im\remPot_{\vsol}\auxobs\r)\big\rangle_{\vsol} = \O(\scp^3), \label{eq:obs-rem-imaginary-clobs-prod}\\
		&\big\langle\op\l(\{\re\remPot_{\vsol},\auxobs\}\r)\big\rangle_{\vsol} = \mathcal O(\scp^2), \label{eq:obs-rem-real-poiss}\\
	&\big\langle\op\l(\nabla^2\im\remPot_{\vsol}J\nabla^2\auxobs J\r)\big\rangle_{\vsol} = \mathcal O(\scp),\label{eq:obs-rem-imaginary-hess}
	\end{align}
	\end{subequations}
	which allows us to conclude that
	\begin{align}
	\frac{1}{\ii\eps}\big\langle \overline\remPot_{\vsol}\hobs-\hobs\remOp_{\vsol}\big\rangle_{\vsol}
	= \O(\scp^2).
	\end{align}
	In order to do so, we first aim at the application of \Cref{cor:av_gauss} statement (a) for $b^\scp = \partial_j\re\remPot_\vsol$ and statement (b) for $b^\eps = \im\remPot_\vsol$; 
	see also \Cref{rem:av_gauss} for the non-subquadratic terms in $\remPot_\vsol$.  

\smallskip
\noindent (d) From now on, we notationally focus on the magnetic Schr\"odinger case, but the analysis works the same for the general case. We denote the phase space center of the variational Gaussian $\vsol$ by $\psc=(\pscone,\psctwo)$. The width matrix of $\vsol$ is $\wm$ and has imaginary part $\IC$.
	We recall that the cubic remainder $R(\auxPot)$ in \eqref{eq:remPot} vanishes together with 
	its first and second derivatives when evaluated in $\pos$.

	We first apply the analysis to the Poisson bracket that involves the real part of the remainder
	potential.
	For any $j=1,\ldots,\dim$, we use \eqref{eq:remPot} and  \Cref{prop:av_gauss} and obtain $	\partial_j\re\remPot_\vsol(\pos) = \mathcal O(\scp)$. Furthermore, by \cite[Lemma~3.15]{LasL20} we have
	\begin{align}
	&\nabla \partial_j\re\remPot_\vsol(\pos) = \nabla\partial_j\re\auxPot(q)-\langle\nabla\partial_j\re\auxPot\rangle_u = \mathcal O(\scp),\\
	&\nabla^2\partial_j\re\remPot_\vsol(\pos) = \nabla^2\partial_j \re R(\auxPot)(q) = 
	\nabla^2\partial_j \re \auxPot(q).
	\end{align}
	Hence, the function $b^\scp = \partial_j\re\remPot_\vsol$ fulfills the assumptions of statement (a) in \Cref{cor:av_gauss}, and together with \Cref{prop:av_gauss}, we obtain
	\begin{align}
	&\langle\partial_j\re\remPot_\vsol \partial_{p_j}\auxobs\rangle_\vsol 
	= \mathcal O(\scp^2).
	\end{align}
	After summation over $j$, we have proven \eqref{eq:obs-rem-real-poiss}.

\smallskip
\noindent(e)
%\textcolor{purple}{The imaginary part of the auxiliary potential in \eqref{eq:remPot} is given by
%	\begin{align}
%	\im\auxPot(x) &= \im\advPot(x) = -\mgPot(x)\cdot\IC(x-q),
%	\end{align}
%	such that, again by \cite[Lemma~3.15]{LasL20},} 
Similarly, the first and second derivatives of $\im \remPot_\vsol$ satisfy
    \begin{equation}\label{eq:deriv-im-remPot}
   \nabla \im\remPot_\vsol(\pos),~ \nabla^2\im\remPot_\vsol(\pos) 	= \mathcal O(\scp).
    \end{equation}
%    \begin{subequations}\begin{align}\label{eq:nabla-im-remPot}
%    	&\nabla \im\remPot_\vsol(\pos) = \nabla\im\auxPot(q)-\langle\nabla\im\auxPot\rangle_u 
%    	= \mathcal O(\scp),\\\label{eq:hess-im-remPot}
%    	&\nabla^2\im\remPot_\vsol(\pos) = \nabla^2\im\auxPot(q)-\langle\nabla^2\im\auxPot\rangle_u 
%    	= \mathcal O(\scp).
%    	\end{align}
%    \end{subequations}
Moreover, \Cref{prop:av_gauss} implies for the point evaluation of the imaginary part of the remainder potential that
	\begin{align}
	\im\remPot_\vsol(\pos) 
%	&= 
%	-\langle\im\auxPot\rangle_\vsol + \frac\scp4 \tr(\langle\nabla^2\im\auxPot\rangle_\vsol\,\IC^{-1}) \\
	&= \frac\scp4\, \tr\!\left(\left(\langle\nabla^2\im\auxPot\rangle_\vsol - \nabla^2\im\auxPot(\pos)\right)\IC^{-1}\right) + \mathcal O(\scp^2)
	\\ %*[1ex]	
	&
	= \mathcal O(\scp^2).\label{eq:im-remPot}
	\end{align}
%	where we have used \eqref{eq:deriv-im-remPot}.} %\eqref{eq:hess-im-remPot}
	At this point, a simple application of \Cref{prop:av_gauss} and \eqref{eq:deriv-im-remPot} yields \eqref{eq:obs-rem-imaginary-hess}.
	%
	%nächstes Item umstrukturiert
	%	
%	\item We denote the phase space center of the variational Gaussian $\vsol$ by $\psc=(\pscone,\psctwo)$. The width matrix of $\vsol$ is $\wm$ and has imaginary part $\IC$. The imaginary part of the auxiliary potential satisfies
%	\begin{align}
%	\im\auxPot(x) &= \im\advPot(x) \\
%	&= -\mgPot(x)\cdot\IC(x-q).
%	\end{align}
%	We also recall that the cubic remainder $R(\auxPot)$ vanishes together with 
%	its first and second derivatives when evaluated in $\pos$. 
%    Then, \Cref{prop:av_gauss} implies for the imaginary part of the remainder potential that
%    \begin{align}
%	\im\remPot_\vsol(\pos) &= 
%	-\langle\im\auxPot\rangle_\vsol + \frac\scp4 \tr(\langle\nabla^2\im\auxPot\rangle_\vsol\,\IC^{-1}) \\
%	&= \frac\scp4\, \tr\!\left(\left(\langle\nabla^2\im\auxPot\rangle_\vsol - \nabla^2\im\auxPot(\pos)\right)\IC^{-1}\right) + \mathcal O(\scp^2)\\*[1ex]	
%	&= \mathcal O(\scp^2),\label{eq:im-remPot}
%	\end{align}
%	where we have used that
%	\[
%	\langle\nabla^2\im\auxPot\rangle_\vsol - \nabla^2\im\auxPot(\pos) = \mathcal O(\scp).
%	\]
%	Moreover, the first and second derivatives satisfy
%	\begin{subequations}\begin{align}\label{eq:nabla-im-remPot}
%	&\nabla \im\remPot_\vsol(\pos) = \nabla\im\auxPot(q)-\langle\nabla\im\auxPot\rangle_u 
%	= \mathcal O(\scp),\\\label{eq:hess-im-remPot}
%	&\nabla^2\im\remPot_\vsol(\pos) = \nabla^2\im\auxPot(q)-\langle\nabla^2\im\auxPot\rangle_u 
%	= \mathcal O(\scp).
%	\end{align}
%	\end{subequations}
%
%

\smallskip
\noindent(f) The expansions in \eqref{eq:deriv-im-remPot} and \eqref{eq:im-remPot}
	%\eqref{eq:nabla-im-remPot},\eqref{eq:hess-im-remPot} 
	show that $b^\eps = \im\remPot_\vsol$  satisfies the assumptions of statement (b) in~\Cref{cor:av_gauss}. 
%	We thus obtain
%	\begin{align}
%	\langle\im\remPot_\vsol \auxobs\rangle_\vsol &= \auxobs(\psc) \left( \im\remPot_\vsol(\pos) + 
%	\scp f_2(\im\remPot_\vsol,\wm) + \scp^2 f_4(\im\remPot_\vsol,\wm)\right)\\*[1ex]
%	& + \scp F_{1,1}(\auxobs,\im\remPot_\vsol,\wm) + 
%	\scp^2 F_{1,3}(\auxobs,\im\remPot_\vsol,\wm) + \mathcal O(\scp^3).
%	\end{align}
In order to prove \eqref{eq:obs-rem-imaginary-clobs-prod}, we analyse the expansion obtained from \Cref{cor:av_gauss} in two steps, aiming at
	\begin{subequations}
	\begin{align}\label{eq:a3}
	&\im\remPot_\vsol(\pos) + \scp f_2(\im\remPot_\vsol,\wm) + \scp^2 f_4(\im\remPot_\vsol,\wm) 
	= \mathcal O(\scp^3),\\
	\label{eq:na3}
	&\scp F_{1,1}(\auxobs,\im\remPot_\vsol,\wm) + \scp^2 F_{1,3}(\auxobs,\im\remPot_\vsol,\wm) = \mathcal O(\scp^3).
	\end{align}
	\end{subequations}

\smallskip
\noindent(g) We start with proving the first estimate \eqref{eq:a3}. For this, we need a slightly more accurate assessment of 
	$\im\remPot_\vsol(q)$ than developed previously. Using \eqref{eq:remPot}, \Cref{prop:av_gauss}, and 
	\eqref{eq:deriv-im-remPot}, we have
	\begin{align}
	\im\remPot_\vsol(\pos) 
%	&= -\langle \im\auxPot\rangle_\vsol 
%	+ \tfrac\scp4\tr(\langle\nabla^2\im\auxPot\rangle_\vsol\IC^{-1})\\*[1ex]
	=&-\scp f_2(\im\auxPot,\wm) - \scp^2 f_4(\im\auxPot,\wm)
% \\
%	&\quad
%	+ \scp \sum_{|\ell|=2} \frac{1}{\ell!} \langle\partial^\ell \im\auxPot\rangle_\vsol\rho_\ell(\wm)  
%	+ \mathcal O(\scp^3)\\
	%
%	&\quad
	+ \scp f_2(\langle \im\auxPot\rangle_\vsol,\wm) 
	+ \mathcal O(\scp^3)\\
	=&  - \scp^2 f_4(\im\auxPot,\wm) 
%	\\
%	&\quad
+ \scp^2 
f_2(f_2(\im\auxPot,\wm),\wm)
%	\sum_{|\ell|=2}\sum_{|k|=2} \frac{1}{k! \ell !}\partial^{k+\ell}\im\auxPot(\pos)\rho_k(\wm)\rho_\ell(\wm) 
	+ \mathcal O(\scp^3).
	\end{align}
	Similarly, we obtain for the second term in \eqref{eq:a3} that
	\begin{align}
	&\scp f_2(\im\remPot_\vsol,\wm) 
%	= \scp\sum_{|\ell|=2}\frac{1}{\ell!}\partial^\ell \im\remPot_\vsol(\pos) \rho_\ell(\wm)\\
%	&= \scp\sum_{|\ell|=2}\frac{1}{\ell!}\left(\partial^\ell \im\auxPot(\pos) 
%	- \langle\partial^\ell \im\auxPot\rangle_\vsol\right) \rho_\ell(\wm) \\
%	&
	= -\scp^2 
	f_2(f_2(\im\auxPot,\wm),\wm)
%	\sum_{|\ell|=2} \sum_{|k|=2} \frac{1}{\ell ! k!} \partial^{k+\ell}\im\auxPot(\pos) 
%	\rho_\ell(\wm)\rho_k(\wm) 
	+ \mathcal O(\scp^3).
	\end{align}
	Therefore, $\im\remPot_\vsol(\pos)$ cancels both the contributions from the second and the fourth derivatives, and we have proven \eqref{eq:a3}.

\smallskip
\noindent(h) We next target the terms on the left hand side of equation \eqref{eq:na3}, that is, 
	\begin{align}
	\scp F_{1,1}(\auxobs,\im\remPot_\vsol,\wm) 
	=-\scp^2 \sum_{|k|=2}
	 F_{1,1}(\auxobs,\partial^k \im\auxPot,\wm) \ \rho_k(\wm)
	+ \mathcal O(\scp^3)
	\end{align}
%\begin{align}
%&\scp F_{1,1}(\auxobs,\im\remPot_\vsol,\wm) 
%%=\\
%%&\scp \sum_{|\ell|=2} \sum_{\beta\le\ell, |\beta|=1} \partial^\beta\auxobs(\psc)\ \left(\partial^{\ell-\beta} 
%%\im\auxPot(\pos) -\langle\partial^{\ell-\beta} \im\auxPot\rangle\right) \rho_\ell(\wm)
%\\
%=&-\scp^2 \sum_{|k|=2}\sum_{|\ell|=2} \sum_{\beta\le\ell, |\beta|=1} \frac{1}{k!}\ \partial^\beta\auxobs(\psc)\  \partial^{k+\ell-\beta} \im\auxPot(\pos)\ \rho_k(\wm)\rho_\ell(\wm)
%\\
%%& \qquad 
%&
%+ \mathcal O(\scp^3)
%\end{align}
and
\begin{align}
 F_{1,3}(\auxobs,\im\remPot_\vsol,\wm) = F_{1,3}(\auxobs,\im\auxPot,\wm).
\end{align}
%\begin{align}
%&\scp^2 F_{1,3}(\auxobs,\im\remPot_\vsol,\wm) =\\
%&\scp^2 \sum_{|m|=4} \sum_{\beta\le m, |\beta|=1} \frac{1}{(m-\beta)!}\ \partial^\beta\auxobs(\psc)\ \partial^{m-\beta} \im\auxPot(\pos) \rho_m(\wm).
%\end{align}
In \Cref{lem:sum}, we provide the combinatorial argument that shows \eqref{eq:na3} as a 
consequence of Isserlis' theorem on the higher moments of multivariate normal distributions.  
Hence, we have proven $\langle \im\remPot_\vsol\auxobs\rangle_\vsol = \mathcal O(\scp^3)$, that is,  \eqref{eq:obs-rem-imaginary-clobs-prod}.
%\item 
%A simple application of \Cref{prop:av_gauss} and \eqref{eq:hess-im-remPot} yields \eqref{eq:obs-rem-imaginary-hess}.
%	\begin{align}
%	\big\langle\nabla^2\im\remPot_{\vsol}J\nabla^2\auxobs J\big\rangle_{\vsol} &= 
%	\nabla^2\im\remPot_{\vsol}(\pos)J\nabla^2\auxobs(\psc) J + \mathcal O(\scp)\\
%	&= \mathcal O(\scp).
%	\end{align}
%	Thus we have also proven \eqref{eq:obs-rem-imaginary-hess}.
%		\end{enumerateletters}
\end{proof}

\begin{remark}
The crucial estimates of the previous proof, namely \eqref{eq:obs-rem-imaginary-clobs-prod} and \eqref{eq:obs-rem-real-poiss} are one order worse for the \semiclassical Gaussian approximation, 
since it lacks the compensating averaging factors of the remainder potential. Therefore, for 
the \semiclassical Gaussians only $\mathcal O(\scp)$ observable accuracy can be expected. 
\end{remark}
%%%%%%%%%%%%

%% file: appendix_order4_sublinear.tex
\appendix

\section{Gaussian moments}\label{app:moments}
By an application of Isserlis' theorem, the fourth order Gaussian moments can 
be written as sums of products of second order moments. That is, for a 
$2\dim$-dimensional Gaussian random vector 
\[
(X_1,\ldots,X_{2\dim})\sim\mathcal N(0,\covm)
\] 
with mean zero $0\in\R^{2\dim}$ and covariance matrix $\covm\in\R^{2\dim\times 2\dim}$, the fourth order moments satisfy
\begin{align}
&\E(X_i^4) = 3 g_{ii}^2,\\
&\E(X_i^3 X_j) = 3 g_{ii}g_{ij},\\
&\E(X_i^2X_j^2) = g_{ii}g_{jj} + 2g_{ij}^2\\
&\E(X_i^2X_jX_k) = g_{ii}g_{jk} + 2g_{ij}g_{ik}\\
&\E(X_iX_jX_kX_\ell) = g_{ij}g_{k\ell} + g_{ik}g_{j\ell} + g_{i\ell}g_{jk}
\end{align}
with $i,j,k,\ell\in\{1,\ldots,2\dim\}$. We crucially use this for proving that the fourth order summations 
that appeared in the proof of \Cref{thm:obs} can be expressed in terms of second order 
summations.

\begin{lemma}[Resummation]\label{lem:sum}
For any family $(a_{\beta,m})_{\beta,m}$ of real numbers, indexed by $m\in\N_0^{2\dim}$ and 
$\beta\le m$ with $|\beta|=1$, we have
\begin{align}
&\sum_{|m|=4} \sum_{\beta\le m, |\beta|=1} \frac{1}{(m-\beta)!}\ a_{\beta,m}\,\rho_m(\wm) 
%\\&
= \sum_{|k|=2}\sum_{|\ell|=2}\sum_{\beta\le \ell,|\beta|=1} \frac{1}{k!} \ a_{\beta,k+\ell}\,
\rho_k(\wm)\rho_\ell(\wm).
\end{align}
\end{lemma}

\begin{proof}
We write a multi-index $m\in\N_0^{2\dim}$ of order $|m|=4$ as
\[
m = \langle j_1\rangle +  \langle j_2\rangle  +  \langle j_3\rangle  +  \langle j_4\rangle 
\]
with coordinates $j_1,\ldots,j_4\in\{1,\ldots,2\dim\}$, where 
the bracket $\langle j\rangle = e_j$ denotes the $j$th canonical basis vector of $\R^{2\dim}$. 
We distinguish five different cases for the order four multi-index $m$.\\[1ex]
(a) $m$ has one non-zero component, that is, $m=4\langle j\rangle$ with $j=1,\ldots,2\dim$. Then, 
\begin{align}
\frac{1}{(m-\beta)!}\ a_{\beta,m}\,\rho_m(\wm)  &= 
\frac{1}{3!}\, a_{\langle j\rangle,4\langle j\rangle} \, 3\rho_{2\langle j\rangle}(\wm)^2\\
&= \frac{1}{k!}\, a_{\beta,k+\ell}\,\rho_k(\wm)\rho_{\ell}(\wm)
\end{align}
with $k=2\langle j\rangle = \ell$ and $\beta=\langle j\rangle$. 
\\[1em]
(b) $m$ has two different non-zero components, that is, 
$m=3\langle j_1\rangle + \langle j_2\rangle$ with $j_1\neq j_2$. 
In this case, $m$ dominates two multi-indices $\beta$ of order one, and generates the terms
\begin{align}
&\left( \frac{1}{2!}\,a_{\langle j_1\rangle,m} + \frac{1}{3!}\,a_{\langle j_2\rangle,m}\right) 
3\rho_{2\langle j_1\rangle}(\wm)\rho_{\langle j_1\rangle + \langle j_2\rangle}(\wm)\\
&= \frac{1}{2!} \left( a_{\langle j_1\rangle,m} + a_{\langle j_2\rangle,m}\right) 
\rho_{2\langle j_1\rangle}(\wm)\rho_{\langle j_1\rangle + \langle j_2\rangle}(\wm)\\
&+\frac{1}{1!} \ a_{\langle j_1\rangle,m}\ 
\rho_{\langle j_1\rangle + \langle j_2\rangle}(\wm)\rho_{2\langle j_1\rangle}(\wm).
\end{align}
This amounts to the two $(k,\ell)$ pairs 
\begin{align}
&k=2\langle j_1\rangle, \ \ell=\langle j_1\rangle+\langle j_2\rangle,\ 
\beta\in\{\langle j_1\rangle,\langle j_2\rangle\},\\ 
&k=\langle j_1\rangle+\langle j_2\rangle,\ \ell=2\langle j_1\rangle, \ \beta=\langle j_1\rangle.
\end{align}
\\[1em]
In a similar manner, we show that in the cases \\[1em]
	(c) $m$ has two identical non-zero components, that is, $m=2\langle j_1\rangle + 2\langle j_2\rangle$ with $j_1\neq j_2$,
%	In this case, $m$ again dominates two multi-indices $\beta$ of order one, and we have to consider
%	\begin{align}
%	&\frac{1}{2!}\left(a_{\langle j_1\rangle,m} + a_{\langle j_2\rangle,m}\right) 
%	\left(\rho_{2\langle j_1\rangle}(\wm)\rho_{2\langle j_2\rangle}(\wm) + 
%	2\rho_{\langle j_1\rangle + \langle j_2\rangle}(\wm)^2\right)\\
%	&= \frac{1}{2!} a_{\langle j_2\rangle,m}\ \rho_{2\langle j_1\rangle}(\wm)\rho_{2\langle j_2\rangle}(\wm)
%	+\frac{1}{2!} a_{\langle j_1\rangle,m}\ \rho_{2\langle j_2\rangle}(\wm)\rho_{2\langle j_1\rangle}(\wm)\\
%	&+\frac{1}{1!} \left( a_{\langle j_1\rangle,m} + a_{\langle j_2\rangle,m}\right)
%	\rho_{\langle j_1\rangle + \langle j_2\rangle}(\wm)^2.
%	\end{align}
%	This amounts to the three pairs 
%	\begin{align}
%	&k=2\langle j_1\rangle, \ \ell=2\langle j_2\rangle,\ \beta=\langle j_2\rangle\\
%	&k=2\langle j_2\rangle, \ \ell=2\langle j_1\rangle,\ \beta=\langle j_1\rangle\\
%	&k=\langle j_1\rangle+\langle j_2\rangle = \ell, \ \beta\in\{\langle j_1\rangle,\langle j_2\rangle\},
%	\end{align}
%	that satisfy $k+\ell = m$.
	\\[1em]
(d) $m$ has three non-zero components, that is, 
$m=2\langle j_1\rangle + \langle j_2\rangle + \langle j_3\rangle$ with pairwise distinct $j_1,j_2,j_3$,
\\[1em]
(e) $m$ has four non-zero components, that is, $m=\langle j_1\rangle + \cdots + \langle j_4\rangle$ with  distinct $j_1,\ldots,j_4$,
\\[1em]	
we obtain the appropriate format of the resulting summands, that is, 
	\[
	\frac{1}{k!} \sum_{\beta\le \ell,|\beta|=1} a_{\beta,k+\ell}\, \rho_k(\wm)\,\rho_\ell(\wm) 
	\]
	with $k,\ell\in\N_0^{2d}$ such that $|k|=|\ell|=2$ and $k+\ell=m$. 
For concluding the proof, we have 
to verify that any possible $(k,\ell)$ pairing of order two multi-indices has appeared in one 
of the five cases (a)--(e). Let $i_1,\ldots,i_4\in\{1,\ldots,2d\}$ be such that 
\[
k=\langle i_1\rangle + \langle i_2\rangle,\quad
\ell = \langle i_3\rangle + \langle i_4\rangle.
\] 
The combinatorics of this situation falls into the following five cases:
\begin{enumerate}
\item[($\alpha$)] All four coordinates agree, 
that is, $i_1=\cdots=i_4=:j$. 
Then, $k+\ell = 4\langle j\rangle$, 
and we recognize the previous case (a).
\item[($\beta$)] Three of the four coordinates coincide with each other,
% that is, $i_1=i_2=i_3=:j_1$ and $i_4=:j_2$ is distinct from the other three ones, or the analogous three possible placements of an outlier. Then, $k+\ell=3\langle j_1\rangle + \langle j_2\rangle$, 
which is case (b). 
\item[($\gamma$)] The four coordinates form two different pairs, 
%
%that is, $i_1=i_2=:j_1$ and $i_3=i_4=:j_2$ or the 
%other two possible pairings. Then, $k+\ell = 2\langle j_1\rangle + 2\langle j_2\rangle$, 
%hence case (c).
and we are in case (c).
\item[($\delta$)] Two of the four coordinates agree, while the other two are different,
%. That is, $i_1=i_2=:j_1$ 
%and $i_3\neq i_4$ and different to $j_1$, or the other two possible pairings. Then, 
%$k+\ell=2\langle j_1\rangle + \langle i_3\rangle + \langle i_4\rangle$,
which is case (d).
\item[($\varepsilon$)] All four coordinates are distinct as in case (e).
\end{enumerate}
Hence, the combinatorics of the order four multi-indices and the one of pairs of order two multi-indices are the same, and we have indeed proven that the two different summation formats yield the same result as claimed. 
\end{proof}